\documentclass[a4, 12pt, reqno]{amsart}

\usepackage[latin1]{inputenc}
\usepackage{amsxtra}
\usepackage{amsmath}
\usepackage{amscd}
\usepackage{amssymb}
\usepackage{amsfonts}
\usepackage[all]{xy}
\usepackage{mathrsfs}
\usepackage{amsthm} 
\usepackage{color}
\usepackage{upgreek}
\usepackage{verbatim}
\usepackage{enumerate}
\usepackage{latexsym}
\usepackage{graphicx}
\usepackage{tikz}
\usepackage{pgfplots}
\usepackage{todonotes}
\usepackage{setspace}
\usepackage{hyperref}
\usepackage{stmaryrd}
\usepackage{nicefrac}
\usepackage{euscript}
\usetikzlibrary{decorations.pathmorphing}
\usetikzlibrary{positioning}

 \definecolor{darkred}{HTML}{993333}
\makeatletter

\newcommand{\arxiv}[1]{\href{http://arxiv.org/abs/#1}{\tt
    arXiv:\nolinkurl{#1}}}

\theoremstyle{plain}
\newtheorem{thm}{Theorem}[section]
\newtheorem{lemma}[thm]{Lemma}
\newtheorem{prop}[thm]{Proposition}

\newtheorem{cor}[thm]{Corollary}
\newtheorem{df-prop}[thm]{Definition-Proposition}
\newtheorem{itheorem}{Theorem}

\theoremstyle{definition}
\newtheorem{df}[thm]{Definition}

\theoremstyle{remark}
\newtheorem{rk}[thm]{Remark}
\newtheorem{ex}[thm]{Example}

\def\A{\mathrm{A}}
\def\B{\mathrm{B}}
\def\C{\mathrm{C}}
\def\D{\mathrm{D}}

\def\G{\mathrm{G}}
\def\H{\mathrm{H}}

\def\P{\mathrm{P}}

\def\S{\mathrm{S}}
\def\T{\mathrm{T}}

\def\V{\mathrm{V}}
\def\W{\mathrm{W}}
\def\X{\mathrm{X}}
\def\Y{\mathrm{Y}}

\def\bbC{\mathbb{C}}

\def\bbF{\mathbb{F}}

\def\bbH{\mathbb{H}}
\def\bbI{\mathbb{I}}

\def\bbN{\mathbb{N}}

\def\bbQ{\mathbb{Q}}
\def\bbR{\mathbb{R}}

\def\bbU{\mathbb{U}}
\def\bbV{\mathbb{V}}

\def\bbX{\mathbb{X}}

\def\bbZ{\mathbb{Z}}

\def\scrE{\mathscr{E}}
\def\scrF{\mathscr{F}}

\def\frakA{\mathfrak{A}}
\def\frakB{\mathfrak{B}}
\def\frakC{\mathfrak{C}}

\def\frakF{\mathfrak{F}}

\def\frakP{\mathfrak{P}}

\def\frakS{\mathfrak{S}}

\def\frakU{\mathfrak{U}}

\def\calA{\mathcal{A}}
\def\calB{\mathcal{B}}
\def\calC{\mathcal{C}}

\def\calE{\mathcal{E}}
\def\calF{\mathcal{F}}

\def\calH{\mathcal{H}}

\def\calK{\mathcal{K}}

\def\calM{\mathcal{M}}

\def\calO{\mathcal{O}}
\def\calP{\mathcal{P}}

\def\calU{\mathcal{U}}

\def\fraka{\mathfrak{a}}

\def\frakg{\mathfrak{g}}

\def\frakm{\mathfrak{m}}

\def\frakp{\mathfrak{p}}

\def\frakt{\mathfrak{t}}

\def\bfa{\mathbf{a}}
\def\bfb{\mathbf{b}}

\def\bfe{\mathbf{e}}
\def\bff{\mathbf{f}}

\def\bfi{\mathbf{i}}

\def\bfk{\mathbf{k}}

\def\bfo{\mathbf{o}}

\def\bfq{\mathbf{q}}

\def\bfB{\mathbf{B}}

\def\bfH{\mathbf{H}}

\def\bfL{\mathbf{L}}
\def\bfM{\mathbf{M}}

\def\bfS{\mathbf{S}}

\def\bfU{\mathbf{U}}
\def\bfV{\mathbf{V}}

\def\simto{\overset{\sim}\to}

\def\l{\langle}
\def\r{\rangle}

\def\al{\alpha}
\def\lam{\lambda}

\def\vep{\varepsilon}

\def\mod{\operatorname{-mod}\nolimits}

\def\proj{\operatorname{-proj}\nolimits}

\def\bgmod{\operatorname{-\textbf{gmod}-}\nolimits}

\def\Hom{\operatorname{Hom}\nolimits}
\def\End{\operatorname{End}\nolimits}

\def\ext{\operatorname{ext}\limits}

\def\id{\operatorname{id}\nolimits}

\def\pr{\operatorname{pr}\nolimits}

\def\cyc{\operatorname{cyc}\nolimits}

\def\si{\operatorname{if}\nolimits}
\def\sinon{\operatorname{otherwise}\nolimits}

\def\SO{\operatorname{SO}\nolimits}

\def\bbIj{\mathbb{I}^\jmath}
\def\bfUj{\mathbf{U}^\jmath}
\def\bdUj{\dot{\mathbf{U}}^\jmath}

\def\wU{\widetilde{\frakU}^\jmath}
\def\frakUj{\frakU^\jmath}

\newcommand{\excise}[1]{}
\newcommand{\thal}{{}^{\theta}\!\alpha}

\def\fraksl{\mathfrak{sl}}

\def\gr{\operatorname{gr}}

\def\iGr{\imath\mathrm{Gr}}

\newcommand{\hf}{\diamond}

\newcommand{\htodo}{\todo[inline,color=red!20]}
\newcommand{\ptodo}{\todo[inline,color=orange!20]}

\title[Categorification of quantum symmetric pairs I]{Categorification of quantum symmetric pairs I}
 \author[Bao, Shan, Wang and Webster]{Huanchen Bao, Peng Shan,
   Weiqiang Wang and Ben Webster}
\address[huanchen@math.umd.edu]{Department of Mathematics, University of Maryland, College Park, MD 20742, USA}

\address[peng.shan@math.u-psud.fr, pshan@math.tsinghua.edu.cn]{
Laboratoire de Math\'ematiques d'Orsay, Univ.~Paris-Sud, CNRS, 91405 Orsay, France
\newline
Yau Mathematical Sciences Center, Tsinghua University, 100084, Beijing, China
}
\address[ww9c@virginia.edu]{ Department of Mathematics, University of Virginia, Charlottesville, VA 22904, USA}
\address[ben.webster@uwaterloo.ca]{ Department of Mathematics, University of Virginia, Charlottesville, VA 22904 \newline
{\em Current address}: Dept. of Pure Mathematics, University of
Waterloo \&
Perimeter Institute for Theoretical Physics, Waterloo, ON, Canada}

\begin{document}

\numberwithin{equation}{section}

\begin{abstract}
We categorify a coideal subalgebra of the quantum group of $\mathfrak{sl}_{2r+1}$ by introducing a $2$-category analogous to the one defined by Khovanov-Lauda-Rouquier, and show that self-dual indecomposable
$1$-morphisms categorify the canonical basis of this
algebra. This allows us to define a categorical action of this
coideal algebra on the categories of modules over cohomology rings
of partial flag varieties
and on the category $\mathcal{O}$ of type $\B/\C$.
\end{abstract}

\maketitle


\setcounter{tocdepth}{2}
\tableofcontents

\section{Introduction}

\renewcommand{\theitheorem}{\Alph{itheorem}}

\newcommand{\frakso}{\mathfrak{so}}
\newcommand{\fraksp}{\mathfrak{sp}}
\newcommand{\Spec}{\operatorname{Spec}}

\subsection{Motivation}\,
One important development in representation theory in the last decade is the theory of categorical actions of Lie algebras, in the sense of Chuang and Rouquier \cite{CR}.
The basic object of this theory is a 2-category $\dot\frakU(\frakg)$ associated with a semisimple or Kac-Moody
Lie algebra $\frakg$,  introduced independently in two different presentations by Khovanov-Lauda \cite{KLIII} and Rouquier \cite{R} 
(these were shown to be equivalent by Brundan \cite{B}).  
The notion of a  categorical action of $\frakg$ is made precise by the structure of $2$-morphisms in $\dot\frakU(\frakg)$, 
and it provides an algebraic way of understanding deep structures of
the quantum group $\bfU(\frakg)$ such as canonical bases \cite{Lu92, Lu94, Ka93}. 


\smallskip

One case of particular interest is when  $\frakg=\mathfrak{sl}_r$, where Schur duality plays a basic role.  The Grothendieck group of $\dot\frakU(\mathfrak{sl}_r)$ coincides with the modified quantum group $\dot\bfU=\dot\bfU(\mathfrak{sl}_r)$, 
and the classes of indecomposable $1$-morphisms are identified with Lusztig's canonical bases \cite{L, WebCB}. 
In the quantum setting, Schur duality relates the quantum group $\bfU(\mathfrak{sl}_r)$ to Hecke algebras of type $\A$  \cite{Jim}.
Beilinson-Lusztig-MacPherson \cite{BLM90} constructed the modified quantum group $\dot\bfU(\mathfrak{sl}_r)$
as a limit of a family of finite-dimensional algebras known as Schur algebras, and they realized the Schur algebras 
using functions on the points of $r$-step partial flag varieties over a finite field.
The BLM construction was then adapted in  \cite{GL92} to realize quantum Schur duality. 
This geometric realization of  $\dot\bfU(\mathfrak{sl}_r)$ has been lifted to 
a categorical action of $\dot\frakU(\mathfrak{sl}_r)$  in  \cite{KLIII}.
This construction also leads to a categorified Schur algebra, which is
intimately related to singular Soergel bimodules \cite{WilSSB, MSV, Webcomparison}.  

\smallskip

Schur duality has been generalized in \cite{BW13} in connection
to the Kazhdan-Lusztig theory for the BGG category $\calO$ of type
$\B/\C$. While other generalizations
of Schur duality have changed $\mathfrak{sl}_r$ to a different
classical Lie algebra \cite{BMW1,BMW2},
this construction replaces the Hecke algebra from type $\A$ with one of
type $\B/\C$.  Like in type A, this construction admits a BLM-type realization \cite{BKLW}, where 
$\bfU({\fraksl}_r)$ is replaced by the coideal algebra $\bfU^\jmath$ associated with the Cartan involution composed with the diagram involution of the $\A_{2r}$ Dynkin diagram. 
Note that $(\bfU(\mathfrak{sl}_{2r+1}), \bfU^\jmath)$ forms a quantum symmetric pair in the sense of \cite{Le03}. 
The commutant of the Hecke algebra  of type $\B/\C$ in this setting  is called $\jmath$Schur algebra.
They are naturally quotients of $\bfU^\jmath$, and admit a geometric realization
in terms of partial flag varieties of type $\B/\C$ over finite fields. The idempotented coideal algebra can be realized as a limit of the family of $\jmath$Schur algebras. 
Based on this construction, it was shown in \cite{LW15} that
$\dot{\bfU}^\jmath$ admits a canonical basis with desirable positivity
properties, analogous to Lusztig's canonical basis for $\dot{\bfU}$ as
defined in \cite[Chapter~ 25]{Lu94}.  The $\calA=\bbZ[q,q^{-1}]$-span of this basis gives an integral form ${}_\calA\dot{\bfU}^\jmath$.

By analogy with the ``weak $\mathfrak{sl}_2$-categorifications'' discussed in \cite[5.1]{CR}, 
we can think of the construction of \cite{BW13} as a weak categorical action of $\bfU^\jmath$.  
The construction of \cite{BKLW} shows the existence of a similar weak action of $\bfU^\jmath$ on constructible sheaves on partial flag varieties of type $\B/\C$.  
Following the analogy with type $\A$, it is natural to expect that we can define a strong  categorical action of $\bfU^\jmath$ 
by a $2$-category $\dot\frakU^\jmath$ analogous to $\dot\frakU$. 

The existence of this 2-category together with its basic 2-representations is the principal result of this paper.

\smallskip
\subsection{Main results}\,

To explain the definition of our $2$-category $\dot\frakU^\jmath$, let us first give a description of $\dot{\bfU}^\jmath$.  
 The idempotented algebra $\dot{\bfU}^\jmath$ can be viewed as a category with objects indexed by a weight lattice $\X_\jmath$  
 and morphisms generated by
\[\calE_i:\lambda\to\lambda+\alpha_i\quad\text{and}\quad \calF_i:\lambda\to\lambda-\alpha_i,
 \qquad \text{ for $i=\hf,\hf+1, \ldots, r-\hf$ with $\hf=\frac{1}{2}$.}
\] See Section \ref{coideal-subalgebra} for a precise definition.
For $i>\hf$, the generators $\calE_i$, $\calF_i$ satisfy the same relations as standard Chevalley generators in $\dot\bfU$.  
However, the relations are different when $i=\hf$: there is no relation between length 2 monomials in $\calE_\hf$ and $\calF_\hf$.   
Instead they satisfy inhomogeneous $\jmath$-Serre relations \eqref{eq:jserre1} and \eqref{eq:jserre2}, whose summands are no longer canonical basis elements.
In particular, $\bfU^\jmath$ does not have  a triangular decomposition like that of $\bfU$.

To define $\dot\frakU^\jmath$, we follow the approach of \cite{L, KLIII}: 
we already know of a weak $\dot\frakU^\jmath$-action on the modules over cohomology rings of partial flag varieties, 
so we can extract the relations in $\dot\frakU^\jmath$ based on computations in this category.  
While the $2$-morphisms acting on $\calE_i$ and $\calF_i$ for $i\neq\hf$ satisfy the same relations as in $\dot\frakU$, new relations appear for $i=\hf$. 

One can think of these new relations as combining the KLR relations for the elements $\{E_{\hf},E_{-\hf}\}$ in $\bfU^+(\mathfrak{sl}_{3})$ 
and for $\{E_{\hf},F_{\hf}\}$ in $\bfU(\mathfrak{sl}_2)$.  
In particular, the usual bicross relations which categorify the commutator relation for $\calE_i$ and $\calF_i$ (\ref{bicross down})--(\ref{bicross up}) combine with (\ref{now}) 
to give  \eqref{j-bicross down}--\eqref{j-bicross up}.
Similarly, there is a new relation involving a triple point \eqref{Pi=1} which combines the relation of \cite[Proposition~5.8]{L} with (\ref{qhalast}); 
this relation is used in proving the categorified $\jmath$-Serre relations, as we will explain in Section~ \ref{sec:Grothendieck} and Appendix~\ref{sec:jserre}. 
The bubble relations and bubble slide formulas for $i =\hf$ here are also somewhat different from \cite{L}.  

The complexity of this last triple point relation makes it difficult to verify the relations of $\dot\frakU^\jmath$ directly (much like the presentation of $\dot\frakU$ in \cite{KLIII}). 
However, we show that if $\calE_i$ and $\calF_i$ satisfy the relations of the coideal subalgebra at the decategorified level, then \eqref{Pi=1} 
is forced by the other relations in $\dot\frakU^\jmath$; see Proposition \ref{df-prop}.
The reader can think of Proposition~ \ref{df-prop} as a weak version of ``control from $K_0$'' theorems, such as \cite[5.27]{R}. 
%

\smallskip
Now, let us describe the main results of this paper.   The first is
that $\dot\frakU^\jmath$ is a categorification of $\dot{\bfU}^\jmath$, in
the same sense as the relationship of $\dot\frakU$ and $\dot{\bfU}$.  


\begin{itheorem}[Theorem~\ref{thm:maindecat}]
\label{th:2-cat}
  There is an algebra isomorphism between the Grothendieck group of $\dot\frakU^\jmath$
  and ${}_\calA\dot{\bfU}^\jmath$, with indecomposable self-dual 1-morphisms matching the canonical basis of ${}_\calA\dot{\bfU}^\jmath$.  
  \end{itheorem}

To prove this theorem, and as evidence for the usefulness of this categorification, we ask for generalizations of the
categorical actions of $\frakU$ discussed earlier.
Let $\G$ be either the group $\SO(2m+1)$ or $Sp(2m)$, and let $\frakg$ be its Lie algebra.
One can define a $2$-category $\frakF_{r,m}$ analogous to the ``flag
category'' of Khovanov and Lauda \cite{KLIII} using cohomology rings of partial flag varieties of $\G$, 
see Section \ref{sec:type-b/c} for the precise definition.  The $2$-category $\dot\frakU^\jmath$ admits 
the following categorical actions. 

\begin{itheorem} [Theorem~\ref{thm:2repflagj}, Theorem~\ref{thm:full1}, Proposition~\ref{prop:Fullgeneral}, Theorem~\ref{thm:actionO}]
\label{th:main-actions}\,\hfill
\begin{enumerate}
\renewcommand{\theenumi}{\alph{enumi}}
\item There is a functor
  $\Gamma: \dot\frakU^\jmath\to \frakF_{r,m}$ for each $m$.
  This functor is essentially surjective on 1-morphisms (and becomes
  full on 2-morphisms after a small modification).

\item The category $\dot\frakU^\jmath$ acts on the category
  $\calO$ of $\frakg$ such that an object $\lambda\in\X_\jmath$ is
  sent to a block of $\calO$, and $\calE_i$, $\calF_i$ act by
  translation functors on blocks.
\end{enumerate}
\end{itheorem}

The actions in Parts (a) and (b) provide the desired enhancements of the weak categorical actions in \cite{BKLW} and in \cite{BW13}, respectively.




These actions are related via the classic links between projective
functors, Harish-Chandra bimodules and singular Soergel bimodules
given by \cite{Soe90,Soe92, Str04}  and \cite{WilSSB}.  The action of
Part (a) allows us to show that $\dot\frakU^\jmath$ acts on any category of representations of Lie algebra of type B/C which are of finite length 
and are closed under tensor product with finite dimensional modules
(see Theorem~\ref{thm:actionO}). Category $\calO$ is an example of
such a category, as are many of the variations on it.

\smallskip
\subsection{Structure of the paper}

In Section \ref{sec:prelim} we will recall the basic structure results on the algebras $\bfU$ and $\bfU^\jmath$, 
along with their idempotented forms $\dot{\bfU}$ and $\dot{\bfU}^\jmath$.
In Section \ref{sec:coideal-2-cat}, we give the definition of the 2-category $\dot\frakU^\jmath$, presented in diagrammatic terms, and study some basic properties of it. 
In Section \ref{sec:Grothendieck}, we precisely formulate the categorification of the $\jmath$-Serre relations, 
which allows us to define a surjective algebra homomorphism $\dot\bfU^\jmath\to K_0(\dot\frakU^\jmath)$. 
The categorification of $\jmath$-Serre relations involves some lengthy diagrammatic computation, which will be given in Appendix \ref{sec:jserre}. 
We construct the $2$-functor $\Gamma: \dot\frakU^\jmath\to \frakF_{r,m}$  in Section~\ref{sec:Schur},
and prove it is locally full in Section~\ref{sec:Schur2}. This allows us to deduce the isomorphism
$\dot\bfU^\jmath\simeq K_0(\dot\frakU^\jmath)$ and the matching of canonical bases.
Finally, we describe the action of $\dot\frakU^\jmath$ on category $\calO$ in Section~\ref{sec:relat-proj-funct}.

\smallskip
\subsection{Future developments}

The construction of this paper seems likely to be only one of a family
of categorifications corresponding to coideal subalgebras, and
relevant to the representation theory of classical groups and Lie
algebras of types
$B/C/D$ in many contexts.  Whereas categorical actions of any
Kac-Moody algebra are built from categorical
$\mathfrak{sl}_2$-actions, actions of $\dot\frakUj$ include a new basic
building block, which may be useful in other contexts.
For reasons of space and complexity, we have left the consideration of several natural questions that arise in this framework to future work. 

In a sequel to this work, we will consider analogues of cyclotomic KLR algebras for $\dot\frakU^\jmath$; 
these are algebras which naturally categorify  the restrictions to $\bfU^\jmath$ of finite-dimensional simple $\bfU$-modules 
and their $\jmath$canonical bases defined in \cite{BW13}. 
We will also study the relationship between $\dot\frakU^\jmath$ and $\dot\frakU$ which can be regarded 
as a categorical quantum symmetric pair.

The connection between the coideal algebras and category $\calO$ of type $\D$ has been independently observed
in \cite{ES13}, where they also studied relations between morphisms in $\calO$ with Nazarov-Wenzl algebras and affine Brauer algebras. 
We expect the categorical action of our $2$-category $\dot\frakU^\jmath$ will bring a new perspective on these algebras.
Note that there should also be a type $\D$ analogue of our main results;
see \cite{Ba16} for Kazhdan-Lusztig theory of type $\D$ and \cite{FL15} for geometric Schur duality of type $\D$.

The algebra ${\bfU}^\jmath$ is merely a special example of the coideal algebras arising from quantum symmetric pairs \cite{Le03, Ko14}. 
Particularly important special cases include the  coideal subalgebra  ${\bfU}^\imath\subset \bfU(\mathfrak{sl}_{2r})$ 
associated to its diagram automorphism, 
and the affine analogues of these algebras, attached to diagram automorphisms for $\widehat{\fraksl}_n$ \cite{FLLLW}.   
The former  also appears in geometric Schur duality and the study of category $\calO$ in type $\B/\C$ \cite{BW13, BKLW, LW15} 
while the latter is expected to play a fundamental role in the study
of modular representations of type $\B$/$\C$.   These other symmetric
pairs will require more work to categorify.
Since these algebras admit natural geometric realizations and canonical bases, 
the techniques of this paper can likely be applied to them as well.

\medskip
\subsection*{Acknowledgements}

H.B. is partially supported by an AMS-Simons travel grant. 
P.S. is partially supported by the ANR grant number ANR-12-JS01-0003, ANR-13-BS01-0001-01. 
W.W. is partially supported by the NSF grant DMS-1405131. 
B.W. is partially supported by the NSF grant DMS-1151473 and the Alfred P. Sloan Foundation. 
We thank the Institute of Mathematical Science at the University of Virginia, Institute of Mathematics, Academia Sinica (Taipei), 
Yau Mathematical Sciences Center at Tsinghua University, and Max
Planck Institute for Mathematics for support which facilitated  this
collaboration.  We thank Zajj Daugherty, Aaron Lauda, Yiqiang Li, Marco Mackaay,
Arun Ram and Catharina Stroppel for useful conversations.  We thank the referee for a careful reading and corrections on the proof of Lemma~A.1.

\section{Preliminaries on coideal subalgebras}
\label{sec:prelim}
\subsection{Notations and conventions}
\label{sec:notation}

Let $q$ be a formal variable and let $\calA =\bbZ[q,q^{-1}]$. For $a \in \bbZ$, let
\[[a] = \frac{q^{a}- q^{-a}}{q-q^{-1}}\qquad \text{and}\qquad 
[a]! = [1][2] \cdots [a] \quad \text{ for $a\geqslant 0$.}\] 
The $a$-th {\em divided power} of an element $E$ in a $\bbQ(q)$-algebra is the element $E^{(a)}:=E^a/[a]!$.

Let $\calC$ be an additive category. For any object $x$ in $\calC$, we
denote by $1_x$ the identity endomorphism of $x$. A {\em grading} on
$\calC$ is an auto-equivalence $\{1\}\colon \calC\to \calC$. Let $\{\ell\}$ denote the
$\ell$-fold composition of the auto-equivalence $\{1\}$. Given an
object $x$ in $\calC$ and $f=\sum f_sq^s\in\calA$, we write
$\oplus_fx$ or $x^{\oplus f}$ for the direct sum over $s$ of $f_s$ copies
of $x\{s\}$. The {\em split Grothendieck group} $K_0(\calC)$ is the
abelian group generated by symbols $[x]$ for $x$ an object, with the
relation $[x\oplus y]=[x]+[y]$.  A grading induces an $\calA$-module
on $K_0(\calC)$ such that $q[x]=[x\{1\}]$.
Given two objects $x$, $y$, let $\Hom^s(x,y)$ denote the space $\Hom(x\{s\},y)$ of morphisms of degree $s$. We abbreviate $\Hom^\bullet(x,y)=\bigoplus_{s\in\bbZ}\Hom^s(x,y)$. 
The idempotent completion (or Karoubi envelope) $\dot\calC$ is the
category whose objects are pairs $(M,e)$ of  $M\in \operatorname{Ob}(\calC)$ and
idempotent endomorphisms $e\colon M \to M$, with $(M,e)$ serving as
the image of the idempotent $e$.  The category $\dot\calC$ is universal
among idempotent complete categories with a functor from $\calC$.

Given a ring $\bfk$, a {\em graded $\bfk$-linear $2$-category} $\frakA$ is a category enriched over graded additive $\bfk$-linear
categories, that is,  a $2$-category such that the $\calH{om}$ categories are graded additive $\bfk$-linear categories 
and the composition maps $\calH{om}_\frakA(x,y)\times\calH{om}_\frakA(y,z)\to\calH{om}_\frakA(x,z)$ form a graded additive $\bfk$-linear functor. 
We will abbreviate $\frakA(x,y)=\calH{om}_\frakA(x,y)$. 
Objects and morphisms in $\frakA(x,y)$ will be respectively called $1$-morphisms and $2$-morphisms in $\frakA$. 
A $2$-category is called {\em strict} if $1$-morphisms compose strictly associatively, and a $2$-functor is called {\em strict} 
when it preserves compositions of $1$-morphisms on the nose. We will always assume that $\frakA(x,x)$ has a unit object, denoted by $1_x$. The Grothendieck group of $\frakA$ is the category $K_0(\frakA)$ with same objects as $\frakA$, and with morphisms given by the Grothendieck group of the corresponding $\calH{om}$-category. If $\frakA$ is graded, then $K_0(\frakA)$ is an $\calA$-linear category.


Recall that  a functor is {\em essentially surjective} if any object in the target is isomorphic to an object in the image. 
It is {\em full/faithful} if it is surjective/injective on Hom sets.  A fully faithful, essentially surjective functor is an equivalence.
A $2$-functor $\Phi: \frakA\to\frakB$ is \emph{locally full (resp. faithful, essentially surjective)} 
if the induced functors $\frakA(x,y)\to \frakB(\Phi(x),\Phi(y))$ is full (resp. faithful, essentially surjective)
for all objects $x,y\in\frakA$.

\medskip


%
%
\subsection{Reminders on $\dot{\mathbf{U}}$}
\label{subsec:theta}

Fix a positive integer $r$ and define
\begin{equation}
    \label{eq:I}
\bbI=\mathbb{I}_{2r}  = \big\{i \in \bbZ+\frac 12 \mid -r < i < r \big\}.
\end{equation}
Due to an extensive use of $\frac 12$ throughout the paper, we will often write 
$$
\hf=\frac 12.
$$
Consider the root datum of type $A_{2r}$ with Cartan matrix indexed by $\bbI$,
weight lattice $\X$, simple roots $\{\alpha_i\}_{i\in \bbI}\subset \X$, simple coroots $\{\alpha^\vee_i\}_{i\in \bbI}$, and the coroot lattice $\Y=\bigoplus_{i\in\bbI}\bbZ\alpha^\vee_i$. There is a perfect pairing 
\begin{equation}\label{eq:pairing}
\langle \cdot, \cdot\rangle: \Y\times\X\longrightarrow  \bbZ.
\end{equation} 
The entries of the Cartan matrix are given by
$\l\alpha^\vee_i,\alpha_j\r$ for $i,j\in\bbI$. 

Consider the lattice $\bigoplus_{a=-r}^r\bbZ \varepsilon_a$ with  the standard pairing $\l\vep_a,\vep_b\r=\delta_{a,b}$.
We will identify 
\begin{align*}
\X &=\bigoplus_{a=-r}^r\bbZ \varepsilon_a \Big / \bbZ (\sum_{a=-r}^r \varepsilon_a). 
\end{align*}
Then $\alpha_i=\varepsilon_{i-{\hf}}-\varepsilon_{i+\hf} \; (\text{mod } \bbZ \sum_{a=-r}^r \varepsilon_a)$ lies in $\X$. 

The quantum group $\bfU=\bfU_q(\fraksl_{2r+1})$ is the $\bbQ(q)$-algebra generated by $E_i$, $F_i$, $K_i^{\pm 1}$, $i\in\bbI$ subject to the relations 
\begin{align*}
K_iK_j=K_jK _i,  
& \qquad
K_iE_jK_i^{-1}=q^{\l\alpha_i^\vee,\alpha_j\r}E_j, 
\\
[E_i,F_j]=\delta_{ij}\frac{K_i-K_i^{-1}}{q-q^{-1}}, 
& \qquad
K_iF_jK_i^{-1}=q^{-\l\alpha_i^\vee,\alpha_j\r}F_j, 
\end{align*}
and the quantum Serre relations (see \eqref{Serre-dot} below). We may write $\bfU=\bfU_{2r+1}$ if necessary.
It is a Hopf algebra, with a coproduct   
$\Delta:\bfU\to\bfU\otimes\bfU$ such that 
$$
\Delta(K_i)=K_i\otimes K_i, 
\quad
\Delta(E_i)=1\otimes E_i+E_i\otimes K_i^{-1}, 
\quad
\Delta(F_i)=F_i\otimes 1+K_i\otimes F_i, \forall i.
$$
There is 
a unique involution (called \emph{bar involution})
$\psi:\bfU\to\bfU$ as $\bbQ$-algebra which sends $q$ to $q^{-1}$, $K_i$ to $K_i^{-1}$ and fixes $E_i$, $F_i$, for all $i$.

We will be interested in an idempotented (or modified) form of $\bfU$ introduced by Lusztig \cite{Lu94}. 
Consider the $\bbQ(q)$-linear category $\dot\bfU$ with the object set $\X$ and morphisms generated by 
$E_i:\lambda\to\lambda+\alpha_i$, $F_i: \lambda\to\lambda-\alpha_i$, subject to the relations 
\begin{align}
[E_i,F_j]1_\lambda&=\delta_{ij}[\l\alpha_i^\vee,\lambda\r] 1_\lambda,
\notag \\
\sum_{\scriptscriptstyle{a+b=1-\l\alpha_i^\vee,\alpha_j\r}}(-1)^aE_i^{(a)}E_jE_i^{(b)}
&=\sum_{\scriptscriptstyle{a+b=1-\l\alpha_i^\vee,\alpha_j\r}}(-1)^aF_i^{(a)}F_jF_i^{(b)}=0,\quad\forall  i\neq j.
\label{Serre-dot}
\end{align}
Lusztig's idempotented algebra (also denoted by $\dot{\bfU}$) can be identified with the direct sum of all Hom-spaces in this category.
Let ${}_\calA\dot\bfU$ be the $\calA$-linear subcategory of $\dot\bfU$ with the same objects and with 
morphisms generated over $\calA$ by $E_i^{(a)}$, $F_i^{(a)}$ for $i\in\bbI$, $a\geqslant 0$.
This corresponds to Lusztig's integral $\calA$-form of $\dot\bfU$.

\subsection{The coideal category $\dot{\mathbf{U}}^\jmath$}
\label{coideal-subalgebra}

Set
\begin{equation}
  \label{eq:Ij}
\mathbb{I}^{\jmath} =\mathbb{I}^{\jmath}_r =\mathbb{I}\cap\bbR_{>0}= \Big\{\hf, \hf+1, \ldots, \hf+r-1 \Big\}.
\end{equation}
We shall consider a $\bbQ(q)$-algebra $\bfU^\jmath$ which admits a presentation 
with generators $\calE_i, \calF_i, \calK_i^{\pm 1}$ for $i\in\bbI^\jmath$, 
and  a set of relations; see \cite[Section~6.1]{BW13}. There is an embedding of algebras 
$\jmath: \bfU^\jmath \rightarrow \bfU$ such that 
\begin{align*}
\calE_i \mapsto E_i+K_i^{-1}F_{-i},
\quad
\calF_i \mapsto F_iK_{-i}^{-1}+E_{-i},
\quad
\calK_i \mapsto K_iK_{-i}^{-1}, \quad  \forall i\in\bbI^\jmath.
\end{align*} 
We will consider the algebra $\bfU^\jmath$ as a subalgebra of $\bfU$ under the embedding $\jmath$.
We may write $\bfU^\jmath=\bfU^\jmath_r$ if necessary.

The subalgebra $\bfU^\jmath$ has a further compatibility with the
coproduct $\Delta$: it is a \emph{coideal} subalgebra of $\bfU$. That
is, the coproduct $\Delta$ on $\bfU$ restricts to a $\bbQ(q)$-algebra
homomorphism $\Delta: \bfU^\jmath\to\bfU^\jmath\otimes\bfU$. 
The specialization of $\bfU^\jmath$ at $q=1$ is just the enveloping
algebra of $\fraksl_{2r+1}^{\vartheta}$, the Lie subalgebra fixed by the involution $\vartheta\colon
\fraksl_{2r+1}\to \fraksl_{2r+1}$ given by rotating
$(2r+1) \times (2r+1)$-matrices by $\pi$ radians. 
Hence $(\bfU,\bfU^\jmath)$ is an example of a quantum symmetric pair \cite{Le03, Ko14}.


Let $\theta$ be the involution of the  lattice $\X$ induced by letting
$\theta(\varepsilon_{a}) = - \varepsilon_{-a}$ for $-r\leqslant a \leqslant r$.
We may write $\lambda^{\theta} = \theta(\lambda)$, for $\lambda \in \X$. 
Denote by $\X^{\theta} $ the sublattice of $\theta$-fixed points in $\X$. 
Note that $ \alpha^{\theta}_i = \alpha_{-i}$ for all $i \in \mathbb{I}$, hence $\theta$ induces an automorphism of the root system. 
Set $\thal_i^\vee=\alpha_i^\vee - \alpha_{-i}^\vee$ for $i\in\bbI^\jmath$. Let
\begin{equation}
\label{Xj}
\X_\jmath =\X/\X^\theta, \qquad \Y^\jmath = \bigoplus_{i \in {\bbI^\jmath}} \bbZ \,\thal_i^\vee.
\end{equation}
The lattice $\X_\jmath$ can be regarded as a weight lattice for $\bfU^\jmath$. 
The pairing \eqref{eq:pairing} induces a nondegenerate pairing
$$\langle \cdot, \cdot\rangle : \Y^\jmath \times {\X_{\jmath}}  \longrightarrow \bbZ.$$
For $\lambda\in \X_\jmath$, we write 
$$
\lambda_i=\l\thal_i^\vee,\lambda\r, \qquad \text{for } i \in \bbI^\jmath.
$$
For $\lambda\in \X$ we will denote its image in $\X_\jmath$ again by $\lambda$ if there is no confusion.
In particular, we will often regard $\alpha_i \in \X$ for $i\in\bbI^\jmath$ as an element in $\X_\jmath$. Note the unusual pairing
\[
\l\thal_\hf^\vee,\alpha_\hf \r =3.
\]


Let $\dot{\bfU}^\jmath$ be the $\bbQ(q)$-linear category with the object set $\X_\jmath$ and morphisms generated by $\calE_i: \lambda\mapsto \lambda+\alpha_i=\lambda-\alpha_{-i}$, 
$\calF_i: \lambda\mapsto \lambda-\alpha_i=\lambda+\alpha_{-i}$, for all $i\in\bbIj$,
subject to the following relations for $i\neq j$:
\begin{align}
[\calE_i,\calF_j]1_\lambda &=0, 
 \label{eq:cartano}\\
[\calE_i,\calF_i]1_\lambda &=[\lambda_i] 1_\lambda,\quad \forall\ i\neq\hf,
 \label{eq:cartan}\\
\sum_{\scriptscriptstyle {a+b=1-\l\alpha_i^\vee,\alpha_j\r}}(-1)^a\calE_i^{(a)}\calE_j\calE_i^{(b)}
 &=\sum_{\scriptscriptstyle{a+b=1-\l\alpha_i^\vee,\alpha_j\r}}(-1)^a\calF_i^{(a)}\calF_j\calF_i^{(b)}=0,  
 \label{eq:serre}\\
(\calE_\hf^{(2)}\calF_\hf-\calE_\hf\calF_\hf\calE_\hf+\calF_\hf\calE_\hf^{(2)})1_\lambda
 &=-(q^{\lambda_\hf+2}+q^{-\lambda_\hf-2})\calE_\hf 1_\lambda,\label{eq:jserre1}
 \\
(\calF_\hf^{(2)}\calE_\hf-\calF_\hf\calE_\hf\calF_\hf+\calE_\hf\calF_\hf^{(2)})1_\lambda
 &=-(q^{\lambda_\hf-1}+q^{-\lambda_\hf+1})\calF_\hf 1_\lambda.
\label{eq:jserre2}
\end{align}
The definition of $\dot{\bfU}^\jmath$ here is basically the same as the one in \cite{LW15}
but differs from the one in \cite{BKLW} which used a different object set. 
Note that a major difference in the presentation of $\dot\bfU^\jmath$ with respect to $\dot\bfU$ is that 
there is no relation between monomials of length $\leq 2$ in $\calE_\hf$, $\calF_\hf$, but rather they satisfy the inhomogeneous 
relations \eqref{eq:jserre1}, \eqref{eq:jserre2}, 
which will be referred as $\jmath$\emph{-Serre relations}.
Note also that the algebra $\bfU^\jmath$ or $\dot{\bfU}^\jmath$ does not have any natural triangular decomposition 
due to the mixture of $\calE_\hf$ and $\calF_\hf$ in these $\jmath$-Serre relations.

Let ${}_{\calA}{\bdUj}$ be the $\calA$-linear subcategory of $\bdUj$ with the same objects and with morphisms generated by 
divide powers $\calE_i^{(a)}1_\lambda$, $\calF_i^{(a)}1_\lambda$, for all $i\in\bbIj$, $\lambda \in \X_\jmath$, and $a\ge 0$. 
It was shown in \cite{BKLW, LW15} that ${}_{\calA}{\bdUj}$ is a free $\calA$-module and $\bbQ(q) \otimes_{\calA} {}_{\calA}{\bdUj}  = {\bdUj}$.

Denote by $\varpi$ the unique element in $\X_\jmath$ such that
\begin{equation*}
\varpi_\hf =1, \qquad \varpi_i=0 \quad (i\in\bbI^\jmath \backslash \{\hf\}).
\end{equation*}
The following are the idempotented counterparts of the symmetries for $\bfU^\jmath$ given in \cite[Lemma 6.1]{BW13}.

\begin{lemma}  \hfill
\label{lem:involutions}
\begin{enumerate}\renewcommand{\theenumi}{\alph{enumi}}
\item
 There exists an involution $\omega_\jmath$ on $\dot\bfU^\jmath$ as a $\bbQ(q)$-algebra
 which  sends  $1_\lambda$ to $1_{-\lambda -\varpi}$ 
 and switches $\calE_i$ with $\calF_i$, for all $i$ and $\lambda \in \X_\jmath$.
 \item
 There exists an anti-involution $\sigma_\jmath$ on $\dot\bfU^\jmath$ as a $\bbQ(q)$-algebra which sends 
 $1_\lambda$ to $1_{-\lambda -\varpi}$ 
and fixes $\calE_i, \calF_i$, for all $i$ and $\lambda \in \X_\jmath$.
 \item
There exists an involution (called  \emph{bar involution}) $\psi_\jmath$ on $\dot \bfU^\jmath$ as a $\bbQ$-algebra
 which sends $q$ to $q^{-1}$,  and fixes $1_\lambda$, $\calE_i$, $\calF_i$, for all $i$ and $\lambda \in \X_\jmath$.  
 \end{enumerate}
 In particular, all three (anti-)involutions $\omega_\jmath, \sigma_\jmath$ and $\psi_\jmath$ preserve the $\calA$-form ${}_{\calA}{\bdUj}$.
 \end{lemma}

\begin{proof}
Follows by inspection of the defining relations for $\dot{\bfU}^\jmath$ in \eqref{eq:cartano}--\eqref{eq:jserre2}. 
\end{proof}


\section{Coideal $2$-categories }
\label{sec:coideal-2-cat}

In this section, we introduce the main object of this paper, the
$2$-category $\frakU^\jmath$ (see Definition~
\ref{df:Uj}). 
This is an analogue of the category $\frakU$ introduced in
\cite[Section~5.1]{KLIII} (denoted by $\calU$ therein).  Recall that $ \frakU$ is a strict graded additive $\bfk$-linear $2$-category
with object set $\X$. Its $1$-morphisms are generated by 
\begin{align*}
E_i:\lambda \rightarrow \lambda+\alpha_i,\quad F_i:\lambda \rightarrow \lambda-\alpha_i \qquad (i \in \bbI),
\end{align*}
with generating $2$-morphisms given by diagrams modulo
certain diagrammatic relations.  We will not reproduce these here,
since they appear in the relations  (\ref{adj}--\ref{mixedup}) below  (with  $\thal^\vee_i$ everywhere replaced by $\alpha_i^\vee$)
which do not involve $i=\hf$.

We start by introducing a weak version $\widetilde{\frakU}^\jmath$ of
this 2-category, which lacks the most difficult relation to verify.
This notion is useful since we will show later that a representation of
$\widetilde{\frakU}^\jmath$ satisfying certain conditions on its
Grothendieck group descends to an action of $\frakU^\jmath$ (Proposition~\ref{df-prop}).

\subsection{The $2$-category $\widetilde{\frakU}^\jmath$}

Fix a commutative ring $\bfk$ with identity such that $2$ is a unit. 
For $i,j\in\bbI^\jmath$ we set 
$$t_{ij}=\begin{cases} 
-1, &\text{ if } j=i-1,\\
1, &\text{ otherwise. }
\end{cases}
$$

\begin{df}\label{df:weakfrakUj}
The {\em weak coideal $2$-category} $\widetilde{\frakU}^\jmath$
is the strict graded additive $\bfk$-linear $2$-category
with objects $\lambda$ for all $\lambda\in\X_\jmath$.
The $1$-morphisms are generated by 
\begin{align*}
\calE_i:\lambda \rightarrow \lambda+\alpha_i,\quad\calF_i:\lambda \rightarrow \lambda-\alpha_i, 
\qquad
\text{for }i \in \bbI^\jmath.
\end{align*}
Here ``generated" means taking all the direct sums of compositions of shifts of these $1$-morphisms.
The generating $2$-morphisms are presented by diagrams:
denote the identity $2$-morphisms
of $\calE_i 1_\lambda$ and $\calF_i 1_\lambda$ by
${\scriptstyle\substack{\lambda+\alpha_i\\\phantom{-}}}\substack{{\color{darkred}{\displaystyle\uparrow}} \\
  {\scriptscriptstyle i}}{\scriptstyle\substack{\lambda\\\phantom{-}}}
$ and ${\scriptstyle\substack{\lambda-\alpha_i
\\\phantom{-}}}\substack{{\color{darkred}{\displaystyle\downarrow}} \\
  {\scriptscriptstyle
    i}}{\scriptstyle\substack{\lambda\\\phantom{-}}}$,  and the other generators are
\begin{alignat*}{3}
&x 
= 
\mathord{

}: \calF_i\calE_i1_\lambda\to 1_\lambda\{1-\langle\thal^\vee_i,\lambda+\alpha_i\rangle\}.
\end{alignat*}

The $2$-morphisms are subject to the following relations (1)--(7): 
\begin{enumerate}
\item (Adjunction)
\begin{align}\label{adj}
\mathord{
%
}.
\end{align}

\item (Quiver Hecke relations)
\begin{align}\label{qha}
\mathord{
%
}.
\end{align}
Then for any $i$ we have
\begin{align}\label{nodal}
\mathord{
%
}.
\end{align}

\noindent In all the sums above, the indices that are not on a circle run over non-negative integers.

\item  (Mixed relations) For any $i\neq j$,

\begin{align}
\label{mixedup}
\mathord{
%
}.
\end{align}
\end{enumerate}
\end{df}

\begin{rk}
The bubbles with negative labels  are determined by the relation
\eqref{bubble3} and called fake bubbles. Note that because of
\eqref{bubble1}, these can only be defined when $2$ is invertible.
\end{rk}
\bigskip
\renewcommand{\arraystretch}{1.2} 
We organize the  generating 2-morphisms with their degrees (other than the identity $2$-morphisms of degree $0$ and  $x, x'$ of degree $2$) 
 in the following table.
\begin{center}
\begin{tabular}{|c|c|c|c|c|c|c|}
\hline
 & 
 $\mathord{
\begin{tikzpicture}[baseline = 0]
	\draw[<-,thick,darkred] (0.4,0.3) to[out=-90, in=0] (0.1,-0.1);
	\draw[-,thick,darkred] (0.1,-0.1) to[out = 180, in = -90] (-0.2,0.3);
    \node at (-0.2,.4) {$\scriptstyle{i}$};
  \node at (0.3,-0.15) {$\scriptstyle{\lambda}$};
\end{tikzpicture}
}$
&
$
\mathord{
\begin{tikzpicture}[baseline = 0]
	\draw[-,thick,darkred] (0.4,0.3) to[out=-90, in=0] (0.1,-0.1);
	\draw[->,thick,darkred] (0.1,-0.1) to[out = 180, in = -90] (-0.2,0.3);
    \node at (0.4,.4) {$\scriptstyle{i}$};
  \node at (0.3,-0.15) {$\scriptstyle{\lambda}$};
\end{tikzpicture}
}$
&
$\mathord{
\begin{tikzpicture}[baseline = 0]
	\draw[<-,thick,darkred] (0.4,-0.1) to[out=90, in=0] (0.1,0.3);
	\draw[-,thick,darkred] (0.1,0.3) to[out = 180, in = 90] (-0.2,-0.1);
    \node at (-0.2,-.2) {$\scriptstyle{i}$};
  \node at (0.3,0.4) {$\scriptstyle{\lambda}$};
\end{tikzpicture}
}$
&
$\mathord{
\begin{tikzpicture}[baseline = 0]
	\draw[-,thick,darkred] (0.4,-0.1) to[out=90, in=0] (0.1,0.3);
	\draw[->,thick,darkred] (0.1,0.3) to[out = 180, in = 90] (-0.2,-0.1);
    \node at (0.4,-.2) {$\scriptstyle{i}$};
  \node at (0.3,0.4) {$\scriptstyle{\lambda}$};
\end{tikzpicture}
}$
&
$\mathord{
\begin{tikzpicture}[baseline = 0]
	\draw[->,thick,darkred] (0.28,-.3) to (-0.28,.4);
	\draw[->,thick,darkred] (-0.28,-.3) to (0.28,.4);
   \node at (-0.28,-.4) {$\scriptstyle{i}$};
   \node at (0.28,-.4) {$\scriptstyle{j}$};
   \node at (.4,.05) {$\scriptstyle{\lambda}$};
\end{tikzpicture}
}$
&
$
\mathord{
\begin{tikzpicture}[baseline = 0]
	\draw[<-,thick,darkred] (0.28,-.3) to (-0.28,.4);
	\draw[<-,thick,darkred] (-0.28,-.3) to (0.28,.4);
   \node at (-0.28,-.4) {$\scriptstyle{i}$};
   \node at (0.28,-.4) {$\scriptstyle{j}$};
   \node at (.4,.05) {$\scriptstyle{\lambda}$};
\end{tikzpicture}
}$
\\
\hline
\end{tabular}
\end{center}
\renewcommand{\arraystretch}{1} 

The grading on $\widetilde{\frakU}^\jmath$ is well defined, \emph{i.e.}, all the relations in the definition are homogeneous. 
This can be checked in the same way as in \cite{KLIII}. 
Note that the main difference here is that 
$$
 -1+\langle\thal^\vee_i,\lambda+\alpha_i\rangle=
 \begin{cases}
1+\lambda_i, & \text{ if } i\neq\hf,
\\
2+\lambda_i, & \text{ if } i= \hf.
\end{cases}
$$ 
It follows that the diagrams $\chi$, $\chi'$ in \eqref{sigrel} have degree zero unless $i=j=\hf$, and in that case they both have degree $1$.

Note that subcategory of $\widetilde{\frakU}^\jmath$ generated by
$\calE_i,\calF_i$ with $i>\hf$ satisfy the relations of
Khovanov-Lauda's $2$-category $\frakU_{r-1}$ of rank $r-1$.  This
yields a functor $\frakU_{r-1}\to \widetilde{\frakU}^\jmath$. 
We will show in a future work that this functor is faithful.

\subsection{The coideal $2$-category $\frakU^\jmath$}
For the sake of simplicity, we make the following convention.

\noindent {\em We will omit the index $\hf$
from the diagrams, that is, all the strands without labels should be viewed as labeled by $\hf$. }

\begin{df}\label{df:Uj}
The {\em coideal $2$-category} $\frakU^\jmath$
is the quotient of the $2$-category $\widetilde{\frakU}^\jmath$ by the
relation
\begin{align}\label{Pi=1}
\mathord{
\begin{tikzpicture}[baseline = 0]
	\draw[->,thick,darkred] (0.45,.6) to (0.45,-.6);
	\draw[->,thick,darkred] (-0.45,.6) to (-0.45,-.6);
        \draw[->,thick,darkred] (0,-.6) to (0,0.6);
   \node at (.6,.05) {$\scriptstyle{\lambda}$};
\end{tikzpicture}
}=
\mathord{
\begin{tikzpicture}[baseline = 0]
	\draw[->,thick,darkred] (0.45,.6) to (-0.45,-.6);
	\draw[<-,thick,darkred] (0.45,-.6) to (-0.45,.6);
        \draw[-,thick,darkred] (0,-.6) to[out=90,in=-90] (-.45,0);
        \draw[->,thick,darkred] (-0.45,0) to[out=90,in=-90] (0,0.6);
   \node at (.5,.05) {$\scriptstyle{\lambda}$};
\end{tikzpicture}
}
-\mathord{
\begin{tikzpicture}[baseline = 0]
	\draw[->,thick,darkred] (0.45,.6) to (-0.45,-.6);
	\draw[<-,thick,darkred] (0.45,-.6) to (-0.45,.6);
        \draw[-,thick,darkred] (0,-.6) to[out=90,in=-90] (.45,0);
        \draw[->,thick,darkred] (0.45,0) to[out=90,in=-90] (0,0.6);
   \node at (.6,.05) {$\scriptstyle{\lambda}$};
\end{tikzpicture}
}
-
\mathord{
\begin{tikzpicture}[baseline = 0]
	\draw[->,thick,darkred] (0.45,.6) to (-0.45,-.6);
        \draw[-,thick,darkred] (0,-.6) to[out=90,in=180] (.25,-.2);
        \draw[->,thick,darkred] (0.25,-.2) to[out=0,in=90] (.45,-0.6);
                \draw[-,thick,darkred] (-.45,.6) to[out=-90,in=180] (-.2,0.2);
        \draw[->,thick,darkred] (-.2,0.2) to[out=0,in=-90] (0,0.6);
   \node at (.6,.05) {$\scriptstyle{\lambda}$};
\end{tikzpicture}
}-
\mathord{
\begin{tikzpicture}[baseline = 0]
	\draw[->,thick,darkred] (-0.45,.6) to (0.45,-.6);
        \draw[<-,thick,darkred] (0,.6) to[out=-90,in=180] (.25,0.2);
        \draw[-,thick,darkred] (0.25,0.2) to[out=0,in=-90] (.45,0.6);
          \draw[<-,thick,darkred] (-.45,-.6) to[out=90,in=180] (-.2,-.2);
        \draw[-,thick,darkred] (-.2,-.2) to[out=0,in=90] (0,-0.6);
   \node at (.5,.05) {$\scriptstyle{\lambda}$};
\end{tikzpicture}
}+
\sum_{\substack{ s+t+u+v\\=-3}}\!\!\!
\mathord{
\begin{tikzpicture}[baseline=0]
	\draw[<-,thick,darkred] (0.4,-0.6) to[out=90, in=0] (.2,-0.3);
	\draw[-,thick,darkred] (.2,-0.3) to[out = 180, in = 90] (0,-.6);
   \node at (.35,-0.35) {$\color{darkred}\bullet$};
   \node at (0.5,-.35) {$\color{darkred}\scriptstyle{t}$};
  \draw[-,thick,darkred] (.2,0.2) to[out=180,in=90] (0,0);
  \draw[->,thick,darkred] (0.4,0) to[out=90,in=0] (.2,.2);
 \draw[-,thick,darkred] (0,0) to[out=-90,in=180] (.2,-0.2);
  \draw[-,thick,darkred] (.2,-0.2) to[out=0,in=-90] (0.4,0);
  \draw[->,thick,darkred](-.4,.6) to (-.4,-.6);
     \node at (-.4,0.1) {$\color{darkred}\bullet$};
        \node at (-.6,0.1) {$\color{darkred}\scriptstyle{u}$};
   \node at (1,0.05) {$\scriptstyle{\lambda}$};
   \node at (0,0) {$\color{darkred}\bullet$};
   \node at (-0.2,0) {$\color{darkred}\scriptstyle{s}$};
	\draw[-,thick,darkred] (0.4,.6) to[out=-90, in=0] (.2,0.3);
	\draw[->,thick,darkred] (.2,0.3) to[out = -180, in = -90] (0,.6);
   \node at (0.4,0.45) {$\color{darkred}\bullet$};
   \node at (0.6,0.45) {$\color{darkred}\scriptstyle{v}$};
\end{tikzpicture}
}
+\sum_{\substack{s+t+u+v\\=-3}}\!\!\!
\mathord{
\begin{tikzpicture}[baseline=0]
	\draw[-,thick,darkred] (0,-0.6) to[out=90, in=0] (-.2,-0.3);
	\draw[->,thick,darkred] (-.2,-0.3) to[out = 180, in = 90] (-0.4,-.6);
   \node at (-.05,-0.35) {$\color{darkred}\bullet$};
   \node at (0.1,-.35) {$\color{darkred}\scriptstyle{t}$};
  \draw[<-,thick,darkred] (-.2,0.2) to[out=180,in=90] (-.4,0);
  \draw[-,thick,darkred] (0,0) to[out=90,in=0] (-.2,.2);
 \draw[-,thick,darkred] (-.4,0) to[out=-90,in=180] (-.2,-0.2);
  \draw[-,thick,darkred] (-.2,-0.2) to[out=0,in=-90] (0,0);
  \draw[->,thick,darkred](.4,.6) to (.4,-.6);
     \node at (.4,0.1) {$\color{darkred}\bullet$};
        \node at (.6,0.1) {$\color{darkred}\scriptstyle{u}$};
   \node at (0.6,-0.2) {$\scriptstyle{\lambda}$};
   \node at (-0.4,0) {$\color{darkred}\bullet$};
   \node at (-0.6,0) {$\color{darkred}\scriptstyle{s}$};
	\draw[<-,thick,darkred] (0,.6) to[out=-90, in=0] (-.2,0.3);
	\draw[-,thick,darkred] (-.2,0.3) to[out = -180, in = -90] (-0.4,.6);
   \node at (-.4,0.45) {$\color{darkred}\bullet$};
   \node at (-.6,0.45) {$\color{darkred}\scriptstyle{v}$};
\end{tikzpicture}
}
\end{align}
\end{df}

This $2$-category $\frakU^\jmath$ admits an induced grading  from $\widetilde{\frakU}^\jmath$ as the relation \eqref{Pi=1}
is homogeneous of degree 0. 

We let $\dot\frakUj$ denote the idempotent completion of $\frakUj$. Note that since the quiver Hecke
relation holds in $\frakUj$, the divided powers $\calE^{(a)}_i$,
$\calF^{(a)}_i$ are well defined in $\dot\frakUj$, see for example \cite[Section~
3.5]{KLIII}.
 We may also write $\calE_{-i}=\calF_i$ for $i\in\bbI^\jmath$ whenever this is convenient.

\begin{rk}
Note that we have automatically the relation
\begin{align*}
\mathord{
\begin{tikzpicture}[baseline = 0]
	\draw[<-,thick,darkred] (0.45,.6) to (0.45,-.6);
	\draw[<-,thick,darkred] (-0.45,.6) to (-0.45,-.6);
        \draw[<-,thick,darkred] (0,-.6) to (0,0.6);
   \node at (.6,.05) {$\scriptstyle{\lambda}$};
\end{tikzpicture}
}=
-\ \mathord{
\begin{tikzpicture}[baseline = 0]
	\draw[<-,thick,darkred] (0.45,.6) to (-0.45,-.6);
	\draw[->,thick,darkred] (0.45,-.6) to (-0.45,.6);
        \draw[<-,thick,darkred] (0,-.6) to[out=90,in=-90] (-.45,0);
        \draw[-,thick,darkred] (-0.45,0) to[out=90,in=-90] (0,0.6);
   \node at (.5,.05) {$\scriptstyle{\lambda}$};
\end{tikzpicture}
}
+\mathord{
\begin{tikzpicture}[baseline = 0]
	\draw[<-,thick,darkred] (0.45,.6) to (-0.45,-.6);
	\draw[->,thick,darkred] (0.45,-.6) to (-0.45,.6);
        \draw[<-,thick,darkred] (0,-.6) to[out=90,in=-90] (.45,0);
        \draw[-,thick,darkred] (0.45,0) to[out=90,in=-90] (0,0.6);
   \node at (.6,.05) {$\scriptstyle{\lambda}$};
\end{tikzpicture}
}
-
\mathord{
\begin{tikzpicture}[baseline = 0]
	\draw[<-,thick,darkred] (0.45,.6) to (-0.45,-.6);
        \draw[<-,thick,darkred] (0,-.6) to[out=90,in=180] (.25,-.2);
        \draw[-,thick,darkred] (0.25,-.2) to[out=0,in=90] (.45,-0.6);
                \draw[<-,thick,darkred] (-.45,.6) to[out=-90,in=180] (-.2,0.2);
        \draw[-,thick,darkred] (-.2,0.2) to[out=0,in=-90] (0,0.6);
   \node at (.6,.05) {$\scriptstyle{\lambda}$};
\end{tikzpicture}
}-
\mathord{
\begin{tikzpicture}[baseline = 0]
	\draw[<-,thick,darkred] (-0.45,.6) to (0.45,-.6);
        \draw[-,thick,darkred] (0,.6) to[out=-90,in=180] (.25,0.2);
        \draw[->,thick,darkred] (0.25,0.2) to[out=0,in=-90] (.45,0.6);
          \draw[-,thick,darkred] (-.45,-.6) to[out=90,in=180] (-.2,-.2);
        \draw[->,thick,darkred] (-.2,-.2) to[out=0,in=90] (0,-0.6);
   \node at (.5,.05) {$\scriptstyle{\lambda}$};
\end{tikzpicture}
}+
\sum_{\substack{ s+t+u+v\\=-3}}\!\!\!
\mathord{
\begin{tikzpicture}[baseline=0]
	\draw[<-,thick,darkred] (0.4,-0.6) to[out=90, in=0] (.2,-0.3);
	\draw[-,thick,darkred] (.2,-0.3) to[out = 180, in = 90] (0,-.6);
   \node at (.35,-0.35) {$\color{darkred}\bullet$};
   \node at (0.5,-.35) {$\color{darkred}\scriptstyle{t}$};
  \draw[-,thick,darkred] (.2,0.2) to[out=180,in=90] (0,0);
  \draw[->,thick,darkred] (0.4,0) to[out=90,in=0] (.2,.2);
 \draw[-,thick,darkred] (0,0) to[out=-90,in=180] (.2,-0.2);
  \draw[-,thick,darkred] (.2,-0.2) to[out=0,in=-90] (0.4,0);
  \draw[<-,thick,darkred](.8,.6) to (.8,-.6);
     \node at (.8,0.1) {$\color{darkred}\bullet$};
        \node at (1.0,0.1) {$\color{darkred}\scriptstyle{u}$};
   \node at (1.2,0.05) {$\scriptstyle{\lambda}$};
   \node at (0,0) {$\color{darkred}\bullet$};
   \node at (-0.2,0) {$\color{darkred}\scriptstyle{s}$};
	\draw[-,thick,darkred] (0.4,.6) to[out=-90, in=0] (.2,0.3);
	\draw[->,thick,darkred] (.2,0.3) to[out = -180, in = -90] (0,.6);
   \node at (0.4,0.45) {$\color{darkred}\bullet$};
   \node at (0.6,0.45) {$\color{darkred}\scriptstyle{v}$};
\end{tikzpicture}
}
+\sum_{\substack{s+t+u+v\\=-3}}\!\!\!
\mathord{
\begin{tikzpicture}[baseline=0]
	\draw[-,thick,darkred] (0,-0.6) to[out=90, in=0] (-.2,-0.3);
	\draw[->,thick,darkred] (-.2,-0.3) to[out = 180, in = 90] (-0.4,-.6);
   \node at (-.05,-0.35) {$\color{darkred}\bullet$};
   \node at (0.1,-.35) {$\color{darkred}\scriptstyle{t}$};
  \draw[<-,thick,darkred] (-.2,0.2) to[out=180,in=90] (-.4,0);
  \draw[-,thick,darkred] (0,0) to[out=90,in=0] (-.2,.2);
 \draw[-,thick,darkred] (-.4,0) to[out=-90,in=180] (-.2,-0.2);
  \draw[-,thick,darkred] (-.2,-0.2) to[out=0,in=-90] (0,0);
  \draw[<-,thick,darkred](-.8,.6) to (-.8,-.6);
     \node at (-.8,0.1) {$\color{darkred}\bullet$};
        \node at (-1.0,0.1) {$\color{darkred}\scriptstyle{u}$};
   \node at (0.2,-0.2) {$\scriptstyle{\lambda}$};
   \node at (-0.4,0) {$\color{darkred}\bullet$};
   \node at (-0.6,0) {$\color{darkred}\scriptstyle{s}$};
	\draw[<-,thick,darkred] (0,.6) to[out=-90, in=0] (-.2,0.3);
	\draw[-,thick,darkred] (-.2,0.3) to[out = -180, in = -90] (-0.4,.6);
   \node at (-.4,0.45) {$\color{darkred}\bullet$};
   \node at (-.6,0.45) {$\color{darkred}\scriptstyle{v}$};
\end{tikzpicture}
}
\end{align*}
in $\frakU^\jmath$, since it is the image of \eqref{Pi=1} under the isomorphism $\End(\calF_\hf\calE_\hf\calF_\hf)\cong \End(\calE_\hf\calF_\hf\calE_\hf)$ given by adjunction.
\end{rk}

\vspace{3mm}
\subsection{Symmetries of $\frakUj$}
 \label{sec:symmetries}
 
Recall \cite[3.3.3]{KLIII} that the KLR 2-category $\frakU$ admits several symmetries which categorify the standard
(anti-)involutions for the quantum groups. We will describe the analogues of such symmetries for $\frakU^\jmath$,
which categorify the three (anti-)involutions in
Lemma~\ref{lem:involutions}. We will use the notation of
\cite[3.3.3]{KLIII}  and let $(\frakU^\jmath)^{\operatorname{op}},(\frakU^\jmath)^{\operatorname{co}},(\frakU^\jmath)^{\operatorname{coop}}$ to
denote the 2-category $\frakU^\jmath$ with the composition of
1-morphisms, 2-morphisms and both 1- and 2-morphisms reversed, respectively.


\begin{df}\hfill
  \begin{enumerate}\renewcommand{\theenumi}{\alph{enumi}}
  \item
   Consider the functor
    $$
    \psi_\jmath: \frakU^\jmath\longrightarrow (\frakU^\jmath)^{\operatorname{co}}
    $$
which sends
$$\lambda \mapsto  \lambda,
\qquad 1_{\mu} \calE_{i_1} \cdots \calE_{i_\ell} 1_{\lambda}
\{t\} \mapsto 1_{\mu } \calE_{i_1} \cdots \calE_{i_\ell}
1_{\lam} \{- t\},
 $$ 
 for $\lambda, \mu \in \X_\jmath$, $t \in \bbZ$, and
 $i_1, \ldots, i_\ell \in \bbI^\jmath$.  It acts on $2$-morphisms by
 reflection through the $x$-axis and reversing orientation of all
 strands.  Note that this action on $2$-morphisms preserves all
 defining relations on $2$-morphisms in $\frakU^\jmath$.  It induces a
 duality on $\dot{\frakUj}$ which will be again denoted by
 $\psi_\jmath$.  We have $\psi_\jmath^2=\text{id}$ and in particular
 $\psi_\jmath$ is invertible.

\item
  Consider the functor
 \begin{equation}\label{eq:def:omegajmath}
   \omega_\jmath:  \frakU^\jmath \longrightarrow \frakU^\jmath
 \end{equation}
 which sends (recall we write $\calE_{-i} = \calF_{i}$ for $i \in  \bbI^\jmath$)
 $$
 \lambda \mapsto -\lambda -\varpi, \qquad 1_{\mu} \calE_{i_1}
 \cdots \calE_{i_\ell} 1_{\lambda} \{t\} \mapsto 1_{-\lambda
   -\varpi} \calE_{- i_1} \cdots \calE_{- i_\ell} 1_{-\mu -\varpi}
 \{t\}.
 $$ 
 It acts on $2$-morphisms by
 \begin{align*}
   x&\mapsto x',&x'&\mapsto x & \tau&\mapsto -\tau',& \tau'&\mapsto
                                                             -\tau,\\
   \eta&\mapsto
       2^{-\delta_{\hf, i}}\eta' ,& \eta' &\mapsto \eta, & \epsilon &\mapsto
                                                                               2^{\delta_{\hf, i}}\epsilon',& \epsilon'&\mapsto
                                                                                                \epsilon.
 \end{align*}
 Thus, each diagram is sent to the diagram with reversed orientation,
 times certain unit in $\bfk$. The functor $\omega_\jmath$ is an equivalence. 
An inverse is provided by the functor which coincides with $\omega_\jmath$ on objects and $1$-morphisms, and acts on $2$-morphisms
 \begin{align*}
   x&\mapsto x',&x'&\mapsto x & \tau&\mapsto -\tau',& \tau'&\mapsto
                                                             -\tau,\\
   \eta&\mapsto
       \eta' ,& \eta' &\mapsto 2^{\delta_{\hf, i}}\eta, & \epsilon &\mapsto
                                                                               \epsilon',& \epsilon'&\mapsto
                                                                                                2^{-\delta_{\hf, i}}\epsilon.
 \end{align*}


\item
 Consider the functor
 \begin{equation}\label{eq:def:taujmath}
   \sigma_\jmath:  \frakU^\jmath \longrightarrow (\frakU^\jmath)^{\operatorname{op}}
 \end{equation}
which sends
 $$
 \lambda \mapsto -\lambda -\varpi, \qquad 1_{\mu} \calE_{i_1}
 \cdots \calE_{i_\ell} 1_{\lambda} \{t\} \mapsto 1_{-\lambda
   -\varpi} \calE_{i_\ell} \cdots \calE_{i_1} 1_{-\mu -\varpi}
 \{t\}.
  $$ 
  It acts on $2$-morphisms by 
  \begin{align*} x&\mapsto x,&x'&\mapsto
    x' & \tau&\mapsto -\tau,& \tau'&\mapsto
                                     -\tau',\\
                                \eta&\mapsto 2^{-\delta_{i, \hf}}
                                      \eta',& \eta' &\mapsto \eta, &
                                                                     \epsilon
             &\mapsto 2^{\delta_{i, \hf}} \epsilon'&
                                                      \epsilon'&\mapsto
                                                                 \epsilon
\end{align*}
Thus, every diagram is sent to its reflection through the $y$-axis, times certain unit in $\bfk$.
The functor $\sigma_\jmath$ is an equivalence. An inverse is provided by the functor which coincides with $\sigma_\jmath$ on objects and 1-morphisms and acts on $2$-morphisms as 
\begin{align*} 
x&\mapsto
                                x,&x'&\mapsto x' & \tau&\mapsto
                                -\tau,& \tau'&\mapsto
                                               -\tau',\\
                                   \eta&\mapsto \eta',& \eta' &\mapsto
                                                                2^{\delta_{i,
                                                                \hf}}
                                                                \eta,
                                                 & \epsilon &\mapsto
                                                              \epsilon'&
                                                                         \epsilon'&\mapsto
                                                                                    2^{-\delta_{i,
                                                                                    \hf}}
                                                                                    \epsilon.
\end{align*} 

      \end{enumerate}
 \end{df}



\begin{rk}
The presence of the factors of 2 above reflects an asymmetry in the
definition of $\frakUj$, specifically in the relations
\eqref{bubble1}--\eqref{bubble2}. 
We can rescale the generators $\eta'$ and $\epsilon'$ so that the bubbles in \eqref{bubble1}--\eqref{bubble2}
 evaluate to any pair of scalars whose product is $2^{\delta_{i, \hf}}
 $.  One obvious choice is to switch the role of clockwise and
 counterclockwise bubbles; we denote the resulting (isomorphic) 2-category with this
 presentation ${}' \frakU^\jmath$.  Alternatively, we could give a more symmetric definition by choosing
 a square root $\sqrt{2}$, and using $\sqrt 2^{\delta_{i, \hf}}$ in both
 \eqref{bubble1} and \eqref{bubble2}.  In this case, the definitions
 of $\sigma_\jmath$ and $\omega_\jmath$ would not require the 
 factors of $2$.

One perspective on how these different forms arise is that this 2-category is connected to geometry and
representation theory in types $\B$ and $\C$.  Many aspects of these constructions, such as the cohomology rings of the partial flag varieties
and the category of Soergel bimodules only depend on the Weyl group and its reflection representation, and thus are
identical in types $\B/\C$.  Others, such as the Demazure operators discussed in Section \ref{ss:demazure}, and thus the formulas for
integration on flag varieties do depend on the root systems, and thus differ by powers of 2 between types $\B$ and $\C$.  
The category  $\frakUj$ naturally arises from the computations with type $\B$ flag varieties (which we choose to use), 
while ${}' \frakU^\jmath$ would arise most naturally with type $\C$ geometry.
\end{rk}
 
\section{The Grothendieck group of $\dot{\frakUj}$}
\label{sec:Grothendieck}

Throughout the rest of the paper, we assume that $\bfk$ is a complete local
ring and continue to assume that $2$ is a unit; equivalently, we
require that the residue field of $\bfk$ is not of characteristic 2.  A particularly
important case is when $\bfk$ is itself a field of characteristic
$\neq 2$. 

In this section, we provide a categorification of the $\jmath$-Serre relations in the framework of $2$-category $\dot\frakUj$.
This allows us to define a surjective map to the Grothendieck group of $\dot\frakUj$ from ${}_\calA\dot{\bfU}^\jmath$. 

\subsection{Reformulation of $\jmath$-Serre relations}	

First let us rewrite the $\jmath$-Serre relation \eqref{eq:jserre2}
in the form \eqref{eq:jserre2:1}--\eqref{eq:jserre2:3}
and the $\jmath$-Serre relation \eqref{eq:jserre1} in the form  \eqref{eq:jserre:1}--\eqref{eq:jserre:3}, respectively:
\begin{align}
\calF_{\hf} \calE_{\hf} \calF_{\hf}  1_{\lambda} 
&= \calF^{(2)}_{\hf} \calE_{\hf}  1_{\lambda} + ( \calE_{\hf}  \calF^{(2)}_{\hf}- [\lambda_\hf - 2] \calF) 1_{\lambda} 
+ { [\lambda_\hf]} \cdot\calF_{\hf} 1_{\lambda} , 
&\text{ if } \lambda_\hf \ge 2; 
 \label{eq:jserre2:1}\\
\calF_{\hf}  \calE_{\hf} \calF_{\hf}  1_{\lambda}  
& =  \calF^{(2)}_{\hf} \calE_{\hf} 1_{\lambda} + \calE_{\hf}  \calF^{(2)}_{\hf}1_{\lambda} + 2\cdot \calF_{\hf} 1_{\lambda} , 
& \text{ if } \lambda_\hf = 1;
 \label{eq:jserre2:2}\\
\calF_{\hf} \calE_{\hf} \calF_{\hf} 1_{\lambda} 
&= \calE_{\hf} \calF^{(2)}_{\hf} 1_{\lambda} + (\calF^{(2)}_{\hf} \calE_{\hf} - [-\lambda_\hf] \calF) 1_{\lambda} 
	+ {[-\lambda_\hf+2]} \cdot \calF_{\hf} 1_{\lambda} , 
&\text{ if } \lambda_\hf \le 0; 
 \label{eq:jserre2:3}  \\
\calE_{\hf} \calF_{\hf} \calE_{\hf} 1_{\lambda} 
&= \calE^{(2)}_{\hf} \calF_{\hf} 1_{\lambda}  + (\calF_{\hf} \calE^{(2)}_{\hf} - [-3 - \lambda_\hf]\calE_{\hf})1_{\lambda} 
   + {[-1 -\lambda_\hf]} \cdot \calE_{\hf}1_{\lambda} , 
&\text{ if } \lambda_\hf \le -3; 
 \label{eq:jserre:1} \\
\calE_{\hf}\calF_{\hf} \calE_{\hf}  1_{\lambda} 
&=  \calF_{\hf} \calE^{(2)}_{\hf} 1_{\lambda} + \calE^{(2)}_{\hf} \calF_{\hf} 1_{\lambda} + 2\cdot  \calE_{\hf} 1_{\lambda},
&\text{ if } \lambda_\hf= -2; 
 \label{eq:jserre:2} \\
 \calE_{\hf} \calF_{\hf}  \calE_{\hf} 1_{\lambda} 
 &= \calF_{\hf} \calE^{(2)}_{\hf} 1_{\lambda} + (\calE^{(2)}_{\hf} \calF_{\hf} - [1+ \lambda_\hf] \calE_{\hf}) 1_{\lambda}
     +{[\lambda_\hf+3]} \cdot \calE_{\hf} 1_{\lambda}, 
 &\text{ if } \lambda_\hf \ge -1. 
  \label{eq:jserre:3}
\end{align}
In each of the formulas \eqref{eq:jserre2:1}--\eqref{eq:jserre:3} above  
the right-hand side 
is the expansion of the left hand side in terms of canonical basis elements of ${}_\calA\dot{\bfU}^\jmath$. 
Knowing these expansions facilitates understanding the structure of
the corresponding 1-morphisms in $\frakUj$, but we will not use explicitly the fact
that this is the canonical basis expansion.  It can be proved 
in several ways,  but we will leave aside the proof.
In fact, we can also derive this from Theorem~\ref{thm:maindecat} below
by showing that the factors on the right-hand sides correspond to indecomposable self-dual $1$-morphisms in the $2$-category $\dot\frakUj$
(see Remark~\ref{rem:iCB} below).

\subsection{Categorification of $\jmath$-Serre \eqref{eq:jserre2:1} } 
\label{subsec:jserre:1}

We denote by $\calE_{\hf} \calF^{(2)}_{\hf}  1_{\lambda}$ the idempotent  
	$\,\,	\mathord{
		\begin{tikzpicture}[baseline=0,scale=1]
			\draw[->,thick,darkred] (-0.56, -0.3) -- (-0.56, 0.4);
			\draw[<-,thick,darkred] (0.28,-.3) --node[label={[shift={(0.5,0.1)}]$\color{black}\scriptstyle{\lambda}$}] {} (-0.28,.4);
			\draw[<-,thick,darkred] (-0.28,-.3) --node[near end] {$\color{darkred}\bullet$} (0.28,.4);
		\end{tikzpicture}
			}
	$
in $\End(\calE_\hf \calF^2_\hf 1_\lambda)$ as a $1$-morphism in $\dot\frakUj$. 

\begin{lemma} \label{lem:dec:eff}
Let  $\lambda \in \X_\jmath$ such that $\lambda_{\hf} = \langle \thal_\hf^\vee, \lambda \rangle \ge 2$. 
\begin{enumerate}\renewcommand{\theenumi}{\alph{enumi}}
	\item
		The element \[\kappa= 
	\mathord{
		\begin{tikzpicture}[baseline=0,scale=1.5]
			\draw[->,thick,darkred] (-0.56, -0.3) -- (-0.56, 0.4);
			\draw[<-,thick,darkred] (0.28,-.3) --node[label={[shift={(0.6,0.2)}]$\color{black}\scriptstyle{\lambda}$}] {} (-0.28,.4);
			\draw[<-,thick,darkred] (-0.28,-.3) --node[near end] {$\color{darkred}\bullet$} (0.28,.4);
		\end{tikzpicture}
			}
 +  {\frac 12} \sum_{a+b + c=-2}
	\mathord{
		\begin{tikzpicture}[baseline = 0pt]
   		\draw[->,thick,darkred] (0.9, 0.2) -- (0.9,-0.7);
 		\draw[thick,darkred] (0.9,-0.7) to[out=-90, in=90] (0.3, -1.3);
  		\draw[->,thick,darkred] (0.2,0) to[out=90,in=0] (0,0.2);
  		\draw[-,thick,darkred] (0,0.2) to[out=180,in=90] (-.2,0);
		\draw[-,thick,darkred] (-.2,0) to[out=-90,in=180] (0,-0.2);
  		\draw[-,thick,darkred] (0,-0.2) to[out=0,in=-90] (0.2,0);
   		\node[label={[shift={(0.2,-0.5)}]$\color{darkred}\scriptstyle{a}$}] at (0.2,0) {$\color{darkred}\bullet$};
		\draw[<-,thick,darkred] (0.3,-.7) to[out=90, in=0] (0,-0.3);
		\draw[-,thick,darkred] (0.3,-.7)  to[out=-90, in=90] (0.9,-1.3) ;
		\draw[-,thick,darkred] (0,-0.3) to[out = 180, in = 90] (-0.3,-.7);
		\draw[-,thick, darkred] (-0.3,-.7) -- (-0.3, -1.3);
   		\node[label={[shift={(-0.2,-0.5)}]$\color{darkred}\scriptstyle{b}$}] at (-0.25,-0.5) {$\color{darkred}\bullet$};
   		\draw[-, thick, darkred] (-0.3, 0.9) -- (-0.3, 1.6);
		\draw[<-, thick, darkred] (0.9, 0.9) to[out=90, in=-90] (0.3, 1.6);
		\draw[->, thick, darkred] (0.9, 1.6) to[out=-90, in=90] node[near start] {$\color{darkred}\bullet$}  (0.3, 0.9);
		\draw[-, thick, darkred] (0.9, 0.9) to[out=-90, in=90] (0.9, 0.5);
		\draw[-,thick, darkred] (0.9, 0.5) -- (0.9, 0.2); 
		\draw[-,thick,darkred] (0.3,0.9) to[out=-90, in=0] (0,0.5);
		\draw[->,thick,darkred] (0,0.5) to[out = 180, in = -90] (-0.3,0.9);
      	\node[label={[shift={(-.3,-0.5)}] $\color{darkred}\scriptstyle{c}$}] at (0,0.5) {$\color{darkred}\bullet$};
        \node at (1.2,1.5) {$\scriptstyle{\lambda}$};
		\end{tikzpicture}
			}-\;
{\frac 12} \sum_{a+b +c = -3}
	\mathord{
		\begin{tikzpicture}[baseline = 0pt]
   			\draw[->,thick,darkred] (0.9, 0.2) -- (0.9,-0.7);
 			\draw[thick,darkred] (0.9,-0.7) to[out=-90, in=90] (0.3, -1.3);
  			\draw[->,thick,darkred] (0.2,0) to[out=90,in=0] (0,0.2);
  			\draw[-,thick,darkred] (0,0.2) to[out=180,in=90] (-.2,0);
			\draw[-,thick,darkred] (-.2,0) to[out=-90,in=180] (0,-0.2);
  			\draw[-,thick,darkred] (0,-0.2) to[out=0,in=-90] (0.2,0);
   			\node[label={[shift={(0.2,-0.5)}]$\color{darkred}\scriptstyle{a}$}] at (0.2,0) {$\color{darkred}\bullet$};
			\draw[<-,thick,darkred] (0.3,-.7) to[out=90, in=0] (0,-0.3);
			\draw[-,thick,darkred] (0.3,-.7)  to[out=-90, in=90] (0.9,-1.3) ;
			\draw[-,thick,darkred] (0,-0.3) to[out = 180, in = 90] (-0.3,-.7);
			\draw[-,thick, darkred] (-0.3,-.7) -- (-0.3, -1.3);
   			\node[label={[shift={(-0.2,-0.5)}]$\color{darkred}\scriptstyle{b}$}] at (-0.25,-0.5) {$\color{darkred}\bullet$};
   			\draw[-, thick, darkred] (-0.3, 0.9) -- (-0.3, 1.6);
			\draw[<-, thick, darkred] (0.9, 0.9) to[out=90, in=-90] (0.3, 1.6);
			\draw[->, thick, darkred] (0.9, 1.6) to[out=-90, in=90] node[near start] {$\color{darkred}\bullet$}  (0.3, 0.9);
			\draw[-, thick, darkred] (0.9, 0.9) to[out=-90, in=90] (0.9, 0.5);
			\draw[-, thick, darkred] (0.9, 0.5) -- node {$\color{darkred}\bullet$} (0.9, 0.2); 
			\draw[-,thick,darkred] (0.3,0.9) to[out=-90, in=0] (0,0.5);
			\draw[->,thick,darkred] (0,0.5) to[out = 180, in = -90] (-0.3,0.9);
     		\node[label={[shift={(-0.3,-0.5)}] $\color{darkred}\scriptstyle{c}$}] at (0,0.5) {$\color{darkred}\bullet$};
        		\node at (1.2,1.5) {$\scriptstyle{\lambda}$};
		\end{tikzpicture}
			} \in \End(\calE_\hf \calF^2_\hf 1_\lambda)
\]
is an idempotent. We denote the image of this element in $\calE_{\hf}  \calF^{2}_{\hf}$ by $(\calE_{\hf}  \calF^{(2)}_{\hf}  - [\lambda_{\hf} - 2]\calF_{\hf})1_{\lambda}$ when considered as a $1$-morphism in the idempotent completion of $\widetilde{\frakU}^\jmath$.

	\item
		 We have the following isomorphism of $1$-morphisms in the idempotent completion of $\widetilde{\frakU}^\jmath$,
\begin{align*}
	\calE_{\hf} \calF^{(2)}_{\hf}  1_{\lambda}  &\cong (\calE_{\hf}  \calF^{(2)}_{\hf}  - [\lambda_{\hf} - 2]\calF_{\hf})1_{\lambda} \oplus \calF_{\hf} 1_{\lambda}^{\oplus [\lambda_{\hf}-2]}.
\end{align*}
\end{enumerate}
\end{lemma}

\begin{proof}
We shall prove the two statements together. 
We first have the following system of $2$-morphisms
\[\mathord{
\begin{tikzpicture}[baseline=0,scale=1.5]
	\node (effabove) at (0,0) {$\scriptstyle{\calE_{\hf} \calF^{(2)}_{\hf} 1_{\lambda}}$};
	\node (1) at (-3, -1.5) {$\scriptstyle{\calF_{\hf} 1_{\lambda} \{\lambda_{\hf} - 3\}}$};
	\node at (-2, -1.5) {$\scriptstyle{\oplus}$};
	\node (2) at (-1, -1.5) {$\scriptstyle{\calF_{\hf} 1_{\lambda} \{\lambda_{\hf} - 5\}}$};
	\node at (0,-1.5) {$\scriptstyle{\dots}$};
	\node (3) at (1, -1.5) {$\scriptstyle{\calF_{\hf} 1_{\lambda} \{5- \lambda_{\hf}\}}$};
	\node at (2, -1.5) {$\scriptstyle{\oplus}$};
	\node (4) at (3, -1.5) {$\scriptstyle{\calF_{\hf} 1_{\lambda} \{3- \lambda_{\hf}\}}$};
	\node (effbelow) at (0,-3) {$\scriptstyle{\calE_{\hf} \calF^{(2)}_{\hf} 1_{\lambda}}$};
	\draw[->] (effbelow) --node[left, below] {$\scriptstyle{\pi_0}$} (1); 
	\draw[->] (effbelow) --node[left] {$\scriptstyle{\pi_1}$} (2); 
	\draw[->] (effbelow) --node[right] {$\scriptstyle{\pi_{\lambda_{\hf}-4}}$} (3); 
	\draw[->] (effbelow) --node[right, below] {$\scriptstyle{\pi_{\lambda_{\hf}-3}}$} (4); 
	\draw[->] (1) --node[left, above] {$\scriptstyle{\iota_0}$} (effabove); 
	\draw[->] (2) --node[left] {$\scriptstyle{\iota_1}$} (effabove); 
	\draw[->] (3) --node[right] {$\scriptstyle{\iota_{\lambda_{\hf}-4}}$} (effabove); 
	\draw[->] (4) --node[right, above] {$\scriptstyle{\iota_{\lambda_{\hf}-3}}$} (effabove);
\end{tikzpicture}
}\]
where the $2$-morphisms $\iota_s$ and $\pi_s$ are defined as follows ($ 0 \le s \le \lambda_{\hf} -3$)
\[
\pi_s := {\frac12} \sum_{a+b = s- \lambda_{\hf} +1}
	\mathord{
		\begin{tikzpicture}[baseline = -14pt]
   			\node at (1,0.02) {$\scriptstyle{\lambda}$};
   			\draw[->,thick,darkred] (0.7, 0.2) -- (0.7,-0.7);
 			\draw[thick,darkred] (0.7,-0.7) to[out=-90, in=90] (0.3, -1.1);
  			\draw[->,thick,darkred] (0.2,0) to[out=90,in=0] (0,0.2);
  			\draw[-,thick,darkred] (0,0.2) to[out=180,in=90] (-.2,0);
			\draw[-,thick,darkred] (-.2,0) to[out=-90,in=180] (0,-0.2);
  			\draw[-,thick,darkred] (0,-0.2) to[out=0,in=-90] (0.2,0);
   			\node[label={[shift={(0.2,-0.5)}]$\color{darkred}\scriptstyle{a}$}] at (0.2,0) {$\color{darkred}\bullet$};
			\draw[<-,thick,darkred] (0.3,-.7) to[out=90, in=0] (0,-0.3);
			\draw[-,thick,darkred] (0.3,-.7)  to[out=-90, in=90] (0.7,-1.1) ;
			\draw[-,thick,darkred] (0,-0.3) to[out = 180, in = 90] (-0.3,-.7);
			\draw[-,thick, darkred] (-0.3,-.7) -- (-0.3, -1.1);
  			\node[label={[shift={(-0.2,-0.5)}]$\color{darkred}\scriptstyle{b}$}] at (-0.25,-0.5) {$\color{darkred}\bullet$};
		\end{tikzpicture}
			} - \;
{\frac12} \sum_{a+b = s- \lambda_{\hf} }
	\mathord{
		\begin{tikzpicture}[baseline = -14pt]
   			\node at (1,0.02) {$\scriptstyle{\lambda}$};
   			\draw[->,thick,darkred] (0.7, 0.2) -- node {$\color{darkred}\bullet$} (0.7,-0.7);
 			\draw[thick,darkred] (0.7,-0.7) to[out=-90, in=90] (0.3, -1.1);
  			\draw[->,thick,darkred] (0.2,0) to[out=90,in=0] (0,0.2);
  			\draw[-,thick,darkred] (0,0.2) to[out=180,in=90] (-.2,0);
			\draw[-,thick,darkred] (-.2,0) to[out=-90,in=180] (0,-0.2);
  			\draw[-,thick,darkred] (0,-0.2) to[out=0,in=-90] (0.2,0);
   			\node[label={[shift={(0.2,-0.5)}]$\color{darkred}\scriptstyle{a}$}] at (0.2,0) {$\color{darkred}\bullet$};
			\draw[<-,thick,darkred] (0.3,-.7) to[out=90, in=0] (0,-0.3);
			\draw[-,thick,darkred] (0.3,-.7)  to[out=-90, in=90] (0.7,-1.1) ;
			\draw[-,thick,darkred] (0,-0.3) to[out = 180, in = 90] (-0.3,-.7);
			\draw[-,thick, darkred] (-0.3,-.7) -- (-0.3, -1.1);
   			\node[label={[shift={(-0.2,-0.5)}]$\color{darkred}\scriptstyle{b}$}] at (-0.25,-0.5) {$\color{darkred}\bullet$};
		\end{tikzpicture}
				}
\qquad \text{ and } \qquad
\iota_s := -
	\mathord{
		\begin{tikzpicture}[baseline = 16pt]
			\draw[-, thick, darkred] (-0.3, 0.7) -- (-0.3, 1.4);
			\draw[<-, thick, darkred] (0.9, 0.7) to[out=90, in=-90] (0.3, 1.4);
			\draw[->, thick, darkred] (0.9, 1.4) to[out=-90, in=90] node[near start] {$\color{darkred}\bullet$}  (0.3, 0.7);
			\draw[-, thick, darkred] (0.9, 0.7) to[out=-90, in=90] (0.9, 0.3);
			\draw[-,thick,darkred] (0.3,0.7) to[out=-90, in=0] (0,0.3);
			\draw[->,thick,darkred] (0,0.3) to[out = 180, in = -90] (-0.3,0.7);
      		\node[label=below: $\color{darkred}\scriptstyle{\lambda_{\hf}-3-s}$] at (0,0.3) {$\color{darkred}\bullet$};
          	\node at (1.2,1.3) {$\scriptstyle{\lambda}$};
		\end{tikzpicture}
			}.
\]

We have ($0 \le s, s' \le \lambda_{\hf} -3$)
\begin{align*}
\pi_s \cdot \iota_{s'} 
&= -  {\frac 12} \sum_{a+b = s- \lambda_{\hf} +1}
	\mathord{
		\begin{tikzpicture}[baseline = -35pt]
   			\draw[->,thick,darkred] (0.9, 0.2) -- (0.9,-0.7);
 			\draw[thick,darkred] (0.9,-0.7) to[out=-90, in=90] (0.3, -1.3);
  			\draw[->,thick,darkred] (0.2,0) to[out=90,in=0] (0,0.2);
  			\draw[-,thick,darkred] (0,0.2) to[out=180,in=90] (-.2,0);
			\draw[-,thick,darkred] (-.2,0) to[out=-90,in=180] (0,-0.2);
  			\draw[-,thick,darkred] (0,-0.2) to[out=0,in=-90] (0.2,0);
   			\node[label={[shift={(0.2,-0.5)}]$\color{darkred}\scriptstyle{a}$}] at (0.2,0) {$\color{darkred}\bullet$};
			\draw[<-,thick,darkred] (0.3,-.7) to[out=90, in=0] (0,-0.3);
			\draw[-,thick,darkred] (0.3,-.7)  to[out=-90, in=90] (0.9,-1.3) ;
			\draw[-,thick,darkred] (0,-0.3) to[out = 180, in = 90] (-0.3,-.7);
			\draw[-,thick, darkred] (-0.3,-.7) -- (-0.3, -1.3);
   			\node[label={[shift={(-0.2,-0.5)}]$\color{darkred}\scriptstyle{b}$}] at (-0.25,-0.5) {$\color{darkred}\bullet$};
   			\draw[-, thick, darkred] (-0.3, -2) -- (-0.3, -1.3);
			\draw[<-, thick, darkred] (0.9, -2) to[out=90, in=-90] (0.3, -1.3);
			\draw[->, thick, darkred] (0.9, -1.3) to[out=-90, in=90] node[near start] {$\color{darkred}\bullet$}  (0.3, -2);
			\draw[-, thick, darkred] (0.9, -2) to[out=-90, in=90] (0.9, -2.4);
			\draw[-,thick, darkred] (0.9, -2.4) -- (0.9, -2.7); 
			\draw[-,thick,darkred] (0.3, -2) to[out=-90, in=0] (0, -2.4);
			\draw[->,thick,darkred] (0, -2.4) to[out = 180, in = -90] (-0.3, -2);
      		\node[label={[shift={(-0.8,-0.5)}] $\color{darkred}\scriptstyle{\lambda_{\hf}-3-s'}$}] at (0,-2.4) {$\color{darkred}\bullet$};
         	\node at (1.2, 0) {$\scriptstyle{\lambda}$};
		\end{tikzpicture}
			}+\;
{\frac 12} \sum_{a+b = s- \lambda_{\hf}}
	\mathord{
		\begin{tikzpicture}[baseline = -35pt]
   			\draw[->,thick,darkred] (0.9, 0.2) -- node {$\color{darkred}\bullet$} (0.9,-0.7);
 			\draw[thick,darkred] (0.9,-0.7) to[out=-90, in=90] (0.3, -1.3);
  			\draw[->,thick,darkred] (0.2,0) to[out=90,in=0] (0,0.2);
  			\draw[-,thick,darkred] (0,0.2) to[out=180,in=90] (-.2,0);
			\draw[-,thick,darkred] (-.2,0) to[out=-90,in=180] (0,-0.2);
  			\draw[-,thick,darkred] (0,-0.2) to[out=0,in=-90] (0.2,0);
   			\node[label={[shift={(0.2,-0.5)}]$\color{darkred}\scriptstyle{a}$}] at (0.2,0) {$\color{darkred}\bullet$};
			\draw[<-,thick,darkred] (0.3,-.7) to[out=90, in=0] (0,-0.3);
			\draw[-,thick,darkred] (0.3,-.7)  to[out=-90, in=90] (0.9,-1.3) ;
			\draw[-,thick,darkred] (0,-0.3) to[out = 180, in = 90] (-0.3,-.7);
			\draw[-,thick, darkred] (-0.3,-.7) -- (-0.3, -1.3);
   			\node[label={[shift={(-0.2,-0.5)}]$\color{darkred}\scriptstyle{b}$}] at (-0.25,-0.5) {$\color{darkred}\bullet$};
     		\draw[-, thick, darkred] (-0.3, -2) -- (-0.3, -1.3);
			\draw[<-, thick, darkred] (0.9, -2) to[out=90, in=-90]  (0.3, -1.3);
			\draw[->, thick, darkred] (0.9, -1.3) to[out=-90, in=90] node[near start] {$\color{darkred}\bullet$}  (0.3, -2);
			\draw[-, thick, darkred] (0.9, -2) to[out=-90, in=90] (0.9, -2.4);
			\draw[-,thick, darkred] (0.9, -2.4) -- (0.9, -2.7); 
			\draw[-,thick,darkred] (0.3, -2) to[out=-90, in=0] (0, -2.4);
			\draw[->,thick,darkred] (0, -2.4) to[out = 180, in = -90] (-0.3, -2);
      		\node[label={[shift={(-0.8,-0.5)}] $\color{darkred}\scriptstyle{\lambda_{\hf}-3-s'}$}] at (0,-2.4) {$\color{darkred}\bullet$};
         	\node at (1.2, 0) {$\scriptstyle{\lambda}$};
		\end{tikzpicture}.
				}
\end{align*}
Then we have 
\begin{align*}
&-  {\frac 12} \sum_{a+b = s- \lambda_{\hf} +1} \!\!\!\!
	\mathord{
		\begin{tikzpicture}[baseline = -35pt]
  			 \draw[->,thick,darkred] (0.9, 0.2) -- (0.9,-0.7);
 			 \draw[thick,darkred] (0.9,-0.7) to[out=-90, in=90] (0.3, -1.3);
  			 \draw[->,thick,darkred] (0.2,0) to[out=90,in=0] (0,0.2);
 			 \draw[-,thick,darkred] (0,0.2) to[out=180,in=90] (-.2,0);
			 \draw[-,thick,darkred] (-.2,0) to[out=-90,in=180] (0,-0.2);
  			 \draw[-,thick,darkred] (0,-0.2) to[out=0,in=-90] (0.2,0);
  			 \node[label={[shift={(0.2,-0.5)}]$\color{darkred}\scriptstyle{a}$}] at (0.2,0) {$\color{darkred}\bullet$};
		   	 \draw[<-,thick,darkred] (0.3,-.7) to[out=90, in=0] (0,-0.3);
			 \draw[-,thick,darkred] (0.3,-.7)  to[out=-90, in=90] (0.9,-1.3) ;
			 \draw[-,thick,darkred] (0,-0.3) to[out = 180, in = 90] (-0.3,-.7);
			 \draw[-,thick, darkred] (-0.3,-.7) -- (-0.3, -1.3);
   			 \node[label={[shift={(-0.2,-0.5)}]$\color{darkred}\scriptstyle{b}$}] at (-0.25,-0.5) {$\color{darkred}\bullet$};
   			\draw[-, thick, darkred] (-0.3, -2) -- (-0.3, -1.3);
			\draw[<-, thick, darkred] (0.9, -2) to[out=90, in=-90] (0.3, -1.3);
			\draw[->, thick, darkred] (0.9, -1.3) to[out=-90, in=90] node[near start] {$\color{darkred}\bullet$}  (0.3, -2);
			\draw[-, thick, darkred] (0.9, -2) to[out=-90, in=90] (0.9, -2.4);
			\draw[-,thick, darkred] (0.9, -2.4) -- (0.9, -2.7); 
			\draw[-,thick,darkred] (0.3, -2) to[out=-90, in=0] (0, -2.4);
			\draw[->,thick,darkred] (0, -2.4) to[out = 180, in = -90] (-0.3, -2);
     	 	\node[label={[shift={(-0.8,-0.5)}] $\color{darkred}\scriptstyle{\lambda_{\hf}-3-s'}$}] at (0,-2.4) {$\color{darkred}\bullet$};
         	\node at (1.2, 0) {$\scriptstyle{\lambda}$};
		\end{tikzpicture}
			} = \; 
-  {\frac 12} \sum_{a+b = s - \lambda_{\hf} +1}\!\!\!\!
	\mathord{
		\begin{tikzpicture}[baseline = -35pt]
  			 \draw[->,thick,darkred] (0.9, 0.2) -- (0.9,-0.7);
 			 \draw[thick,darkred] (0.9,-0.7) to[out=-90, in=90] (0.3, -1.3);
  			 \draw[->,thick,darkred] (0.2,0) to[out=90,in=0] (0,0.2);
  		 	 \draw[-,thick,darkred] (0,0.2) to[out=180,in=90] (-.2,0);
			 \draw[-,thick,darkred] (-.2,0) to[out=-90,in=180] (0,-0.2);
  			 \draw[-,thick,darkred] (0,-0.2) to[out=0,in=-90] (0.2,0);
   			 \node[label={[shift={(0.2,-0.5)}]$\color{darkred}\scriptstyle{a}$}] at (0.2,0) {$\color{darkred}\bullet$};
		     \draw[<-,thick,darkred] (0.3,-.7) to[out=90, in=0] (0,-0.3);
			 \draw[-,thick,darkred] (0.3,-.7)  to[out=-90, in=90] (0.9,-1.3) ;
			 \draw[-,thick,darkred] (0,-0.3) to[out = 180, in = 90] (-0.3,-.7);
			 \draw[-,thick, darkred] (-0.3,-.7) -- (-0.3, -1.3);
   			 \node[label={[shift={(-0.9,-0.5)}]$\color{darkred}\scriptstyle{b + \lambda_{\hf} - 3 - s'}$}] at (-0.25,-0.5) {$\color{darkred}\bullet$};
    			 \node at (1.2, 0) {$\scriptstyle{\lambda}$};     
			 \draw[-, thick, darkred] (-0.3, -1.3) to[out=-90, in =180] (0, -1.7) to[out=0, in= -90] (0.3, -1.3);
			 \draw[->, thick, darkred] (0.9, -1.3) -- (0.9, -1.7);
			 \draw[-, thick, darkred] (0.9, -1.7)  -- (0.9, -2);
		\end{tikzpicture}
			}\\
= &\, {\frac 12} \sum_{a+d+c = s- s' -3}
	\mathord{
		\begin{tikzpicture}[baseline = -15pt]
 	  		\draw[->,thick,darkred] (0.7, 0.4) -- node[label={[shift={(0.2,-0.5)}]$\color{darkred}\scriptstyle{c}$}] {$\color{darkred}\bullet$} (0.7,-1.3);
  			\draw[->,thick,darkred] (0.2,0) to[out=90,in=0] (0,0.2);
 			\draw[-,thick,darkred] (0,0.2) to[out=180,in=90] (-.2,0);
			\draw[-,thick,darkred] (-.2,0) to[out=-90,in=180] (0,-0.2);
  			\draw[-,thick,darkred] (0,-0.2) to[out=0,in=-90] (0.2,0);
  			\node[label={[shift={(0.2,-0.5)}]$\color{darkred}\scriptstyle{a}$}] at (0.2,0) {$\color{darkred}\bullet$};  
 			\draw[-,thick,darkred] (0.2,-1) to[out=90,in=0] (0,-0.8);
  			\draw[<-,thick,darkred] (0,-0.8) to[out=180,in=90] (-.2,-1);
			\draw[-,thick,darkred] (-.2,-1) to[out=-90,in=180] (0,-1.2);
 			\draw[-,thick,darkred] (0,-1.2) to[out=0,in=-90] (0.2,-1);
  			\node[label={[shift={(0.2,-0.5)}]$\color{darkred}\scriptstyle{d}$}] at (0.2,-1) {$\color{darkred}\bullet$};
  			\node at (1, 0.2) {$\scriptstyle{\lambda}$} ;
		\end{tikzpicture}
			}
=\; 
	\begin{cases}
		\mathord{
		\begin{tikzpicture}[baseline=0pt]
		\draw[->, thick, darkred] (0, 0.5) -- node[label=right:$\color{darkred}		\scriptstyle{s - s'}$] {$\color{darkred}\bullet$} (0,-0.5);
		\end{tikzpicture}
				}, &\text{if } s - s' \ge 0;\\
		0, &\text{if } s -s' <0. 
	\end{cases}
\end{align*}

Therefore $\pi_s \cdot \iota_{s'} = \delta_{s, s'}$. Hence we know $\iota_s \pi_s$ ($0 \le s \le \lambda_{\hf} -3$) are mutually orthogonal idempotents in $\calU^{\jmath}(\calE_{\hf} \calF^{2}_{\hf} 1_{\lambda}, \calE_{\hf} \calF^{2}_{\hf} 1_{\lambda})$. The idempotent 
\[\kappa= 
\id_{\calE_{\hf} \calF^{(2)}_{\hf} 1_{\lambda}} - \sum_{0 \le s \le \lambda_{\hf}-3} \iota_s \pi_s
\]
can then be expressed diagrammatically as in the lemma.
\end{proof}

\begin{rk}
 \label{rem:iCB}
  One can show that $(\calE_{\hf}  \calF^{(2)}_{\hf}  -
  [\lambda_{\hf} - 2]\calF_{\hf})1_{\lambda}$ is indecomposable in $\dot{\frakU}^\jmath$.
  In fact, since $\pi_{t} \kappa=\kappa
  \iota_s=0$, a direct diagrammatic computation implies that 
  the algebra
  $\End(\calE_\hf\calF_\hf^{(2)}1_{\lambda})$ is spanned by $\kappa$ and
  $\iota_{s+t}x^s \pi_{t}$ for $0\leq s\leq s+t\leq \lam_{\hf}-3$.
  In particular, the algebra
  $\kappa\End(\calE_\hf\calF_\hf^{(2)}1_{\lambda})\kappa$ is a rank one $\bfk$-module spanned by $\kappa$.  It is a local ring, hence
  $(\calE_{\hf}  \calF^{(2)}_{\hf}  -
  [\lambda_{\hf} - 2]\calF_{\hf})1_{\lambda}$ is indecomposable. 
  Similar arguments can be used to show that all the individual terms appearing in the sums
 on the right-hand sides of \eqref{eq:jserre2:1}--\eqref{eq:jserre:3} correspond to
 indecomposable 1-morphisms. 
\end{rk}

We denote by $ \calF_\hf^{(2)}\calE_\hf 1_\lambda$ the $1$-morphism in the idempotent completion of $\mathfrak{\tilde U}^\jmath$ (or in $\dot\frakUj$) defined as the image of the idempotent
$
\mathord{
		\begin{tikzpicture}[baseline=0,scale=1]
			\draw[->,thick,darkred] (0.56, -0.3) --node[label={[shift={(0.4,0.1)}]$\color{black}\scriptstyle{\lambda}$}] {} (0.56, 0.4);
			\draw[<-,thick,darkred] (0.28,-.3) -- (-0.28,.4);
			\draw[<-,thick,darkred] (-0.28,-.3) --node[near end] {$\color{darkred}\bullet$} (0.28,.4);
		\end{tikzpicture}
			} \in  \End(\calF^2_\hf \calE_\hf  1_\lambda).
$

\begin{prop}\label{prop:jserre}
Let  $\lambda \in \X_\jmath$ with $\lambda_{\hf} = \langle
\thal_\hf^\vee, \lambda \rangle \ge 2$.  
Then,
\begin{enumerate}\renewcommand{\theenumi}{\alph{enumi}}
 	\item
		we have the following split surjection in the idempotent completion of $\mathfrak{\tilde U}^\jmath$ (and in $\dot\frakUj$),
\[\xymatrix{
 \calF_\hf^{(2)}\calE_\hf 1_\lambda \oplus (\calE_\hf\calF_\hf^{(2)} - [\lambda_\hf-2] \calF_{\hf})1_\lambda   
  \oplus \calF_\hf 1_\lambda^{\oplus [\lambda_\hf]}  \ar@{->>}[r] &\calF_\hf\calE_\hf\calF_\hf 1_\lambda.}
\]	
	\item
		we have the following isomorphism in $\dot\frakUj$
	\[
	 \calF_\hf\calE_\hf\calF_\hf 1_\lambda \cong \calF_\hf^{(2)}\calE_\hf 1_\lambda \oplus (\calE_\hf\calF_\hf^{(2)} - [\lambda_\hf-2] \calF_{\hf})1_\lambda   
	 \oplus \calF_\hf 1_\lambda^{\oplus [\lambda_\hf]}.
	\]
\end{enumerate}
\end{prop}
\begin{proof}
We introduce the following diagrams ($0 \le k \le \lambda_\hf -1$): 
\begin{equation}
\label{diag:BBP}
	B_1 = -\mathord{
		\begin{tikzpicture}[baseline=0,scale=0.5]
			\draw[->,thick,darkred] (-1,1) -- (1, -1);
			\draw[->,thick,darkred] (0, 1) --node[pos=0.125] {$\color{darkred}\bullet$} (-1, -1);
			\draw[->,thick, darkred] (0, -1) --node[label={[shift={(0.5,0.2)}]$\color{black}\scriptstyle{\lambda}$}] {} (1,1);
		\end{tikzpicture}
					},	
\qquad
 B_2 = \mathord{
			\begin{tikzpicture}[baseline=0,scale=0.5]
				\draw[<-,thick,darkred] (-1,1) -- (0, -1);
				\draw[<-,thick,darkred] (1, -1) -- (0, 1);
				\draw[->,thick, darkred] (1, 1) --node[pos=0.125] {$\color{darkred} \bullet$} (-1,-1);
				\node at (1.4, 0.8) {$\scriptstyle\lambda$};
			\end{tikzpicture}
				},
\qquad 
P_k = \mathord{
			\begin{tikzpicture}[baseline=0,scale=0.5]
				\draw[->, thick, darkred] (-1, 1) -- (-1, -1);
				\draw[->, thick, darkred] (0, -1) to[out=90, in=180] (0.5, 0);
				\draw[-, thick, darkred] (0.5, 0) to[out=0, in=90] node[label={[shift={(0.3, -0.05)}]$\scriptstyle  \lambda_{\hf} - 1 - k$}] {$\color{darkred} \bullet$} (1, -1);
				\node at (1.8, 0) {$\scriptstyle \lambda$};
			\end{tikzpicture}
				},
\end{equation}
\begin{equation}
\label{diag:CCI}
C_1= 
		\mathord{
			\begin{tikzpicture}[baseline=0,scale=0.5]
				\draw[->,thick,darkred] (-1,1) -- (0, -1);
				\draw[->,thick,darkred] (1, -1) -- (0, 1);
				\draw[->,thick, darkred] (1, 1) -- (-1,-1);
				\node at (1.4, 0.8) {$\scriptstyle\lambda$};
			\end{tikzpicture}
				},
\qquad 
C_2 = \mathord{
		\begin{tikzpicture}[baseline=-8, scale=0.5]
			\draw[->,thick,darkred] (-1,1) -- (1, -1);
			\draw[<-,thick,darkred] (0, 1) -- (-1, -1);
			\draw[<-,thick, darkred] (0, -1) -- (1,1);
			\node at (1.4, 0.8) {$\scriptstyle\lambda$};
		\end{tikzpicture}
				},
\qquad 
I_k = {\frac 12} \sum_{s+t = k}\,
	\mathord{
		\begin{tikzpicture}[baseline=0]
  			\draw[-,thick,darkred] (.2,0.2) to[out=180,in=90] (0,0);
 			\draw[-,thick,darkred] (0,0) to[out=-90,in=180] (.2,-0.2);
  			\draw[-,thick,darkred] (.2,-0.2) to[out=0,in=-90]  node[label={[shift={(0.2,-0.75)}]$\color{darkred}\scriptstyle{-\lambda_\hf-2+t}$}] {$\color{darkred}\bullet$} (0.4,0);
			 \draw[->,thick,darkred] (0.4,0) to[out=90,in=0] (.2,.2);			
  			 \node at (0.8,0.05) {$\scriptstyle{\lambda}$};
			\draw[-,thick,darkred] (0.4,.6) to[out=-90, in=0] (.2,0.3);
			\draw[->,thick,darkred] (.2,0.3) to[out = -180, in = -90] (0,.6);
   			\node at (0.4,0.45) {$\color{darkred}\bullet$};
   			\node at (0.6,0.45) {$\color{darkred}\scriptstyle{s}$};
			\draw[->,thick,darkred](-.4,.6) to (-.4,-.6);
 		\end{tikzpicture}
			}.
\end{equation}
Following \eqref{eq:I'}--\eqref{eq:P'}, we introduce modifications of $I_{\lambda_{\hf}-1}$ and $P_0$ as follows:
\begin{align}
I_{\lambda_{\hf}-1}'  =   {\frac 12}\sum_{s+t = \lambda_{\hf} -1 }
	\mathord{
		\begin{tikzpicture}[baseline=0]
  			\draw[-,thick,darkred] (.2,0.2) to[out=180,in=90] (0,0);
 			\draw[-,thick,darkred] (0,0) to[out=-90,in=180] (.2,-0.2);
  			\draw[-,thick,darkred] (.2,-0.2) to[out=0,in=-90]  node[label={[shift={(0.2,-0.75)}]$\color{darkred}\scriptstyle{-\lambda_{\hf}-2+t}$}] {$\color{darkred}\bullet$} (0.4,0);
			 \draw[->,thick,darkred] (0.4,0) to[out=90,in=0] (.2,.2);			
  			 \node at (0.8,0.05) {$\scriptstyle{\lambda}$};
			\draw[-,thick,darkred] (0.4,.6) to[out=-90, in=0] (.2,0.3);
			\draw[->,thick,darkred] (.2,0.3) to[out = -180, in = -90] (0,.6);
   			\node at (0.4,0.45) {$\color{darkred}\bullet$};
   			\node at (0.6,0.45) {$\color{darkred}\scriptstyle{s}$};
			\draw[->,thick,darkred](-.4,.6) to (-.4,-.6);
 		\end{tikzpicture}
			}
		-
	\mathord{
		\begin{tikzpicture}[baseline=0, scale=0.4]
			\draw[-, thick, darkred] (-1, 1) to[out=-90, in =180] (-0.5, 0);
			\draw[->, thick, darkred] (-0.5, 0) to[out=0, in=-90] (0,1);
			\draw[->, thick, darkred] (1,1) -- (1,-1);
			\node at (1.6, 1) {$\scriptstyle{\color{black} \lambda}$};
		\end{tikzpicture}
			}
		+ {\frac 12}\sum_{s+t +u  = \lambda_{\hf} -1 }
	\mathord{
		\begin{tikzpicture}[baseline=0]
  			\draw[-,thick,darkred] (.2,0.2) to[out=180,in=90] (0,0);
 			\draw[-,thick,darkred] (0,0) to[out=-90,in=180] (.2,-0.2);
  			\draw[-,thick,darkred] (.2,-0.2) to[out=0,in=-90]  node[label={[shift={(0.2,-0.75)}]$\color{darkred}\scriptstyle{-\lambda_{\hf}-2+t}$}] {$\color{darkred}\bullet$} (0.4,0);
			 \draw[->,thick,darkred] (0.4,0) to[out=90,in=0] (.2,.2);			
  			 \node at (0.8,0.05) {$\scriptstyle{\lambda}$};
			\draw[-,thick,darkred] (0.4,.6) to[out=-90, in=0] (.2,0.3);
			\draw[->,thick,darkred] (.2,0.3) to[out = -180, in = -90] (0,.6);
   			\node at (0.4,0.45) {$\color{darkred}\bullet$};
   			\node at (0.6,0.45) {$\color{darkred}\scriptstyle{s}$};
			\draw[->,thick,darkred](-.4,.6) --node[label={[shift={(-0.2, -0.45)}]$\scriptstyle u$}] {$\bullet$} (-.4,-.6);
 		\end{tikzpicture}
			},\\
P_{0}' 
= \mathord{
		\begin{tikzpicture}[baseline=0,scale=0.4]
			\draw[->, thick, darkred] (-1, 1) -- (-1, -1);
			\draw[->, thick, darkred] (0, -1) to[out=90, in=180] (0.5, 0);
			\draw[-, thick, darkred] (0.5, 0) to[out=0, in=90] node[label={[shift={(0.1, -0.05)}]$\scriptstyle\lambda_{\hf} - 1$}] {$\color{darkred} \bullet$} (1, -1);
			\node at (1.8, 0) {$\scriptstyle \lambda$};
		\end{tikzpicture}
			}
		- 2 \,
		\mathord{
			\begin{tikzpicture}[baseline=0,scale=0.4]
				\draw[<-, thick, darkred] (-1, -1) to[out=90, in=180] (-0.5, 0);
				\draw[-, thick, darkred] (-0.5,0) to[out=0, in=90] (0, -1);
				\draw[->, thick, darkred] (1, 1) --node[label={[shift={(0.3,0)}]$\color{black}\scriptstyle \lambda$}] {} (1, -1);
			\end{tikzpicture}
				}
		+ \sum_{s+t+u= -3} 
	\mathord{
		\begin{tikzpicture}[baseline=-4]
			\draw[<-,thick,darkred] (0.4,-0.6) to[out=90, in=0]  node[label={[shift={(0.2, -0.45)}]$\color{darkred}\scriptstyle{t}$}] {$\color{darkred}\bullet$} (.2,-0.3);
			\draw[-,thick,darkred] (.2,-0.3) to[out = 180, in = 90] (0,-.6);
  			\draw[-,thick,darkred] (.2,0.2) to[out=180,in=90] (0,0);
 			 \draw[->,thick,darkred] (0.4,0) to[out=90,in=0] (.2,.2);
 			\draw[-,thick,darkred] (0,0) to[out=-90,in=180] (.2,-0.2);
  			\draw[-,thick,darkred] (.2,-0.2) to[out=0,in=-90] (0.4,0);			
  			 \node at (0.8,0.05) {$\scriptstyle{\lambda}$};
   			\node at (0,0) {$\color{darkred}\bullet$};
   			\node at (-0.2,0) {$\color{darkred}\scriptstyle{s}$};
			\draw[->,thick,darkred](-.4,.4) --node[label={[shift={(-0.2, -0.35)}]$\scriptstyle u$}] {$\color{darkred}\bullet$} (-.4,-.6);
		\end{tikzpicture}
			}
		+
	\mathord{
		\begin{tikzpicture}[baseline=-4]
			\draw[-,thick,darkred] (.2,0.2) to[out=180,in=90] (0,0);
 			\draw[-,thick,darkred] (0,0) to[out=-90,in=180] (.2,-0.2);
  			\draw[-,thick,darkred] (.2,-0.2) to[out=0,in=-90] node[label={[shift={(0, -0.75)}]$\scriptstyle {-1}$}] {$\color{darkred}\bullet$} (0.4,0);
			 \draw[->,thick,darkred] (0.4,0) to[out=90,in=0] (.2,.2);			
  			 \node at (1.1,0.05) {$\scriptstyle{\lambda}$};
			\draw[->, thick, darkred] (0.8,.4) --node {$\color{darkred}\bullet$} (0.8,-.6);
			\draw[<-,thick,darkred] (1.6,-0.6) to[out=90, in=0]  (1.4,-0.3);
			\draw[-,thick,darkred] (1.4,-0.3) to[out = 180, in = 90] (1.2,-.6);
		\end{tikzpicture}
			}
		- \,
			\mathord{
		\begin{tikzpicture}[baseline=-4]
			\draw[-,thick,darkred] (.2,0.2) to[out=180,in=90] (0,0);
 			\draw[-,thick,darkred] (0,0) to[out=-90,in=180] (.2,-0.2);
  			\draw[-,thick,darkred] (.2,-0.2) to[out=0,in=-90] (0.4,0);
			 \draw[->,thick,darkred] (0.4,0) to[out=90,in=0] (.2,.2);			
  			 \node at (1.1,0.05) {$\scriptstyle{\lambda}$};
			\draw[->, thick, darkred] (0.8,.4) -- (0.8,-.6);
			\draw[<-,thick,darkred] (1.6,-0.6) to[out=90, in=0]  (1.4,-0.3);
			\draw[-,thick,darkred] (1.4,-0.3) to[out = 180, in = 90] (1.2,-.6);
		\end{tikzpicture}
		}.
\end{align}
Recall the $2$-morphism $\kappa \in \End(\calE_\hf \calF^2_\hf 1_\lambda)$ in Lemma~\ref{lem:dec:eff}. We define the $2$-morphisms: 
\begin{align*}
\begin{bmatrix}
B_1 \\ \kappa \cdot B_2 \\ P'_0 \\ P_1 \\ \vdots \\ P_{\lambda_{\hf}-1}
\end{bmatrix}
:
 \calF_\hf\calE_\hf\calF_\hf 1_\lambda \longrightarrow \calF_\hf^{(2)}\calE_\hf 1_\lambda \oplus (\calE_\hf\calF_\hf^{(2)} - [\lambda_\hf-2] \calF_{\hf})1_\lambda   
  \oplus \calF_\hf 1_\lambda^{\oplus [\lambda_\hf]},\\
 \begin{bmatrix}
C_1 & C_2 \cdot \kappa &I_0 & \cdots &  I'_{\lambda_{\hf}-1} 
\end{bmatrix}
:
 \calF_\hf^{(2)}\calE_\hf 1_\lambda \oplus (\calE_\hf\calF_\hf^{(2)} - [\lambda_\hf-2] \calF_{\hf})1_\lambda 
   \oplus \calF_\hf 1_\lambda^{\oplus [\lambda_\hf]}
   \longrightarrow \calF_\hf\calE_\hf\calF_\hf 1_\lambda.
\end{align*}
The proof of the proposition is reduced to verifying the identities
\eqref{eq:prop:jserre1}--\eqref{eq:prop:jserre2} below:
\begin{align}\label{eq:prop:jserre1}
\begin{bmatrix}
B_1 \\ \kappa \cdot B_2 \\ P'_0 \\ P_1 \\ \vdots \\ P_{\lambda_{\hf}-1}
\end{bmatrix}
\cdot
\begin{bmatrix}
C_1 & C_2 \cdot \kappa &I_0 & \cdots & I_{\lambda_{\hf}-2} & I'_{\lambda_{\hf}-1}  
\end{bmatrix}
=\id_{(\lambda_{\hf}+2) \times (\lambda_{\hf}+2)} \qquad  \text{ in } \mathfrak{\tilde U}^\jmath,
\end{align}
and 
\begin{align}\label{eq:prop:jserre2}
\begin{bmatrix}
C_1 & C_2 \cdot \kappa &I_0 & \cdots & I'_{\lambda_{\hf}-1} 
\end{bmatrix}
\cdot
\begin{bmatrix}
B_1 \\ \kappa \cdot B_2 \\ P'_0 \\ \vdots \\ P_{\lambda_{\hf}-1}
\end{bmatrix}
=\id_{\calF_{\hf} \calE_{\hf} \calF_{\hf} 1_{\lambda}} \qquad \text{ in } \mathfrak{\dot U}^\jmath.
\end{align}
The proofs of the identities \eqref{eq:prop:jserre1}--\eqref{eq:prop:jserre2}
via   elementary but lengthy diagrammatic computations are given in the Appendix~\ref{sec:jserre} (see
Proposition~\ref{p:left-inv'} and Proposition~\ref{prop:invR:2}). 
\end{proof}

\subsection{Categorification of $\jmath$-Serre \eqref{eq:jserre2:2} }  
\label{subsec:jserre:2}

In this subsection we shall assume $\lambda \in \X_\jmath$ with $\lambda_{\hf} = \langle \thal_\hf^\vee, \lambda \rangle = 1$, and recall
in this case the expression of the $\jmath$-Serre relation in \eqref{eq:jserre2:2}.
We denote by $ \calF_\hf^{(2)}\calE_\hf 1_\lambda$ the image of the idempotent
$
\mathord{
		\begin{tikzpicture}[baseline=0,scale=1]
			\draw[->,thick,darkred] (0.56, -0.3) --node[label={[shift={(0.4,0.1)}]$\color{black}\scriptstyle{\lambda}$}] {} (0.56, 0.4);
			\draw[<-,thick,darkred] (0.28,-.3) -- (-0.28,.4);
			\draw[<-,thick,darkred] (-0.28,-.3) --node[near end] {$\color{darkred}\bullet$} (0.28,.4);
		\end{tikzpicture}
			} \in  \End(\calF^2_\hf \calE_\hf  1_\lambda)
$
and by $\calE_\hf  \calF_\hf^{(2)}1_\lambda$ the image of the idempotent \,\,
$
\mathord{
		\begin{tikzpicture}[baseline=0,scale=1]
			\draw[->,thick,darkred] (0.56, -0.3) --node[label={[shift={(0.4,0.1)}]$\color{black}\scriptstyle{\lambda}$}] {} (0.56, 0.4);
			\draw[<-,thick,darkred] (0.28,-.3) -- (-0.28,.4);
			\draw[<-,thick,darkred] (-0.28,-.3) --node[near end] {$\color{darkred}\bullet$} (0.28,.4);
		\end{tikzpicture}
			} \in  \End(\calF^2_\hf \calE_\hf  1_\lambda)
$
in the idempotent completion of $\frak{\tilde U}^\jmath$ (or $\dot\frakUj$).

\begin{prop}
Let  $\lambda \in \X_\jmath$ with $\lambda_{\hf} = \langle \thal_\hf^\vee, \lambda \rangle =1$.  Then,
\begin{enumerate}\renewcommand{\theenumi}{\alph{enumi}}
 	\item
		we have the following split surjection in the idempotent completion of $\mathfrak{\tilde U}^\jmath$ (and in $\mathfrak{\dot U}^\jmath$): 
\[\xymatrix{
 \calF_\hf^{(2)}\calE_\hf 1_\lambda \oplus \calE_\hf\calF_\hf^{(2)}1_\lambda   \oplus \calF_\hf 1_\lambda  \oplus \calF_\hf 1_\lambda  \ar@{->>}[r] &\calF_\hf\calE_\hf\calF_\hf 1_\lambda;}
\]	
	\item
		we have the following isomorphism in $\mathfrak{\dot U}^\jmath$:
	\[
	 \calF_\hf\calE_\hf\calF_\hf 1_\lambda \cong  \calF_\hf^{(2)}\calE_\hf 1_\lambda \oplus \calE_\hf\calF_\hf^{(2)}1_\lambda   \oplus \calF_\hf 1_\lambda  \oplus \calF_\hf 1_\lambda.
	\]
\end{enumerate}
\end{prop}

\begin{proof}
In this case, the relation~\eqref{Pi=1} simplifies to the form: 
\begin{align}\label{Pi=1-simpler}
\mathord{
\begin{tikzpicture}[baseline = 0]
	\draw[->,thick,darkred] (0.45,.6) to (0.45,-.6);
	\draw[->,thick,darkred] (-0.45,.6) to (-0.45,-.6);
        \draw[->,thick,darkred] (0,-.6) to (0,0.6);
   \node at (.6,.05) {$\scriptstyle{\lambda}$};
\end{tikzpicture}
}=
\mathord{
\begin{tikzpicture}[baseline = 0]
	\draw[->,thick,darkred] (0.45,.6) to (-0.45,-.6);
	\draw[<-,thick,darkred] (0.45,-.6) to (-0.45,.6);
        \draw[-,thick,darkred] (0,-.6) to[out=90,in=-90] (-.45,0);
        \draw[->,thick,darkred] (-0.45,0) to[out=90,in=-90] (0,0.6);
   \node at (.5,.05) {$\scriptstyle{\lambda}$};
\end{tikzpicture}
}
-\mathord{
\begin{tikzpicture}[baseline = 0]
	\draw[->,thick,darkred] (0.45,.6) to (-0.45,-.6);
	\draw[<-,thick,darkred] (0.45,-.6) to (-0.45,.6);
        \draw[-,thick,darkred] (0,-.6) to[out=90,in=-90] (.45,0);
        \draw[->,thick,darkred] (0.45,0) to[out=90,in=-90] (0,0.6);
   \node at (.6,.05) {$\scriptstyle{\lambda}$};
\end{tikzpicture}
}
-
\mathord{
\begin{tikzpicture}[baseline = 0]
	\draw[->,thick,darkred] (0.45,.6) to (-0.45,-.6);
        \draw[-,thick,darkred] (0,-.6) to[out=90,in=180] (.25,-.2);
        \draw[->,thick,darkred] (0.25,-.2) to[out=0,in=90] (.45,-0.6);
                \draw[-,thick,darkred] (-.45,.6) to[out=-90,in=180] (-.2,0.2);
        \draw[->,thick,darkred] (-.2,0.2) to[out=0,in=-90] (0,0.6);
   \node at (.6,.05) {$\scriptstyle{\lambda}$};
\end{tikzpicture}
}-
\mathord{
\begin{tikzpicture}[baseline = 0]
	\draw[->,thick,darkred] (-0.45,.6) to (0.45,-.6);
        \draw[<-,thick,darkred] (0,.6) to[out=-90,in=180] (.25,0.2);
        \draw[-,thick,darkred] (0.25,0.2) to[out=0,in=-90] (.45,0.6);
          \draw[<-,thick,darkred] (-.45,-.6) to[out=90,in=180] (-.2,-.2);
        \draw[-,thick,darkred] (-.2,-.2) to[out=0,in=90] (0,-0.6);
   \node at (.5,.05) {$\scriptstyle{\lambda}$};
\end{tikzpicture}
}+
\mathord{
\begin{tikzpicture}[baseline=0]
	\draw[<-,thick,darkred] (0.4,-0.6) to[out=90, in=0] (.2,-0.3);
	\draw[-,thick,darkred] (.2,-0.3) to[out = 180, in = 90] (0,-.6);
  \draw[->,thick,darkred](-.4,.6) to (-.4,-.6);
   \node at (0.5,0.05) {$\scriptstyle{\lambda}$};
	\draw[-,thick,darkred] (0.4,.6) to[out=-90, in=0] (.2,0.3);
	\draw[->,thick,darkred] (.2,0.3) to[out = -180, in = -90] (0,.6);
\end{tikzpicture}
}
+ 2 \,\,\,
\mathord{
\begin{tikzpicture}[baseline=0]
	\draw[-,thick,darkred] (0,-0.6) to[out=90, in=0] (-.2,-0.3);
	\draw[->,thick,darkred] (-.2,-0.3) to[out = 180, in = 90] (-0.4,-.6);
  \draw[->,thick,darkred](.4,.6) to (.4,-.6);
   \node at (0.6,0.05) {$\scriptstyle{\lambda}$};
	\draw[<-,thick,darkred] (0,.6) to[out=-90, in=0] (-.2,0.3);
	\draw[-,thick,darkred] (-.2,0.3) to[out = -180, in = -90] (-0.4,.6);
\end{tikzpicture}
}.\end{align}
We introduce the following diagrams
\[
	B_1 = -\mathord{
		\begin{tikzpicture}[baseline=0,scale=0.5]
			\draw[->,thick,darkred] (-1,1) -- (1, -1);
			\draw[->,thick,darkred] (0, 1) --node[pos=0.125] {$\color{darkred}\bullet$} (-1, -1);
			\draw[->,thick, darkred] (0, -1) --node[label={[shift={(0.5,0.2)}]$\color{black}\scriptstyle{\lambda}$}] {} (1,1);
		\end{tikzpicture}
					},	
\qquad
 B_2 = \mathord{
			\begin{tikzpicture}[baseline=0,scale=0.5]
				\draw[<-,thick,darkred] (-1,1) -- (0, -1);
				\draw[<-,thick,darkred] (1, -1) -- (0, 1);
				\draw[->,thick, darkred] (1, 1) --node[pos=0.125] {$\color{darkred} \bullet$} (-1,-1);
				\node at (1.4, 0.8) {$\scriptstyle\lambda$};
			\end{tikzpicture}
				},
\qquad 
P_0 = \mathord{
			\begin{tikzpicture}[baseline=0,scale=0.5]
				\draw[->, thick, darkred] (-1, 1) -- (-1, -1);
				\draw[->, thick, darkred] (0, -1) to[out=90, in=180] (0.5, 0);
				\draw[-, thick, darkred] (0.5, 0) to[out=0, in=90] (1, -1);
				\node at (1, 1) {$\scriptstyle \lambda$};
			\end{tikzpicture}
				}
			-
		 \mathord{
			\begin{tikzpicture}[baseline=0,scale=0.5]
				\draw[->, thick, darkred] (2, 1) -- (2, -1);
				\draw[<-, thick, darkred] (0, -1) to[out=90, in=180] (0.5, 0);
				\draw[-, thick, darkred] (0.5, 0) to[out=0, in=90] (1, -1);
				\node at (2.5, 1) {$\scriptstyle \lambda$};
			\end{tikzpicture}
				},				
	\qquad 
P_1=  \mathord{
			\begin{tikzpicture}[baseline=0,scale=0.5]
				\draw[->, thick, darkred] (2, 1) -- (2, -1);
				\draw[<-, thick, darkred] (0, -1) to[out=90, in=180] (0.5, 0);
				\draw[-, thick, darkred] (0.5, 0) to[out=0, in=90] (1, -1);
				\node at (2.5, 1) {$\scriptstyle \lambda$};
			\end{tikzpicture}
				},	
\]
\\
\[
C_1= 
		\mathord{
			\begin{tikzpicture}[baseline=0,scale=0.5]
				\draw[->,thick,darkred] (-1,1) -- (0, -1);
				\draw[->,thick,darkred] (1, -1) -- (0, 1);
				\draw[->,thick, darkred] (1, 1) -- (-1,-1);
				\node at (1.4, 0.8) {$\scriptstyle\lambda$};
			\end{tikzpicture}
				},
\qquad 
C_2 = \mathord{
		\begin{tikzpicture}[baseline=0, scale=0.5]
			\draw[->,thick,darkred] (-1,1) -- (1, -1);
			\draw[<-,thick,darkred] (0, 1) -- (-1, -1);
			\draw[<-,thick, darkred] (0, -1) -- (1,1);
			\node at (1.4, 0.8) {$\scriptstyle\lambda$};
		\end{tikzpicture}
				},
\qquad 
I_0 =  \mathord{
			\begin{tikzpicture}[baseline=0,scale=0.5]
				\draw[->, thick, darkred] (2, 1) -- (2, -1);
				\draw[<-, thick, darkred] (0, 1) to[out=-90, in=180] (0.5, 0);
				\draw[-, thick, darkred] (0.5, 0) to[out=0, in=-90] (1, 1);
				\node at (2.5, 1) {$\scriptstyle \lambda$};
			\end{tikzpicture}
				}
				- \,\,\,
				 \mathord{
			\begin{tikzpicture}[baseline=0,scale=0.5]
				\draw[->, thick, darkred] (-1, 1) -- (-1, -1);
				\draw[->, thick, darkred] (0, 1) to[out=-90, in=180] (0.5, 0);
				\draw[-, thick, darkred] (0.5, 0) to[out=0, in=-90] (1, 1);
				\node at (1.3, 1) {$\scriptstyle \lambda$};
			\end{tikzpicture}
				},
\qquad
 I_1 = \mathord{
			\begin{tikzpicture}[baseline=0,scale=0.5]
				\draw[->, thick, darkred] (2, 1) -- (2, -1);
				\draw[->, thick, darkred] (0, 1) to[out=-90, in=180] (0.5, 0);
				\draw[-, thick, darkred] (0.5, 0) to[out=0, in=-90] (1, 1);
				\node at (2.5, 1) {$\scriptstyle \lambda$};
			\end{tikzpicture}
				}.
\]
We consider the following $2$-morphisms: 
\begin{align*}
\begin{bmatrix}
B_1 \\   B_2 \\ P_0  \\ P_1
\end{bmatrix}
:
 \calF_\hf\calE_\hf\calF_\hf 1_\lambda \longrightarrow \calF_\hf^{(2)}\calE_\hf 1_\lambda \oplus (\calE_\hf\calF_\hf^{(2)} - [\lambda_\hf-2] \calF_{\hf})1_\lambda 
   \oplus \calF_\hf 1_\lambda^{\oplus [\lambda_\hf]},\\
 \begin{bmatrix}
C_1 & C_2  & I_0 & I_{1} 
\end{bmatrix}
:
 \calF_\hf^{(2)}\calE_\hf 1_\lambda \oplus (\calE_\hf\calF_\hf^{(2)} - [\lambda_\hf-2] \calF_{\hf})1_\lambda   \oplus \calF_\hf 1_\lambda^{\oplus [\lambda_\hf]}
 \longrightarrow \calF_\hf\calE_\hf\calF_\hf 1_\lambda.
\end{align*}

The proof of the proposition is reduced to verifying the two  identities 
\begin{align}\label{eq:prop:jserre3}
\begin{bmatrix}
B_1 \\  B_2 \\ P_0 \\  P_{1}
\end{bmatrix}
\cdot
\begin{bmatrix}
C_1 & C_2  &I_0 & I_1 
\end{bmatrix}
=\id_{4 \times 4} \qquad  \text{ in } \mathfrak{\widetilde U}^\jmath,
\end{align}
and 
\begin{align}\label{eq:prop:jserre4}
\begin{bmatrix}
C_1 & C_2 &I_0 & I_{1} 
\end{bmatrix}
\cdot
\begin{bmatrix}
B_1 \\ B_2 \\ P_0  \\ P_{1}
\end{bmatrix}
=\id_{\calF_{\hf} \calE_{\hf} \calF_{\hf} 1_{\lambda}} \qquad \text{ in } \mathfrak{\dot U}^\jmath.
\end{align}

The identities \eqref{eq:prop:jserre3} and \eqref{eq:prop:jserre4} admit
proofs which are analogous (and much easier) to those of the identities
\eqref{eq:prop:jserre1} and \eqref{eq:prop:jserre2} given in the
Appendix~\ref{sec:jserre}. We will leave the details of these
computations to the reader.
\end{proof}

%
%
\subsection{Categorification of  $\jmath$-Serre \eqref{eq:jserre2:3}--\eqref{eq:jserre:3}} 

In this subsection, we complete the categorification of the $\jmath$-Serre relations \eqref{eq:jserre1} and  \eqref{eq:jserre2} in the remaining cases  \eqref{eq:jserre2:3}--\eqref{eq:jserre:3} by using results from Sections~\ref{subsec:jserre:1} and \ref{subsec:jserre:2} and the equivalences \eqref{eq:def:omegajmath} and \eqref{eq:def:taujmath}. 

Indeed the categorification of the $\jmath$-Serre relation \eqref{eq:jserre2:1} implies the categorification of the $\jmath$-Serre relation \eqref{eq:jserre2:3} under the equivalence \eqref{eq:def:taujmath}. Then entirely similarly the categorification of the $\jmath$-Serre relations \eqref{eq:jserre2:1}--\eqref{eq:jserre2:3} implies the categorification of the $\jmath$-Serre relations \eqref{eq:jserre:1} -- \eqref{eq:jserre:3} under the equivalence \eqref{eq:def:omegajmath}. We shall only prove the categorification of the $\jmath$-Serre relation \eqref{eq:jserre2:3} here and leave the other cases to the reader. 


\begin{prop}\label{prop:jserre:3}
\renewcommand{\theenumi}{\alph{enumi}}
Let $\lambda \in X_\jmath$ with $\lambda_{\hf} = \langle \thal_\hf^\vee, \lambda \rangle \le 0$.  
\begin{enumerate}
\item
 We have the following isomorphism of $1$-morphisms in $\dot\frakUj$:
\begin{equation*}
	\calF^{(2)}_{\hf}  \calE_{\hf}  1_{\lambda}  \cong ( \calF^{(2)}_{\hf} \calE_{\hf}   - [-\lambda_{\hf} ]\calF_{\hf})1_{\lambda} \oplus \calF_{\hf} 1_{\lambda}^{\oplus [-\lambda_{\hf}]},
\end{equation*}
where $( \calF^{(2)}_{\hf} \calE_{\hf}   - [-\lambda_{\hf} ]\calF_{\hf})1_{\lambda} $ is the image of the idempotent $\kappa$ in Lemma~\ref{lem:dec:eff} under the equivalence $\sigma_\jmath$ in \eqref{eq:def:taujmath}. 

 	\item
		We have the following split surjection in the idempotent completion of $\mathfrak{\tilde U}^\jmath$ (and in $\dot\frakUj$): 
\[
\xymatrix{
 \calE_\hf  \calF_\hf^{(2)}1_\lambda \oplus (\calF_\hf^{(2)} \calE_\hf - [-\lambda_\hf] \calF_{\hf})1_\lambda   
	 \oplus \calF_\hf 1_\lambda^{\oplus [-\lambda_\hf+2]}    \ar@{->>}[r] &\calF_\hf\calE_\hf\calF_\hf 1_\lambda.}
\]	
	\item
		We have the following isomorphism in $\dot\frakUj$:
	\[
	 \calF_\hf\calE_\hf\calF_\hf 1_\lambda \cong \calE_\hf  \calF_\hf^{(2)}1_\lambda \oplus (\calF_\hf^{(2)} \calE_\hf - [-\lambda_\hf] \calF_{\hf})1_\lambda   
	 \oplus \calF_\hf 1_\lambda^{\oplus [-\lambda_\hf+2]}.
	\]
\end{enumerate}
\end{prop}

\begin{proof}
Let $\lambda'= -\lambda  +\alpha_\hf -\varpi$. We have $\lambda'_\hf=
- \lambda_\hf + 2 \geq 2$ by assumption. Thus, the desired maps are just the image under  the equivalence
$\sigma_\jmath$ of the split surjection and the isomorphism in
Lemma~\ref{lem:dec:eff} and Propositions~\ref{prop:jserre}--\ref{prop:jserre:3} applied to
$\calF_\hf\calE_\hf\calF_\hf 1_{\lambda'}$. 
\end{proof}

\subsection{The Grothendieck group}\hfill

\begin{lemma}\label{lem: gp}
The assignment $\lambda\mapsto\lambda$, $\calE^{(a)}_i 1_\lambda \mapsto [\calE^{(a)}_i 1_\lambda]$, 
$\calF^{(a)}_i1_\lambda \mapsto [\calF^{(a)}_i1_\lambda]$ for all $\lambda\in\X_\jmath$, $i\in\bbIj$, $a\in\bbN$
defines an $\calA$-linear functor $\aleph: {}_\calA\bdUj\to K_0(\dot\frakUj)$. 
Furthermore, we have $\aleph\circ\psi_\jmath =\psi_\jmath\circ\aleph$.
\end{lemma}

\begin{proof}
It suffices to prove that the assignment in the lemma defines a $\bbQ(q)$-linear functor $\bdUj\to\bbQ(q)\otimes_{\calA}K_0(\dot\frakUj)$. To this end,
we must check that all defining relations \eqref{eq:cartano}--\eqref{eq:jserre2} in $\bdUj$ are satisfied in $K_0(\dot\frakUj)$.
For \eqref{eq:cartano}--\eqref{eq:serre} this follows from the same argument as in  \cite[Proposition~3.27]{KLIII}.
The $\jmath$-Serre relations \eqref{eq:jserre1}--\eqref{eq:jserre2} are proven in Proposition~ \ref{prop:jserre}.
\end{proof}

Our goal in this subsection is to show that the functor $\aleph:
{}_\calA\bdUj\to K_0(\dot\frakUj)$ is full (see
Proposition~\ref{prop:gammasuj}). We achieve this by showing that the graded category associated with a filtration on $\frakU^\jmath$ is equivalent to the KLR categorification of the positive half of ${}_\calA\dot{\bfU}$. Let us explain this in details.

Fix $\lambda\in\X_\jmath$. Recall from \eqref{eq:I} and \eqref{eq:Ij} that $\bbI=\bbIj\cup-\bbIj$, and that we write 
$\calE_{-i}:=\calF_i$ for $i\in\bbI^\jmath$.
Given a sequence $\bfi=(i_1,i_2,\dots, i_m)\in \bbI^m$ we write
$\calE_\bfi= \calE_{i_1} \calE_{i_2}\cdots \calE_{i_m}$. 
Consider the category 
$$
\calC=\bigoplus_{\mu\in\X_\jmath}\dot\frakUj(\lambda,\mu).
$$
We will always view $\calE_{\bfi}$ as the object $\calE_{\bfi}1_\lambda$ in $\calC$, hence omit $1_\lambda$ from the notation.
Since $\aleph$ is the identity on objects, it is enough to prove that $K_0(\calC)$ is in the image of $\aleph$.
The category $\calC$ is Krull-Schmidt, hence its
Grothendieck group is a free $\calA$-module generated by indecomposable objects. 
Up to a grading shift, any indecomposable object $P\in\calC$ is a direct summand of $\calE_\bfi$ for some sequence $\bfi$. 
We define the {\em width} of $P$ to be the minimum length of such a
sequence $\bfi$.  We define the {\em width} of a morphism
$\beta\colon u\to v$ in $\calC$ to be the
minimal $m$ such that $\beta$ factors through a sum of objects of
width $\leq m$.   
Let $\calC_{\leqslant m}$ (respectively, $\calC_{< m}$) be the full subcategory of $\calC$ generated by indecomposable objects of width 
$\leqslant m$ (respectively of width $<m$). The quotient category $\gr^m\calC=\calC_{\leqslant m}/\calC_{<m}$
is the additive category with the same objects as in $\calC_{\leqslant
  m}$, and $\Hom$-space given by the quotient of the $\Hom$-space in 
$\calC_{\leqslant m}$ by the 2-sided ideal of morphisms of width $<m$. The indecomposable objects in $\gr^m\calC$ are in bijection with those in $\calC$ of length $m$. Hence we have $K_0(\calC)=K_0(\gr\calC)$,
where $\gr\calC=\bigoplus_{m\geqslant 0}\gr^m\calC$.

Let $\calC^+$ be the monoidal category (under induction) of graded projective
modules over the quiver Hecke algebra of type $A_{2r}$ (with $\bbI$ identified with the
Dynkin diagram of this type).  This is a graded $\bfk$-linear monoidal category with objects generated by $E_i$, $i\in\bbI$, and the morphisms generated by 
\begin{align*}
x 
= 
\mathord{
\begin{tikzpicture}[baseline = 0]
	\draw[->,thick,darkred] (0.08,-.3) to (0.08,.4);
      \node at (0.08,0.05) {$\color{darkred}\bullet$};
   \node at (0.08,-.4) {$\scriptstyle{i}$};
\end{tikzpicture}
}
: E_i\to E_i\{ -2\}, 
\qquad
\tau
= 
\mathord{
\begin{tikzpicture}[baseline = 0]
	\draw[->,thick,darkred] (0.28,-.3) to (-0.28,.4);
	\draw[->,thick,darkred] (-0.28,-.3) to (0.28,.4);
   \node at (-0.28,-.4) {$\scriptstyle{i}$};
   \node at (0.28,-.4) {$\scriptstyle{j}$};
\end{tikzpicture}
}: E_iE_j\to E_jE_i\{\langle\alpha^\vee_i,\alpha_j\rangle\}
\qquad
\end{align*}
subject to quiver Hecke relations \eqref{qha}-\eqref{qhalast}.   By \cite[Proposition~3.18]{KLI}, $K_0(\calC^+)$ is isomorphic to ${}_\calA\bfU^+$, the positive part of ${}_\calA\dot\bfU$ generated by $E_i^{(a)}$, $i\in\bbI$. 
Let $\Pi_{\lambda}$ be the polynomial ring in the commuting variables (for $i \in \mathbb{I}^\jmath$)
\begin{align*}
\mathord{
\begin{tikzpicture}[baseline = 0]
  \draw[<-,thick,darkred] (0,0.4) to[out=180,in=90] (-.2,0.2);
  \draw[-,thick,darkred] (0.2,0.2) to[out=90,in=0] (0,.4);
 \draw[-,thick,darkred] (-.2,0.2) to[out=-90,in=180] (0,0);
  \draw[-,thick,darkred] (0,0) to[out=0,in=-90] (0.2,0.2);
 \node at (0,-.1) {$\scriptstyle{i}$};
   \node at (0.3,0.2) {$\scriptstyle{\lambda}$};
   \node at (-0.2,0.2) {$\color{darkred}\bullet$};
   \node at (-0.4,0.2) {$\color{darkred}\scriptstyle{s}$};
\end{tikzpicture}
} &: 1_{\lambda} \rightarrow 1_{\lambda} \{-2(s+ 1-  \langle \thal^\vee_i,\lambda\rangle) \}, \text{ for $s \ge \langle \thal^\vee_i,\lambda\rangle-1$},
\end{align*}
Then $K_0(\Pi_\lambda)=\calA$ and we have $K_0(\calC^+ \otimes \Pi_{\lambda})\simeq K_0( \calC^+)$; see \cite[Proposition~3.35]{KLIII} for more details.
\begin{lemma}
  There is a full functor
  $\zeta: \calC^+ \otimes \Pi_{\lambda}\to \gr\calC$ which sends
  $E_\bfi$ to $\calE_\bfi$, for all sequences $\bfi$ and a diagram $D$
  to $(-1)^d$ times that diagram with the
  orientation on the $i$-th strands reversed for all $i<0$. 
  Here $d$ is the number of dots on strands whose orientation was reversed.
\end{lemma}
\begin{proof}
Let us first show that $\zeta$ is well defined. It is enough to show that
the assignment in the lemma defines an algebra homomorphism from $R_m=\End_{\calC^+\otimes\Pi_\lambda}(\bigoplus_{|\bfi|=m}E_\bfi1_\lambda)$ to 
$S_m=\End_{\gr^m\calC}(\bigoplus_{|\bfi|=m}\calE_\bfi1_\lambda)$.
In other words $\zeta(x)$ and $\zeta(\tau)$ satisfy the quiver Hecke relations for all $i,j\in\bbI$. This is obvious if $i$ and $j$ have the same sign.
Next, note that indecomposable summands of $\calE_m=\bigoplus_{|\bfi|=m}\calE_\bfi$ are in bijection with primitive idempotents in $A=\End_{\calC}(\calE_m)$.
By definition $S_m$ is the quotient of $A$ by the $2$-sided
ideal $I$ of morphisms of width $<m$.
Now, relations \eqref{j-bicross down},
\eqref{j-bicross up}, \eqref{Pi=1} in $\frakU^\jmath$ modulo morphisms of
width $<2$ give the quiver Hecke relations for $i,j$ of different sign. We deduce that 
$\zeta: R_m\to A/I=S_m$ is well defined.

It remains to show that $\zeta: R_m\to S_m$ is surjective.  
To this end, note that $\End_{\calC}(\bigoplus_{|\bfi|=m}\calE_\bfi1_\lambda)$ is spanned by diagrams where no pair of strands cross twice and all bubbles are at far left.  Every such diagram either has no cups or caps (except in bubbles) or has width $<m$.  Thus  $S_m$ is spanned over $\Pi_\lam$ by
diagrams without cups or caps. 
The image of $\zeta$ contains all
diagrams of this type, so the surjectivity follows.  
\end{proof}

\begin{prop}\label{prop:gammasuj}
The functor $\aleph: {}_\calA\bdUj\to K_0(\dot\frakUj)$ is full, \emph{i.e.}, for any $\lambda,\mu\in\X_\jmath$, it defines a surjective map 
$${}_\calA\bdUj(\lambda,\mu)\twoheadrightarrow K_0(\dot\frakUj(\lambda,\mu)).$$
\end{prop}
It will be proved in Theorem \ref{thm:maindecat} that $\aleph: {}_\calA\bdUj\to K_0(\dot\frakUj)$ is in fact an equivalence.

\begin{proof}
The same argument as in \cite[Section~3.8.3]{KLIII} implies that
the map
\begin{equation}
  K_0(\zeta): {}_\calA\bfU^+\simeq K_0(\calC^+ \otimes \Pi_\lambda)  \longrightarrow  K_0(\gr\calC).
\end{equation}
is surjective.

By \cite[7.8]{Lu90}, the algebra ${}_\calA\bfU^+$ has an $\calA$-basis $\bfM$ whose elements are monomials in divided powers $E_i^{(a)}$.
Hence $K_0(\zeta)$ maps them to a spanning set of the $\calA$-module $K_0(\gr\calC)$ which are monomials in $\calE_i^{(a)}$ for $i\in \bbI$.
By the isomorphism $K_0(\calC)=K_0(\gr\calC)$, we deduce that $K_0(\calC)$ is also spanned by these monomials, which are contained in the image of $\aleph$. Hence $\aleph$ is full.
\end{proof}

Let $\bfM$ be a monomial basis of ${}_{\calA}\bfU^+$ in the sense of \cite[7.8]{Lu90}. Let 
$$
{}_{\calA}\bdUj(\lambda,-)=\bigoplus_{\mu\in\X_\jmath}{}_{\calA}\bdUj(\lambda,\mu). 
$$
For $m=E_{i_1}^{(c_1)}E_{i_2}^{(c_2)} \cdots$, we write $m^\jmath=\calE_{i_1}^{(c_1)}\calE_{i_2}^{(c_2)} \cdots$ accordingly.
The corollary below was known only for some special choice of the monomial bases (see the proof of \cite[Theorem~4.7]{BKLW}). 

\begin{cor}\label{lem:U+Uj2}
For any $\lambda\in\X_\jmath$,
the assignment $m\mapsto m^\jmath$ for all $m\in\bfM$ yields an isomorphism of $\calA$-modules 
$$\varsigma: {}_{\calA}\dot\bfU^+ \simto {}_{\calA}\bdUj(\lambda,-).$$
\end{cor}
\begin{proof}
The proof of \cite[Proposition~6.2]{Ko14} adapted to the idempotented algebra case implies that $\varsigma$ is an isomorphism 
after base change to $\bbQ(q)$. Since ${}_{\calA}\dot\bfU^+$ is free over $\calA$, we get injectivity of $\varsigma$. 
The surjectivity follows from the proof of Proposition~\ref{prop:gammasuj}.
\end{proof}

\subsection
{Control from the Grothendieck group}

\begin{prop}\label{df-prop}
Let $\frakC$ be any idempotent complete $2$-category. 
A $2$-functor $\phi: \widetilde{\frakU}^\jmath\to\frakC$ induces a $2$-functor $\dot\frakUj\to\frakC$ if and only if the classes of $\phi(\calE_\hf1_\lambda)$ and $\phi(\calF_\hf1_\lambda)$ satisfy the $\jmath$-Serre relations \eqref{eq:jserre1}--\eqref{eq:jserre2} in $K_0(\frakC)$, for any $\lambda \in \X_\jmath$.
\end{prop}
\begin{proof}
Let $\Xi$ denote the right hand side of \eqref{Pi=1}.
Since $\frakC$ is idempotent complete, the $2$-functor $\phi$ factors through $\dot\frakUj$ if and only if 
$\phi(\Xi)$ is identity for all $\lambda$.
We may assume $\lambda_\hf\geqslant 2$, as the proof for the other cases are similar. By Proposition \ref{prop:jserre}, we have a split surjective map
$$
\Gamma: \calF_\hf\calE_\hf\calF_\hf1_\lambda
\longrightarrow 
\calF^{(2)}_\hf\calE_\hf\oplus \left(\calE_\hf\calF_\hf^{(2)}
-(\lambda_\hf-2)\calF_\hf\right)1_\lambda\oplus \calF_\hf 1_\lambda^{\oplus [\lambda_\hf]}
$$
in the idempotent completion of $\widetilde{\frakU}^\jmath$ given by $\Gamma= B_1\oplus (\kappa\cdot B_2)\oplus P_0'\oplus...\oplus P_{\lambda_\hf-1}$. 
A splitting is provided by $\Gamma'=C_1+(C_2\cdot\kappa)+I_0+...+I_{\lambda_\hf-1}$, that is,
$\Gamma \circ\Gamma' =\text{id}$; see the proof of Proposition \ref{prop:jserre} for the notation. 
Furthermore, we have $\Xi=\Gamma'\circ\Gamma$ by Proposition~\ref{prop:invR:2}. It follows that $\Xi$ is an idempotent in $\widetilde{\frakU}^\jmath$.
Hence $\phi(\Xi)=1$ if and only if $\phi(\Gamma)$ is an isomorphism, if and only if the $\jmath$-Serre relation holds
in $K_0(\frakC)$. 
\end{proof}

\section{The Schur $2$-category}
\label{sec:Schur}


%
%
\subsection{Soergel bimodules}\label{ss:isoflag}


We start by reviewing some standard facts about singular Soergel bimodules from \cite{WilSSB} that will be used below.

Let $\W$ be a Weyl group with the set of simple reflections $\S$, and let $\frakt$ be a faithful $\bfk$-linear reflection representation of $\W$. Consider the graded ring of symmetric product $R=S(\frakt)$ with elements in $\frakt$ in degree $2$. Then the $\W$-action on $\frakt$ induces a homogeneous action on $R$.

Given $I\subset \S$, let $\W_I$ be the parabolic subgroup generated by $I$. We write $R^I=R^{\W_I}$ for the invariants in $R$ under $\W_I$. 
Given a triple $I\supset K\subset J$, we view $R^K$ as an $(R^I, R^J)$-bimodule via the canonical embeddings $R^I\hookrightarrow R^K\hookleftarrow R^J$.

Let $\mathfrak{B}im_\W$ denote the $2$-category with objects consisting of subsets in $\S$, and $\frakB{im}_\W(J,I)=R^I\bgmod R^J$, 
with composition given by tensor product of bimodules.
The category of \emph{singular Soergel bimodules} $\mathbf{S}\frakB{im}_\W$ is the full $2$-subcategory of $\frakB{im}_\W$ 
generated by the same objects and $1$-morphisms $R^I\hookrightarrow R^K\hookleftarrow R^J$ for all triples $I\supset K\subset J$. 
In particular, indecomposable objects in $\mathbf{S}\frakB{im}_\W(J,I)$ are direct summands of 
$$
R^{I_1}\otimes_{R^{J_1}}R^{I_2}\otimes_{R^{J_2}}...\otimes_{R^{J_{n-1}}}R^{I_n}
$$
for certain sequence $I=I_1\subset J_1\supset I_2\subset J_2\supset...\subset J_{n-1}\supset I_n=J$.
Let $\psi: \mathbf{S}\frakB{im}_\W\to \mathbf{S}\frakB{im}_\W$ be the duality given by taking graded duals.
In particular $\psi\circ\{1\}=\{-1\}\circ\psi$. 

Consider the \emph{Schur algebroid} $\bfS_\W$ associated with $\W$. This is an $\calA$-linear category with objects finitary $I\subset S$, 
and morphisms from $I$ to $J$ given by intersections of parabolic modules over the Hecke algebra $\bfH_\W$, 
see \cite[Definition~2.7]{WilSSB}. Note that $\bfS_\W(\emptyset,\emptyset)=\bfH_\W$.

\renewcommand{\theenumi}{\alph{enumi}}
\begin{thm}\cite{Soe07, WilSSB}\label{thm:SBim}
\mbox{}
  \begin{enumerate}
  \item 
 There is an equivalence of categories $$\kappa:
K_0(\bfS\frakB{im}_\W)\simto \bfS_\W$$ which is the identity on objects, and is given by the character map on morphisms.
  \item 
 The equivalence $\kappa$ intertwines $\psi$ with the bar
involution on $\bfS_\W$.
  \item 
 For each element $w$ in $\W_I\backslash\W/\W_J$, there is a unique
self-dual indecomposable object $\calB_w$ in $\bfS\frakB{im}_\W(I,J)$
characterized by a support condition. These objects form a complete
and irredundant set of indecomposable objects in $\bfS\frakB{im}_\W(I,J)$ up to
isomorphism and grading shift.
  \item 
 If the residue field of $\bfk$ has characteristic zero, then
$\kappa$ sends $[\calB_w]$ to the element in the canonical basis of
$\bfS_\W(I,J)$ indexed by $w$.
\end{enumerate}
\end{thm}

We have a geometric interpretation of $\bfS\frakB{im}_W$ as follows. 
Let $\G$ be a connected reductive group with Weyl group $\W$.
Fix a Borel subgroup $\B$ and a maximal torus $\T$. Take $\frakt$ to be the $\bfk$-module spanned by the characters of $\T$. 
For $I\subset \S$ let $\B\subset\P_I\subset\G$ be the parabolic subgroup corresponding to $I$. 
The diagonal $\G$-orbits in $\G/\P_I\times\G/\P_J$ are parametrized by $\W_I\backslash\W/\W_J$.
Then $\bfS_\W(I,J)$ is the generic algebra of $\calA$-valued functions on
$\G\backslash(\G/\P_I\times\G/\P_J)$  
arising from the convolution product of functions on
$\P_I(\bbF_{\bfq^{2}})\backslash\G(\bbF_{\bfq^{2}})/\P_J(\bbF_{\bfq^{2}})$ over finite fields $\bbF_{\bfq^{2}}$.

\begin{prop}
  The category $\bfS\frakB{im}_W(I,J)$ is equivalent to the category
  of $\G$-equivariant parity complexes on $\G/\P_I\times\G/\P_J$, with
  $\calB_w$ corresponding to the unique parity sheaf $\calE_w$ whose
  support is the closure of the orbit $O_w$.
\end{prop}
All the machinery needed
to prove this result is given in \cite[4.1]{JMW} but it is not stated
there.  In type A, this is discussed in greater detail in
\cite[Theorem 6]{Webcomparison}, and the proof is identical in other
types. 
When the residue field of $\bfk$ has characteristic zero, the parity
sheaf $\calE_w$ is an intersection cohomology complex.

\medskip

\subsection{Schur category in type $\B/\C$}
\label{sec:type-b/c}

From now on let $\W=\W_m$ be the Weyl group of type $B_m/C_m$ with
simple reflections $\S=\{s_0,s_1,...,s_{m-1}\}$. Take
$\frakt=\oplus_{i=1}^m\bfk t_i$ with $s_i$ acting by permuting $t_i$
and $t_{i+1}$, and $s_0(t_i)=(-1)^{\delta_{i,1}}t_i$.
We have 
$$
R=S(\frakt)=\bfk[t_1,...,t_m].
$$ 
Let $\G=\SO(\V)$ with $\V=\bbC^{2m+1}$ equipped with a standard nondegenerate symmetric bilinear form.

\smallskip

Recall $\bbI_r^\jmath$ from \eqref{eq:Ij}. 
Let $\Sigma_{r,m}$ be the set of (weakly) increasing maps from $\bbI^\jmath_r\to [0,m]$,
and we write an element $\bfa\in \Sigma_{r,m}$ as an increasing sequence $\bfa=(a_\hf, a_{\hf+1},...,a_{\hf+r-1})$. 
To $\bfa\in\Sigma_{r,m}$ we associate the subset $I_\bfa$ of $\S$ with $s_{a_p}$ removed for all $p$ such that $a_p<m$. 
Note that if $r\geqslant m$, then every subset of $\S$ is of this form.
Given $\bfa\in\Sigma_{r,m}$ and $i\in\bbIj_r$ let
\begin{alignat*}{3}
&\bfa^{+i}=(\ldots, a_{i-1},a_{i}, a_{i}+1, a_{i+1},\ldots), \qquad 
&&\bfa^{-i}=(\ldots, a_{i-1},a_{i}-1, a_{i}, a_{i+1}, \ldots),\\ 
&{}_{+i}\bfa=(\ldots, a_{i-1},a_{i}+1, a_{i+1}, \ldots),\qquad
&&{}_{-i}\bfa=(\ldots, a_{i-1}, a_{i}-1, a_{i+1}, \ldots).
\end{alignat*}
For any sequence $\bfa$ of $r$ integers, we write
$$R^\bfa=\begin{cases} R^{I_\bfa}, &\text{if } \bfa\in\Sigma_{r,m},\\
0,  &\text{if } \bfa\not\in \Sigma_{r,m}.\end{cases}$$ 
Then $R^{\bfa^{\pm i}}$ is naturally an $(R^{{}_{\pm i}\bfa}, R^\bfa)$-bimodule.
Indeed $\bfa^{\pm i}$ is a sequence of $r+1$ integers. If it belongs to $\Sigma_{r+1,m}$, then ${}_{\pm i}\bfa$ also belongs to $\Sigma_{r,m}$, moreover $I_{{}_{\pm i}\bfa}\supset I_{\bfa^{\pm i}}\subset I_\bfa$, hence $R^{\bfa^{\pm i}}$ is an $(R^{{}_{\pm i}\bfa}, R^\bfa)$-bimodule.
If $\bfa^{\pm i}\not\in \Sigma_{r+1,m}$, then $R^{\bfa^{\pm i}}=0$, the statement is trivial.
The partial flag variety $\iGr_\bfa:=\G/\P_{I_\bfa}$ is the variety of isotropic flags
$$0=F_{-r-\hf}\subset ...\subset F_{-\hf}\subset F_{\hf}\subset...\subset F_{\hf+r}=\V$$
such that $F_p=F_{-p}^\perp$, $\dim F_{-p}=m-a_p$ (and so $\dim F_{p}=m +1 +a_p$) for all $p \in \bbI^\jmath$.
\smallskip

Consider the following 1-morphisms in $\bfS\frakB{im}_\W$ for all $i\in\bbI^\jmath_r$
\begin{align*}
1_\bfa:=R^\bfa\in R^\bfa\bgmod R^\bfa,
\\ 
\scrE_i1_\bfa:=R^{\bfa^{+i}}\{1+a_i-a_{i+1}\}\in R^{{}_{+ i}\bfa}\bgmod R^\bfa,
\\
\scrF_i1_\bfa:=R^{\bfa^{-i}}\{1+a_{i-1}-a_{i}\}\in R^{{}_{-i}\bfa}\bgmod R^\bfa.
\end{align*}
Here and below, we use the convention
$$
a_{-\hf}=-a_{\hf},\qquad a_{\hf+r}=m.
$$

We define the {\em Schur $2$-category} $\frakF_{r,m}$ as the locally fully faithful monoidal $2$-subcategory of $\bfS\frakB{im}_\W$ with object set $\Sigma_{r,m}$ and with $1$-morphisms given by direct summands of direct sums of products of $1_\bfa$, $\scrE_i1_\bfa$, $\scrF_i1_\bfa$, for $\bfa\in \Sigma_{r,m}$, $i\in\bbI^\jmath_r$.
Recall that the locally fully faithfulness here means that we impose for any $1$-morphisms $M$, $N$ in $\frakF_{r,m}$ that $\Hom_{\frakF_{r,m}}(M,N)=\Hom_{\bfS\frakB{im}_\W}(M,N)$.

Recall that the {\em Schur category} $\bfS^\jmath_{r,m}$ (also called the {\em $\jmath$Schur
algebra}) is the category with object set
$\Sigma_{r,m}$ and
$\bfS^\jmath_{r,m}(\bfa,\bfb):=\bfS_W(I_\bfa,I_\bfb)$, for $\bfa, \bfb \in\Sigma_{r,m}$.  Similarly,
we have a fully faithful functor 
\begin{math}
 \bfS^\jmath_{r,m}\rightarrow  \bfS_W.
\end{math}
For $\bfa\in\Sigma_{r,m}$, let $\lambda(\bfa) \in \X_\jmath$ be
such that  
\begin{equation}\label{eq:alam}
\langle\thal_i,\lambda(\bfa)\rangle=-a_{i-1}+2a_i-a_{i+1},
\qquad \forall i\in\bbI^\jmath_r. 
\end{equation}
In this way $\Sigma_{r,m}$ can be viewed as a subset of $\X_\jmath$. It is easy to see that $\lambda(\bfa) = a_{\hf} \varepsilon_0 + \sum_{i = \hf +1}^{r+\hf} (a_{i} - a_{i-1}) \varepsilon_{i -\hf} \in X_\jmath$.

By \cite[Proposition~ 3.1, Corollary~3.13]{BKLW}, there is a well-defined functor
\begin{equation}
 \label{eq:ga}
\gamma: \dot\bfU^\jmath\longrightarrow \bfS^\jmath_{r,m}
\end{equation}
such that for $\lambda\in\X_\jmath$ we have $\gamma(\lambda)=\bfa$ if there exists $\bfa\in\Sigma_{r,m}$ (which must be unique by \eqref{eq:alam}) such that $\lambda=\lambda(\bfa)$, and $\gamma(\lambda)=0$ otherwise. 
The images $\bfe_i=\gamma(\calE_i)$, $\bff_i=\gamma(\calF_i)$, for $i\in \bbI^\jmath_r$, generate $\bfS^\jmath_{r,m}\otimes_{\calA}\bbQ(q)$.

 \smallskip

\begin{prop}\label{prop:iFl-Grothendieck}
\mbox{}
\begin{enumerate}
\item For $\bfa$, $\bfb\in\Sigma_{r,m}$, we have $\frakF_{r,m}(\bfa,\bfb)=\bfS\frakB{im}_\W(\bfa,\bfb)$.
\item  There is an equivalence
  $\kappa: K_0(\frakF_{r,m})\simto \bfS^\jmath_{r,m}$ such that
$$
\kappa(\bfa)= \bfa, \quad\kappa([\scrE_i1_\bfa])=\bfe_i1_\bfa,\quad \kappa([\scrF_i1_\bfa])=\bff_i1_\bfa.
$$ 
It intertwines the duality $\psi$ on $\frakF_{r,m}$ and the bar involution.
\item  If the residue field of $\bfk$ has characteristic zero, then
$\kappa$ sends self-dual indecomposable objects in $\frakF_{r,m}$ to
canonical basis in $\bfS^\jmath_{r,m}$.
\end{enumerate}
\end{prop}

\begin{proof}
The nontrivial statement here is part (a). The rest follows automatically from Theorem~ \ref{thm:SBim} and (a). To prove (a), let $\frakF'_{r,m}$ be the $2$-subcategory of $\bfS\frakB{im}_\W$ with object set $\Sigma_{r,m}$ and such that $\frakF'_{r,m}(\bfa,\bfb)=\bfS\frakB{im}_\W(\bfa,\bfb)$ for any $\bfa$, $\bfb\in \Sigma_{r,m}$. 
By definition, we have a locally fully faithful embedding $\iota: \frakF_{r,m}\to \frakF'_{r,m}$, which 
sends indecomposable $1$-morphisms to indecomposable ones. 
We must show that $\iota$ is full, that is all indecomposable $1$-morphisms in $\frakF'_{r,m}$ are in the image. This is true if and only if the fully faithful functor $[\iota]: K_0(\frakF_{r,m})\hookrightarrow K_0(\frakF'_{r,m})$ induced by $\iota$ is an equivalence.
Note that the morphism spaces of both Grothendieck groups are free $\calA$-modules with bases given by classes of indecomposable $1$-morphisms, hence $[\iota]$ is a split injection. Hence it is an equivalence if and only if it is so after base exchange from $\calA$ to $\bbQ(q)$. Now, by Theorem~ \ref{thm:SBim}, we have $K_0(\frakF'_{r,m})\cong\bfS^\jmath_{r,m}$.
Therefore $K_0(\frakF'_{r,m})\otimes_\calA\bbQ(q)$ is generated by $1_\bfa$, $\bfe_i1_\bfa$, $\bff_i1_\bfa$, $i\in\bbI^\jmath_r$. By the definition of $\frakF_{r,m}$, all these generators lie in the image of $[\iota]$. Therefore $[\iota]$ is an equivalence after base change to $\bbQ(q)$. Part (a) is proved.
\end{proof}

\begin{rk}
We write $\frakF^\fraka_{r,m}$ for the $2$-category defined with the same $R =S(\frakt)$ as for  $\frakF_{r,m}$
but with $\W_m$ replaced by the symmetric group $\frakS_m = \langle s_1, s_2, \dots, s_{m-1} \rangle$. 
\end{rk}

\begin{rk}\label{rk:SUj}
We will also consider a quotient $\bfS\dot\frakUj$ of the $2$-category $\dot\frakUj$ by setting to 0 all objects of weight which are not of the form $\lambda(\bfa)$.  The functor $\Gamma: \dot\frakUj\to \frakF_{r,m}$ factors through this quotient. We claim that 
\begin{equation}\label{eq:SUj}
K_0(\bfS\dot\frakUj)\cong \bfS^\jmath_{r,m}.
\end{equation}
Indeed, consider the following composition of functors $${}_\calA\bdUj\to K_0(\dot\frakUj)\to K_0(\bfS\dot\frakUj)\to K_0(\frakF_{r,m})\simto \bfS^\jmath_{r,m}.$$  Since the map ${}_\calA\bdUj\to \bfS^\jmath_{r,m}$ is surjective, with kernel generated by the objects not of the form $1_{\lambda(\bfa)}$, see \cite[Prop~4.11, Lemma~A.20, Thm A.21]{BKLW},
 we must have $K_0(\bfS\dot\frakUj)\cong \bfS^\jmath_{r,m}$.
\end{rk}

\medskip

\subsection{Frobenius forms and Demazure operators}
\label{ss:demazure}

Let $I$ be a subset of $[1,m]$. Let $\Lambda_I$ be the ring of symmetric functions in $\{t_i;\,i\in I\}$. For $p\geqslant 0$, let $e_{p,I}=e_p(t_i; i\in I)$ be the $p$-th elementary symmetric polynomial, and let 
 $h_{p,I}=h_p(t_i; i\in I)$ be the $p$-th complete symmetric polynomial. They are defined by the following generating functions
 \begin{align*}
&\sum_{p\geqslant 0} e_{p,I}z^p=\prod_{i\in I}(1+t_iz),\quad \sum_{p\geqslant 0}h_{p,I}z^p=\prod_{i\in I}(1-t_iz)^{-1},
\end{align*}
where $z$ is a formal variable.
Note that
\begin{align}
\sum_{p+s=k}(-1)^pe_{p, I}h_{s, I}=\delta_{k,0}.\label{sym1}
\end{align}
We also write $\Lambda_I^\jmath$ for the ring of symmetric functions in $\{t^2_i;\,i\in I\}$, and 
$e^\jmath_{p,I}=e_p(t^2_i; i\in I)$, $h^\jmath_{p,I}=h_p(t^2_i; i\in I)$. For future use, we introduce the convention that $h^\jmath_{p,I}=0$ if $p\not\in\bbN$.


Let $B$ be a $\bfk$-algebra, and $A$ a finitely generated $B$-algebra which is free as a $B$-module. 
A $(B,B)$-linear map $\phi: A\to B$ is a \emph{Frobenius form} if $A\to\Hom_B(A,B),$ $a\mapsto (b\mapsto \phi(ab))$ is an isomorphism. 
The \emph{Casimir element} associated with $\phi$ is an element $\pi\in (A\otimes_{B} A)^A$ such that $(\phi\otimes 1)(\pi)=(1\otimes\phi)(\pi)=1$. 
The restriction from $A$ to $B$ is left adjoint to $A\otimes_B-$ with counit given by $\phi$ and unit given by $A\mapsto  A\otimes_B A$, $1\mapsto\pi$. 
See for example \cite[Section 2.3]{R} for more details.

We recall some basic facts on Demazure operators. Let $\alpha_{s_0}^\vee=-2t_1$, and $\alpha_{s_p}^\vee=t_p-t_{p+1}$ for $p>0$. 
For any reflection $\tau \in \W$, let $\alpha_\tau^\vee$ denote the corresponding positive coroot.
For $s\in\S$ the \emph{Demazure operator} $\partial_s: R\to R$ is given by
$$
\partial_s(f)=\frac{f-s(f)}{\alpha_s^\vee},\quad \forall\ f\in R.
$$
For $w\in \W$ let $\partial_w=\partial_{s_{i_1}}\partial_{s_{i_2}}...$ for a reduced expression $s_{i_1}s_{i_2}...$ of $w$. 
It is a well defined operator of degree $-2\ell(w)$, where $\ell(w)$ is the length of $w$.
For $I\subset\S$, let $w_I$ be the longest element in $W_I$. Let $d_I$ be the product of $\alpha_\tau^\vee$ for all reflections $\tau$ in $\W_I$. Then 
$$
\partial_{w_I}(f)=\frac{1}{|\W_I|}\sum_{w\in\W_I}w(d_I^{-1}f), \quad \forall\ f\in R.
$$
It yields a Frobenius form $\partial_{w_I}:R\{- 2 \ell(w_I)\}\to R^{I}$, see e.g. \cite[Section~3.1]{Wilthesis}.
\begin{rk}
  As mentioned before, everything we have done up to this point only
  depends on the Weyl group, not on the underlying root system.  However, $d_I$ does depend on the choice of a root system; if
  we let $d_I'$ be the corresponding product of coroots in type C rather
  than type B,  $d_I$ and $d_I'$ differ by a power of $2$.  
\end{rk}

\smallskip

For $I'\subset I\subset \S$, let $w_{I,I'}=w_Iw_{I'}^{-1}$, $d_{I,I'}=d_I/d_{I'}$. Let $W^{I,{I'}}$ be the set of minimal coset representatives of $W_I/W_{I'}$. Then 
\begin{equation}\label{eq:demazure}
\partial_{w_{I,{I'}}}: R^{I'}\{-2\ell(w_{I'}w_I^{-1})\}\longrightarrow R^I,\quad f\mapsto \frac{1}{|\W_I/\W_{I'}|}\sum_{w\in\W^{I,{I'}}}w(d_{I,{I'}}^{-1}f)
\end{equation}
is also a Frobenius form by \cite[Lemma~2.12]{R}.
Let us compute this form in some useful examples.

\begin{ex} \label{ex:demazure}
Let $1\leqslant a<b\leqslant m$. 

\begin{itemize}
\item[(a)] 
Consider $I'=\{s_p;\,p\in [a+1,b-1]\}\subset I=\{s_p;\,p\in [a,b-1]\}$. 
Then $w_{I,I'}=s_{b-1}...s_a$. We will write $\partial_{[b,a]}=\partial_{w_{I,I'}}$.
The Frobenius form \eqref{eq:demazure} in this case becomes 
\begin{align*}
\partial_{[b,a]}: \bfk[t_a]\otimes\Lambda_{[a+1,b]}\{2(a-b)\}\longrightarrow\Lambda_{[a,b]}, \quad f\mapsto \sum_{p=a}^{b}s_{(a,p)}\Big(f\prod_{u=a+1}^{b}(t_a-t_u)^{-1}\Big).
\end{align*}
Here we have ignored the variables $t_i$ for $i\not\in [a,b]$, since $\partial_{[a,b]}$ acts trivially on them.
An easy computation shows that $\partial_{[b,a]}(t_a^k)=h_{k-b+a,[a,b]}$ for all $k$.
Therefore $\{t_a^k\}$ and $\{(-1)^re_{r,[a+1,b]}\}$ are dual bases with respect to the Frobenius form $\partial_{[b,a]}$.
Hence the Casimir element is given by 
\begin{equation*}
\pi_{[b,a]}=\sum_{r=0}^{b-a}t_a^{b-a-r}\otimes (-1)^re_{r,[a+1,b]}.
\end{equation*} 

\item[(b)] 
Consider $I'=\{s_p;\,p\in [a,b-2]\}\subset I=\{s_p;\,p\in [a,b-1]\}$. 
Then $w_{I,I'}=s_a...s_{b-1}$. We will write $\partial_{[a,b]}=(-1)^{b-a}\partial_{w_{I,I'}}$.
The Frobenius form \eqref{eq:demazure} in this case becomes 
\begin{align*}
\partial_{[a,b]}: \Lambda_{[a,b-1]}\otimes\bfk[t_b]\{2(a-b)\}\longrightarrow\Lambda_{[a,b]}, \quad f\mapsto \sum_{p=a}^{b}s_{(p,b)}\Big(f\prod_{u=a}^{b-1}(t_b-t_u)^{-1}\Big).
\end{align*}
We have $\partial_{[a,b]}(t_b^k)=h_{k-b+a,[a,b]}$ for all $k$, and the Casimir element is
\begin{equation*}
\pi_{[a,b]}=\sum_{r=0}^{b-a}(t_b)^{b-a-r}\otimes (-1)^re_{r,[a,b-1]}.
\end{equation*}

\item[(c)] 
Consider $I'=\{s_p;\,p\in [0,a-2]\}\subset I=\{s_p;\,p\in [0,a-1]\}$. We have $w_{I,I'}=s_{a-1}...s_1s_0s_1...s_{a-1}$. Call this element $\gamma_a$. Note that $\gamma_a(t_i)=(-1)^{\delta_{i,a}}t_i$.
Write $\tilde{\partial}_{[1,a]}=(-1)^a\partial_{w_{I,I'}}$.
The Frobenius form \eqref{eq:demazure} in this case becomes 
\begin{align*} 
\tilde{\partial}_{[1,a]}&: \Lambda^\jmath_{[1,a-1]}\otimes\bfk[t_a]\{2(1-2a)\}\longrightarrow\Lambda^\jmath_{[1,a]}, \\
&f\mapsto\sum_{p=1}^{a}s_{(p,a)}(1+\gamma_a)\left(f \prod_{u=1}^{a-1}(t_a^2-t^2_u)^{-1} (2t_a)^{-1}\right).
\end{align*}
Let us compute $\tilde{\partial}_{[1,a]}(t_a^k)$.
Note that $(1+\gamma_a)(t_a^k)=2t_a^{k}$ if $k$ is even, and zero otherwise.
Hence $\tilde{\partial}_{[1,a]}(t_a^k)$ is nontrivial only when $k$ is odd, and in this case by Example (b)
\begin{align*}
\tilde{\partial}_{[1,a]}(t_a^k)
&=\sum_{p=1}^{a}s_{(p,a)}\left(t_a^{k-1}\prod_{u=1}^{a-1}(t_a^2-t^2_u)^{-1}\right)
\\
&=h^\jmath_{(k-1)/2-a+1,[1,a]}.
\end{align*}
The Casimir element is given by 
\begin{equation*}
\tilde\pi_{[1,a]}=(t_a\otimes 1+1\otimes t_a)\sum_{r=0}^{b-a}t_a^{2(a-1-r)}\otimes (-1)^re^\jmath_{r,[1,a-1]}.
\end{equation*}
\end{itemize}
\end{ex}

\subsection{Action of $\frakU^\jmath $ on the Schur $2$-category}

Recall Khovanov and Lauda defined a $2$-functor $\Gamma^\fraka: \frakU\to\frakF ^\fraka$ in \cite{KLIII},
giving a $2$-representation in terms of (equivariant)  cohomology rings of partial flag varieties of type $\A$. 
We now define an analogous $2$-functor 
\begin{equation}
\label{eq:Gamma}
\Gamma: \frakU^\jmath  \longrightarrow  \frakF_{r,m}.
\end{equation}
On objects $\Gamma$ is given by
$$
\X_\jmath\ni\lambda \mapsto \begin{cases} I_\bfa &\si {\lambda=\lambda(\bfa)},\\ 0 &\sinon.
\end{cases}
$$
On the generating $1$-morphisms $\Gamma$ sends
\begin{align*}
1_\lambda\{s\} &\mapsto \begin{cases} R^\bfa\{ s\} &\si {\lambda=\lambda(\bfa)},\\ 0 &\sinon,\end{cases}\\
\calE_i1_\lambda\{ s\} &\mapsto 
\begin{cases} R^{\bfa^{+i}}\{ s+1+a_i-a_{i+1}\} &\si {\lambda=\lambda(\bfa)},\\ 0 &\sinon,
\end{cases}\\
\calF_i1_\lambda\{ s\} &\mapsto 
\begin{cases} R^{\bfa^{-i}}\{ s+1+a_{i-1}-a_i\} &\si {\lambda=\lambda(\bfa)},\\ 0 &\sinon.
\end{cases}
\end{align*}
On the generating $2$-morphisms $\Gamma$ is given as follows:
\begin{itemize} 
\item If $\calM$ is a monomial in only $\calE_i$'s or only $\calF_i$'s, then the multiplication map identifies the bimodules $\Gamma(\calM)$ with a subring of $R$. 
Thus we will write the image of $\Gamma$ on diagrams consisting of only upward arrows or only downward arrows in terms of endomorphism of $R$ 
which preserves the corresponding subrings.  
Then
\begin{alignat}{3}
\Gamma(\mathord{
\begin{tikzpicture}[baseline = 0]
	\draw[->,thick,darkred] (0.08,-.3) to (0.08,.4);
      \node at (0.08,0.05) {\color{darkred}$\bullet$};
         \node at (0.08,-.4) {$\scriptstyle{i}$};
\end{tikzpicture}
}
{\scriptstyle \lambda}\: )=t_{a_i+1},
&
\qquad \qquad 
\Gamma(\mathord{
\begin{tikzpicture}[baseline = 0]
	\draw[<-,thick,darkred] (0.08,-.3) to (0.08,.4);
       \node at (0.08,0.05) {\color{darkred}$\bullet$};
             \node at (0.08,-.4) {$\scriptstyle{i}$};
\end{tikzpicture}
}
{\scriptstyle \lambda}\: )=t_{a_i},
\\
\Gamma(\mathord{
\begin{tikzpicture}[baseline = 0]
	\draw[->,thick,darkred] (0.28,-.3) to (-0.28,.4);
	\draw[->,thick,darkred] (-0.28,-.3) to (0.28,.4);
   \node at (.4,.05) {$\scriptstyle{\lambda}$};
      \node at (-0.28,-.4) {$\scriptstyle{i}$};
   \node at (0.28,-.4) {$\scriptstyle{j}$};
\end{tikzpicture}
}\:)
=\begin{cases}
\partial_{a_i+1}&\si i=j\\
(t_{a_j+1}-t_{a_i+1}) &\si i=j+1\\
1 &\sinon\,,
\end{cases}
& \quad
%
\Gamma(\mathord{
\begin{tikzpicture}[baseline = 0]
	\draw[<-,thick,darkred] (0.28,-.3) to (-0.28,.4);
	\draw[<-,thick,darkred] (-0.28,-.3) to (0.28,.4);
   \node at (.4,.05) {$\scriptstyle{\lambda}$};
      \node at (-0.28,-.4) {$\scriptstyle{i}$};
   \node at (0.28,-.4) {$\scriptstyle{j}$};
\end{tikzpicture}
}\:)
=\begin{cases}
\partial_{a_i-1}  &\si i=j\\
(t_{a_i}-t_{a_j}) &\si j=i+1\\
1 &\sinon.
\end{cases}
\end{alignat}

\item 
The adjunction maps are given by
\begin{align}\label{label:ep'-Gamma}
\Gamma(\mathord{
\begin{tikzpicture}[baseline = 0]
	\draw[-,thick,darkred] (0.4,-0.1) to[out=90, in=0] (0.1,0.3);
	\draw[->,thick,darkred] (0.1,0.3) to[out = 180, in = 90] (-0.2,-0.1);
    \node at (-0.2,-.2){};   
    \node at (0.4,-.2) {$\scriptstyle{i}$};
    \node at (0.3,0.4) {$\scriptstyle{\lambda}$};
\end{tikzpicture}
}): \quad&R^{\bfa^{+i}}\otimes_{R^{{}_{+i}\bfa}}R^{\bfa^{+i}}\{1-\delta_{i,\hf}+a_{i-1}-a_{i+1}\}\longrightarrow  R^\bfa\{1-\l\thal_i^\vee,\lambda+\alpha_i\r\},\\ \displaybreak[0]
\notag
 &f\otimes g \mapsto \partial_{[a_{i+1}, a_i+1]}(fg),\\
 \label{label:eta'-Gamma} 
\Gamma(\mathord{
\begin{tikzpicture}[baseline = 0]
	\draw[-,thick,darkred] (0.4,0.3) to[out=-90, in=0] (0.1,-0.1);
	\draw[->,thick,darkred] (0.1,-0.1) to[out = 180, in = -90] (-0.2,0.3);
	    \node at (0.4,.4) {$\scriptstyle{i}$};
  \node at (0.3,-0.15) {$\scriptstyle{\lambda}$};
\end{tikzpicture}
}\:): \quad&R^\bfa\longrightarrow  R^{\bfa^{-i}}\otimes_{R^{{}_{-i}\bfa}}R^{\bfa^{-i}}\{a_{i-1}-a_{i+1}+\l\thal_i^\vee,\lambda\r\},\\
\notag
 &1\mapsto\pi_{[a_{i+1},a_i]}
 \\
\label{label:ep-Gamma} 
\Gamma(\mathord{
\begin{tikzpicture}[baseline = 0]
	\draw[<-,thick,darkred] (0.4,-0.1) to[out=90, in=0] (0.1,0.3);
	\draw[-,thick,darkred] (0.1,0.3) to[out = 180, in = 90] (-0.2,-0.1);
	    \node at (-0.2,-.2) {$\scriptstyle{i}$};
  \node at (0.3,0.4) {$\scriptstyle{\lambda}$};
\end{tikzpicture}
}): \quad&R^{\bfa^{-i}}\otimes_{R^{{}_{-i}\bfa}}R^{\bfa^{-i}}\{1+a_{i-1}-a_{i+1}\}\longrightarrow  R^\bfa\{-1+\l\thal_i^\vee,\lambda\r\},\\ \displaybreak[0]
\notag
&f\otimes g \mapsto \begin{cases}
\tilde\partial_{[1,a_i]}(fg), &\si i=\hf,\\
\partial_{[a_{i-1}+1,a_i]}(fg), &\sinon,\end{cases}\\
\label{label:eta-Gamma}
\Gamma(\mathord{
\begin{tikzpicture}[baseline = 0]
	\draw[<-,thick,darkred] (0.4,0.3) to[out=-90, in=0] (0.1,-0.1);
	\draw[-,thick,darkred] (0.1,-0.1) to[out = 180, in = -90] (-0.2,0.3);
	    \node at (-.2,.4) {$\scriptstyle{i}$};
  \node at (0.3,-0.15) {$\scriptstyle{\lambda}$};
\end{tikzpicture}
}\:): \quad&R^\bfa\longrightarrow  R^{\bfa^{+i}}\otimes_{R^{{}_{+i}\bfa}}R^{\bfa^{+i}}\{2-\delta_{i,\hf}+a_{i-1}-a_{i+1}-\l\thal^\vee_i,\lambda+\alpha_i\r\},\\  \displaybreak[0]
\notag
 &1\mapsto\begin{cases}  
\tilde\pi_{[1,a_i+1]} 
&\si i=\hf,\\
 \pi_{[a_{i-1}+1,a_i+1]}
 &\sinon.
 \end{cases}
\end{align}
\end{itemize}

\begin{thm}\label{thm:2repflagj}
The above assignments define a locally essentially surjective $2$-functor
$$\Gamma: \frakU^\jmath \longrightarrow  \frakF_{r,m}.
$$
\end{thm}

\begin{proof}
Obviously $\Gamma$ induces $\kappa^{-1}\circ\gamma$ on the Grothendieck group;
 see Proposition~\ref{prop:iFl-Grothendieck} and the paragraph above it for the notation. Hence it is enough to check it is well defined on $2$-morphisms.

First, let us check the compatibility with the grading. This is obvious on the generators $x$ and $\tau$. For $\mathord{
\begin{tikzpicture}[baseline = 0]
	\draw[-,thick,darkred] (0.4,-0.1) to[out=90, in=0] (0.1,0.3);
	\draw[->,thick,darkred] (0.1,0.3) to[out = 180, in = 90] (-0.2,-0.1);
    \node at (-0.2,-.2){};   
        \node at (0.3,0.4) {$\scriptstyle{\lambda}$};
\end{tikzpicture}
}$, 
note that the coefficient $a_{i-1}$ for ${}_{+i}\bfa$ is the same as the one for $\bfa$ if $i\neq\hf$, but it differs from the one for $\bfa$ by $-1$ for $i=\hf$. 
This is why $-\delta_{i,0}$ appears in the degree shift on the left. Recall that $\lambda_i=-a_{i-1}+2a_i-a_{i+1}$ and $\l\thal_i^\vee,\al_i\r=2+\delta_{i,\hf}$.
Hence \[1-\delta_{i,\hf}+a_{i-1}-a_{i+1}-1+\l\thal_i^\vee,\lambda+\alpha_i\r=2+2a_i-2a_{i+1}=\deg(\partial_{[a_{i+1},a_i+1]}).\]
Using similar computations, we can verify that the other three adjunction maps are compatible with grading.

Next, let us check the relations in $\wU$. 
By definition, the $2$-morphism generators which do not involve $i=\hf$ act in the same way as in \cite{KLIII}, 
hence all the relations involving only these generators are satisfied.
The computations for relations 
\eqref{cyc-x}--\eqref{qhalast}, \eqref{nodal}, \eqref{mixedup}, even with $i=\hf$ involved, are also entirely similar to loc.~cit.~, and we omit the details. 
The adjunction \eqref{adj} for $i=\hf$ follows from the fact that $\tilde\partial_{[1,a_i]}$ is a Frobenius form, see Example~\ref{ex:demazure}(c). 

To check the bubble relations and the bicross relations for $i=\hf$, let us write $a=a_\hf$, $b=a_{\hf+1}$.
By Example~\ref{ex:demazure}, we have
\begin{align}
\Gamma\left(\mathord{
\begin{tikzpicture}[baseline = 0]
  \draw[<-,thick,darkred] (0,0.4) to[out=180,in=90] (-.2,0.2);
  \draw[-,thick,darkred] (0.2,0.2) to[out=90,in=0] (0,.4);
 \draw[-,thick,darkred] (-.2,0.2) to[out=-90,in=180] (0,0);
  \draw[-,thick,darkred] (0,0) to[out=0,in=-90] (0.2,0.2);
   \node at (0.35,0.2) {$\scriptstyle{\lambda}$};
   \node at (0,0) {$\color{darkred}\bullet$};
   \node at (0,-.2) {$\color{darkred}\scriptstyle{s}$};
\end{tikzpicture}
}\right)&=2\sum_{p=0}^{a}(-1)^pe^\jmath_{p,[1,a]}h_{-\lambda_\hf+1-2p+s,[a+1,b]}
,\label{eq:bubble1}\\
\Gamma\left(\mathord{
\begin{tikzpicture}[baseline = 0]
  \draw[->,thick,darkred] (0,0.4) to[out=180,in=90] (-.2,0.2);
  \draw[-,thick,darkred] (0.2,0.2) to[out=90,in=0] (0,.4);
 \draw[-,thick,darkred] (-.2,0.2) to[out=-90,in=180] (0,0);
  \draw[-,thick,darkred] (0,0) to[out=0,in=-90] (0.2,0.2);
   \node at (0.35,0.2) {$\scriptstyle{\lambda}$};
   \node at (0,0) {$\color{darkred}\bullet$};
   \node at (0,-.2) {$\color{darkred}\scriptstyle{s}$};
\end{tikzpicture}
} \right)&=\sum_{p=0}^{b-a}(-1)^pe_{p,[a+1,b]}h^\jmath_{\frac{\lambda_\hf+2-p+s}{2},[1,a]}.\label{eq:bubble2}
\end{align}
Let $z$ be a formal variable. Then we have
\begin{align}
&\sum_{s\in\bbZ}\Gamma\left(\mathord{
\begin{tikzpicture}[baseline = 0]
  \draw[<-,thick,darkred] (0,0.4) to[out=180,in=90] (-.2,0.2);
  \draw[-,thick,darkred] (0.2,0.2) to[out=90,in=0] (0,.4);
 \draw[-,thick,darkred] (-.2,0.2) to[out=-90,in=180] (0,0);
  \draw[-,thick,darkred] (0,0) to[out=0,in=-90] (0.2,0.2);
   \node at (0.35,0.2) {$\scriptstyle{\lambda}$};
   \node at (0,0) {$\color{darkred}\bullet$};
   \node at (0,-.2) {$\color{darkred}\scriptstyle{\lambda_\hf-1+s}$};
\end{tikzpicture}
}\right)z^s
=2\sum_{s\in\bbZ}\sum_{p\geqslant 0}(-1)^pe^\jmath_{p,[1,a]}z^{2p}h_{s-2p,[a+1,b]}z^{s-2p}
=\frac{2\prod_{i=1}^{a}(1-t_i^2z^2)}{\prod_{i=a+1}^{b}(1-t_iz)},
\label{eq:buble1} \\
&\sum_{s\in\bbZ}\Gamma\left(\mathord{
\begin{tikzpicture}[baseline = 0]
  \draw[->,thick,darkred] (0,0.4) to[out=180,in=90] (-.2,0.2);
  \draw[-,thick,darkred] (0.2,0.2) to[out=90,in=0] (0,.4);
 \draw[-,thick,darkred] (-.2,0.2) to[out=-90,in=180] (0,0);
  \draw[-,thick,darkred] (0,0) to[out=0,in=-90] (0.2,0.2);
   \node at (0.35,0.2) {$\scriptstyle{\lambda}$};
   \node at (0,0) {$\color{darkred}\bullet$};
   \node at (0,-.2) {$\color{darkred}\scriptstyle{-\lambda_\hf-2+s}$};
\end{tikzpicture}
} \right)z^s
=\sum_{s\in\bbZ}\sum_{p\geqslant 0}(-1)^pe_{p,[a+1,b]}z^ph^\jmath_{\frac{s-p}{2},[1,a]}z^{s-p}
=\frac{\prod_{i=a+1}^{b}(1-t_iz)}{\prod_{i=1}^{a}(1-t_i^2z^2)}.
\label{eq:buble2} 
\end{align}
All bubble relations for $i=\hf$ now follow 
from that 
$$\text{LHS } \eqref{eq:buble1} \cdot \text{LHS } \eqref{eq:buble2} =\text{RHS } \eqref{eq:buble1} \cdot \text{RHS } \eqref{eq:buble2} =2.
$$

To check the bicross relations \eqref{j-bicross down}--\eqref{j-bicross up}, note that we have
\begin{align*}
\Gamma\left(
\mathord{
\begin{tikzpicture}[baseline = 0]
	\draw[->,thick,darkred] (0.28,-.3) to (-0.28,.4);
	\draw[<-,thick,darkred] (-0.28,-.3) to (0.28,.4);
\end{tikzpicture}
}
\right)
\cdot(1\otimes 1)&=
\Gamma\left(
\mathord{
\begin{tikzpicture}[baseline = 0]
	\draw[<-,thick,darkred] (0,-.3) to[out=90,in=180] (.2,.1) to[out=0,in=135] (.56,-.3);
	\draw[->,thick,darkred] (0.28,-.3) to[out=45,in=-45] (0.28,.4);
	\draw[<-,thick,darkred] (0.84,-.3) to (0.84,.4);
\end{tikzpicture}
}
\right)
\cdot\left(\sum_{s=0}^{b-a}1\otimes 1\otimes t_{a}^{b-a-s}\otimes (-1)^se_{s,[a+1,b]}\right) \\
&=\Gamma\left(
 \mathord{\begin{tikzpicture}[baseline = 0]
	\draw[<-,thick,darkred] (0,-.3) to[out=90,in=180] (.14,.1) to[out=0,in=90] (.28,-.3);
	\draw[->,thick,darkred] (0.56,-.3) to (0.56,.4);
	\draw[<-,thick,darkred] (0.84,-.3) to (0.84,.4);
\end{tikzpicture}}
\right)
\cdot\left(\sum_{s=0}^{b-a-1}\sum_{k=0}^{b-a-s-1}1\otimes t_{a+1}^{b-a-s-k-1}\otimes t_{a}^{k}\otimes (-1)^{s}e_{s,[a+1,b]}\right) \\
&=1\otimes 1,
\end{align*}
\begin{align*}
\Gamma\left(
 \mathord{
\begin{tikzpicture}[baseline = 0]
	\draw[<-,thick,darkred] (0.28,-.3) to (-0.28,.4);
	\draw[->,thick,darkred] (-0.28,-.3) to (0.28,.4);
\end{tikzpicture}
}
\right)
\cdot(1\otimes 1)
&= \Gamma\left(
 \mathord{
\begin{tikzpicture}[baseline = 0]
	\draw[<-,thick,darkred] (0,-.3) to[out=90,in=0] (-.2,.1) to[out=180,in=45] (-.56,-.3);
	\draw[->,thick,darkred] (-0.28,-.3) to[out=135,in=-135] (-0.28,.4);
	\draw[<-,thick,darkred] (-0.84,-.3) to (-0.84,.4);
\end{tikzpicture}}
\right)
 \cdot\left( \sum_{s=0}^a(-1)^{a-s}(t_{a+1}e^\jmath_{s,[1,a]}\otimes t_{a+1}^{2(a-s)}+e^\jmath_{s,[1,a]}\otimes t_{a+1}^{2(a-s)+1})\otimes 1\otimes 1\right)\\
&= \Gamma\left(
 \mathord{\begin{tikzpicture}[baseline = 0]
	\draw[<-,thick,darkred] (0,-.3) to[out=90,in=0] (-.14,.1) to[out=180,in=90] (-.28,-.3);
	\draw[->,thick,darkred] (-0.56,-.3) to (-0.56,.4);
	\draw[<-,thick,darkred] (-0.84,-.3) to (-0.84,.4);
\end{tikzpicture}}
\right)
 \cdot\left(\sum_{s=0}^a (-1)^{a-s-1}
\Big(\sum_{k=0}^{2(a-s)-1} t_{a+1}e^\jmath_{s,[1,a]}\otimes t_{a+1}^{2(a-s)-k-1}\otimes t_a^k
\right. 
\\
&\qquad \qquad \qquad \qquad \qquad \qquad \qquad  
\left. 
+\sum_{k=0}^{2(a-s)} e^\jmath_{s,[1,a]}\otimes t_{a+1}^{2(a-s)-k}\otimes t_a^k \Big)\otimes 1  \right)\\
&=
-t_{a+1}\otimes 1-1\otimes t_{a+1}\label{eq:1}.
\end{align*}
The last equality follows by noting that only the terms with $k=2a-1,s=0$ contribute.  
This shows that the relations \eqref{j-bicross down} and \eqref{j-bicross up} are correct on $1\otimes 1$.  
On the other hand, a direct diagram computation shows that both sides of each of these relations have the same commutator 
with multiplication by $x\otimes 1$ or $1\otimes x$, so the relations are correct on all elements of the form
$(x^p\otimes x^q)\cdot (1\otimes 1)=t_{a+1}^p\otimes t_{a+1}^q$.  

This completes the verification of
all relations of generating $2$-morphisms under $\Gamma$, and so $\Gamma$ is well defined on $\wU$.

Finally, by Proposition~\ref{prop:iFl-Grothendieck} and  \eqref{eq:ga}, the $\jmath$-Serre relations hold in the
Grothendieck group, so Proposition \ref{df-prop} implies
that this representation of $\wU$ factors through $\frakU^\jmath$. We are done.
\end{proof}

\begin{rk}
A cyclic version of $\frakU$ was introduced in \cite{BHLW}. 
There is also an analogue of the cyclic version for $\frakU^\jmath$, and let us denote it by $\frakU^{\jmath,\cyc}$. 
In the definition of $\Gamma$, if we use adjunctions defined by Demazure operators $\partial_{w_{I,I'}}$ in Example~\ref{ex:demazure} instead of $\partial_{[a,b]}$, 
then the same formulas define a $2$-functor $\frakU^{\jmath,\cyc}\to\frakF_{r,m}$;   we will not use this version in this paper.
\end{rk}

\subsection{Decategorification}

A $2$-\emph{representation} of a $2$-category $\frakC$ is a $2$-functor from $\frakC$ to the 
2-category of $\bfk$-linear categories. That is to each object of $\frakC$, we associate a graded $\bfk$-linear category, to each $1$-morphisms a functor between corresponding categories, and $2$-morphisms are sent to natural transformation of functors.
The $2$-category $\frakF_{r,m}$ has a $2$-representation given by
sending $\bfa$ to  $R^\bfa\proj$ for $\bfa\in\Sigma_{r,m}$ and the
$1$-morphisms sent to the functors 
$$\scrE_i=R^{\bfa^{+i}}\otimes_{R^\bfa}-: R^\bfa\proj\longrightarrow  R^{{}_{+i}\bfa}\proj,\qquad
\scrF_i=R^{\bfa^{-i}}\otimes_{R^\bfa}-: R^\bfa\proj\longrightarrow  R^{{}_{-i}\bfa}\proj.$$
Combining it with $\Gamma$ defines a $2$-representation of $\frakU^\jmath$ on $\bigoplus_{\bfa\in\Sigma_{r,m}}R^\bfa\proj$.

Since  $R^\bfa$ is a polynomial ring, all projective modules
over it are free, so $K_0(R^\bfa\proj)\cong \mathcal{A}$ for all
$\bfa$.     We will identify this with the constant $
\mathcal{A}$-valued functions on the set of the $\mathbb{F}_{\bfq^{2}}$ points of  
$\G/\P_{\bfa}$; these have a natural action of the $\jmath$Schur algebra $\mathbf{S}^\jmath_{r,m}$ 
via convolution (and hence of $\bfU^\jmath$ by the pullback of $\gamma$ \eqref{eq:ga}).  Let ${}_\calA\bfL(m)$ denote
$\oplus_{\bfa\in \Sigma_{r,m}}K_0(R^\bfa\proj)$, which we have identified with locally
constant $\mathcal{A}$-valued functions on $\cup_{\bfa}\G/\P_{\bfa}$.

Under the duality of \cite{BKLW}, if we identify Schubert constant
functions on $\cup_{\bfa}\G/\P_{\bfa}$ as the tensor product
$\bfV^{\otimes m}$, then this space has mutually commuting
actions of $\bfUj$ and the Hecke algebra $\bfH_W$.  The subspace
${}_\calA\bfL(m)\subset \bfV^{\otimes m}$ is precisely the
invariants of $\bfH_W$; if we set $q=1$, then
$\bfL(m)|_{q=1}=\operatorname{Sym}^m(\bfL(1))$.

\section{Categorification and canonical bases}
 \label{sec:Schur2}
 
 In this section we will show that a slightly extended version of the
 $2$-functor $\Gamma: \frakU^\jmath  \longrightarrow  \frakF_{r,m}$ is
 locally full. This leads to the completion of the proof of the main
 theorem that the $2$-category $\dot\frakU^\jmath$ categorifies
 ${}_\calA\bfU^\jmath$. Moreover, the natural projection from
 ${}_\calA\bfU^\jmath$ to the $\jmath$Schur algebra sends canonical
 basis elements to canonical basis elements or to 0. 

\subsection{Local fullness of $\Gamma$}
\label{sec:local-fullness}

Following the idea of \cite[Theorem~9, Proof \#2]{Webcomparison}, we prove the theorem below.

\begin{thm}
  \label{thm:full1}
If $r\geqslant m$,
the $2$-functor $\Gamma: \frakU^{\jmath}_r  \longrightarrow
\frakF_{r,m}$ is locally full.
\end{thm}

\begin{proof}
We must prove that for any $\lambda$, $\mu\in\X$, the functor
$$\calH{om}_{\frakU^{\jmath}}(\lambda,\mu)\longrightarrow  \calH{om}_{\frakF_{r,m}}(\Gamma(\lambda),\Gamma(\mu))$$ is full, that is,
it is surjective on morphisms.
We may assume  $\lambda=\lambda(\bfa)$, $\mu=\lambda(\bfb)$ for some $\bfa$, $\bfb\in\Sigma_{r,m}$, otherwise the statement is trivial.
By definition, the right hand side is a full subcategory of the category of $R^\bfb\bgmod R^\bfa$. So it is enough to prove the following (here we recall
notation $\Hom^\bullet$ from Section~\ref{sec:notation}).

\smallskip
\noindent\textbf{Claim ($\star$).} 
For any $M,N\in\calH{om}_{\frakU^{\jmath}}(\lambda,\mu)$ monomials in $\calE_i$, $\calF_i$, the map 
$$\Gamma_{M,N}: \Hom^\bullet_{\frakU^{\jmath}}(M,N)\longrightarrow  \Hom^\bullet_{R^\bfb\bgmod R^\bfa}(\Gamma(M),\Gamma(N))$$
induced by $\Gamma$ is surjective. 

The proof will be carried out in two steps. Set
\begin{equation}
\bfo=(0,1,2,\ldots,m-1,m,m,\ldots)\in\Sigma_{r,m}. 
\end{equation}

\medskip
\noindent\emph{Step 1. The case $\bfa=\bfb=\bfo$.}
\smallskip

In this case, we are considering the category of (nonsingular) Soergel bimodules, since $I_\bfo=\emptyset$ and $R^\bfo=R$.  Let 
\[
\calB_{s_a}=\Gamma(\calF_{\hf+a}\calE_{\hf+a})1_\bfo=R\otimes_{R^{s_a}}R\{-1\}\text{ for }0\leqslant a \leqslant m-1.
\] 
Such bimodules are called
{\it Bott-Samelson bimodules}.
They are generators for the 1-morphisms $\bfo\to \bfo$, since they generate the tensor category of Soergel bimodules.  

Next, recall the $2$-category $\bfS\dot \frakUj$ from Remark \ref{rk:SUj}, by definition the $2$-functor $\dot \frakUj\to \bfS\dot \frakUj$ is full, and $\Gamma$ factorises through this quotient. Hence it is enough to show ($\star$) for $1$-morphisms $M$, $N$ in $\bfS\dot \frakUj$. But $K_0(\bfS\dot \frakUj(\lambda(\bfo),\lambda(\bfo)))$ is the same as $\bfS_W(\bfo,\bfo)$ by \eqref{eq:SUj}. Since the classes $[\calB_{s_a}]$ generate the algebra $\bfS_W(\bfo,\bfo)$, the $1$-morphisms $\calF_{\hf+a}\calE_{\hf+a}$ for $0\leqslant a\leqslant m-1$ also generate the 1-morphisms in $\bfS\dot \frakUj(\lambda(\bfo),\lambda(\bfo))$.  
Thus it suffices to show
($\star$) for all $M$, $N$
that are tensor products of $\calF_{\hf+a}\calE_{\hf+a}$'s.

The category of Bott-Samelson bimodules has a description due to
Elias and Williamson \cite{EW}, using Soergel calculus, via a set of
generators and relations for homomorphisms between them.  
We will not need the full power of this presentation, just that it
gives us a small set of generators for all morphisms; this was shown earlier by Libedinsky using his {\it light leaf basis} \cite{Lib}.  
Given a monoidal category $\calC$ and a collection $C$ of objects which is closed under tensor
product, we say that morphisms between these objects are {\it locally
  generated} by a finite collection $F$ of morphisms if there is no proper
subset of morphisms between the objects in $C$ which contains $F$, and the
identity on each object, and is closed under composition and tensor product.

\begin{prop}[\mbox{\cite[5.1]{Lib}, \cite[6.28]{EW}}]\label{prop:EW-generate}
  The morphisms between Bott-Samelson bimodules are
  locally generated by (for $s,t \in S$):
  
    \begin{enumerate}
  \item all polynomial multiplications on the left and right;
  \item the unit, counit, multiplication and comultiplication for the
    Frobenius extension $R\supset R^s$; 
\item the unique nonzero (up to scalar) degree 0 morphism 
$$
{\rm \texttt{b}} \colon \underbrace{\calB_s\otimes_R \calB_t\otimes_R \calB_s \cdots}_{m_{st}} \longrightarrow 
   \underbrace{\calB_t\otimes_R \calB_s\otimes_R \calB_t \cdots}_{m_{st}}
   $$ 
   where    $m_{st}$ denotes the order of $st$ in the Weyl group.
  \end{enumerate}
\end{prop}

Thus, in order to complete Step 1, we must find 2-morphisms in
$\dot\frakU^\jmath$ whose images are the morphisms  listed in Proposition~\ref{prop:EW-generate}.  Let us consider these in turn:
\begin{enumerate}
\item By the equations \eqref{eq:bubble1}--\eqref{eq:bubble2}, we see
  that the degree 1 bubbles with label $\diamond$ give $\pm t_1$.
  Those with label $a+\diamond$ give $\pm t_a \mp t_{a+1}$ by
  \cite[(6.23)]{KLIII}, for $1 \le a \le m-1$.  Together, these generate all elements of $R$
  acting on the identity 1-morphism of $\bfo$.  
 
\item The adjoints in $\frakU^\jmath$ are sent under $\Gamma$ to the adjoints
  induced by the standard Frobenius structure for $R\supset R^s$ by the definitions
  \eqref{label:ep'-Gamma}--\eqref{label:eta-Gamma}.
\item Fix $i,j\in \bbI^\jmath$ and let $a=i-\hf$ and $b=j-\hf$. Let $s=s_a$ and $t=s_b$.  
 If $|i-j| >1$, then we simply need an isomorphism $\calB_{s_a}\calB_{s_b}\cong \calB_{s_b}\calB_{s_a}$, 
which is supplied by the following mutual inverse diagrams:
\begin{equation}
B=\mathord{
		\begin{tikzpicture}[baseline=0,scale=0.5]
			\draw[->,thick,darkred] (-2,1) -- node [below, at end]{$j$} (1, -1);
			\draw[->,thick,darkred] (1, 1) --  node [below, at end]{$i$} (-2, -1);
			\draw[->,thick, darkred] (-1, -1) --node [below, at start]{$i$} (2,1);
\draw[->,thick, darkred] (2, -1) --node [below, at start]{$j$} (-1,1);
		\end{tikzpicture}
					}\qquad C= \mathord{
		\begin{tikzpicture}[baseline=0,scale=0.5]
			\draw[->,thick,darkred] (-2,1) -- node [below, at end]{$i$} (1, -1);
			\draw[->,thick,darkred] (1, 1) --  node [below, at end]{$j$} (-2, -1);
			\draw[->,thick, darkred] (-1, -1) --node [below, at start]{$j$} (2,1);
\draw[->,thick, darkred] (2, -1) --node [below, at start]{$i$} (-1,1);
		\end{tikzpicture}
					}\label{eq:sl2xsl2}
                                      \end{equation}
If $j=i+1$ with $i>\hf$, then $\texttt{b}$ (up to a scalar multiple) is given by the
diagrams
\begin{equation}
  \label{eq:sl3}
   B=\mathord{   \begin{tikzpicture}[baseline=0,thick,scale=0.15, xscale=2, darkred]
      \draw[<-] (-5,-5) to[out=90,in=180] (0,1) to[out=0,in=90] (5,-5);
      \draw[->] (-3,-5) to[out=90,in=-90] (1,1) to (1,5);
      \draw[<-] (3,-5) to[out=90,in=-90] (-1,1) to (-1,5);
\draw[<-] (-1,-5) to[out=90,in=-90] (-5,5);
\draw[->] (1,-5) to[out=90,in=-90] (5,5);
\draw[->] (3,5) to[out=-90,in=0] (0,2) to[out=180,in=-90] (-3,5);
      \node at (-3,-7) { $i$};
     \node at (1,-7) { $j$};
     \node at (5,-7) { $i$};
  \node at (-5,-7) { $i$};
     \node at (-1,-7) { $j$};
     \node at (3,-7) { $i$};
     \node at (-3,7) { $j$};
     \node at (1,7) { $i$};
     \node at (5,7) { $j$};
  \node at (-5,7) { $j$};
     \node at (-1,7) { $i$};
     \node at (3,7) { $j$};
    \end{tikzpicture} }\quad \quad 
     C=\mathord{  \begin{tikzpicture}[baseline=0,thick,scale=-0.15,xscale=2, darkred]
      \draw[<-] (-5,-5) to[out=90,in=180] (0,1) to[out=0,in=90] (5,-5);
      \draw[->] (-3,-5) to[out=90,in=-90] (1,1) to (1,5);
      \draw[<-] (3,-5) to[out=90,in=-90] (-1,1) to (-1,5);
\draw[<-] (-1,-5) to[out=90,in=-90] (-5,5);
\draw[->] (1,-5) to[out=90,in=-90] (5,5);
\draw[->] (3,5) to[out=-90,in=0] (0,2) to[out=180,in=-90] (-3,5);
      \node at (3,7) { $j$};
     \node at (-1,7) { $i$};
     \node at (-5,7) { $j$};
      \node at (5,7) { $j$};
     \node at (1,7) { $i$};
     \node at (-3,7) { $j$};
      \node at (3,-7) { $i$};
     \node at (-1,-7) { $j$};
     \node at (-5,-7) { $i$};
      \node at (5,-7) { $i$};
     \node at (1,-7) { $j$};
     \node at (-3,-7) { $i$};
    \end{tikzpicture}  }
\end{equation}
If $i=\hf$ and $j=\hf+1$, then $\texttt{b}$ (up to a scalar multiple)  is given by the
diagrams
\begin{equation}
  \label{eq:sp4}
       B=\mathord{      \begin{tikzpicture}[baseline=0,thick,scale=0.15, xscale=2, darkred]
      \draw[<-] (-7,-5) to[out=90,in=-90] (-7,-3) to[out=90,in=-90]  (5,5);
      \draw[->] (-5,-5) to[out=90,in=-90] (7,3) to[out=90,in=-90]
      (7,5);
      \draw[<-] (1,-5) to[out=90,in=-90]  (-3,5);
      \draw[->] (3,-5) to[out=90,in=-90] (-1,5);
\draw[<-] (-3,-5) to[out=90,in=-90] (-7,5);
\draw[->] (7,-5) to[out=90,in=-90] (3,5);
\draw[->] (1,5) to[out=-90,in=0] (-2,2) to[out=180,in=-90] (-5,5);
\draw[<-] (5,-5) to[out=90,in=0] (2,-2) to[out=180,in=90] (-1,-5);
      \node at (-3,-7) { $j$};
     \node at (1,-7) { $i$};
     \node at (5,-7) { $j$};
  \node at (-5,-7) { $i$};
     \node at (7,-7) { $j$};
  \node at (-7,-7) { $i$};
     \node at (-1,-7) { $j$};
     \node at (3,-7) { $i$};
     \node at (-3,7) { $i$};
     \node at (1,7) { $j$};
     \node at (5,7) { $i$};
  \node at (-5,7) { $j$};
     \node at (-1,7) { $i$};
     \node at (3,7) { $j$};
     \node at (7,7) { $i$};
  \node at (-7,7) { $j$};
    \end{tikzpicture} }\quad \quad 
       C=\mathord{ \begin{tikzpicture}[baseline=0,thick,scale=0.15, xscale=-2, darkred]
      \draw[->] (-7,-5) to[out=90,in=-90] (-7,-3) to[out=90,in=-90]  (5,5);
      \draw[<-] (-5,-5) to[out=90,in=-90] (7,3) to[out=90,in=-90]
      (7,5);
      \draw[->] (1,-5) to[out=90,in=-90]  (-3,5);
      \draw[<-] (3,-5) to[out=90,in=-90] (-1,5);
\draw[->] (-3,-5) to[out=90,in=-90] (-7,5);
\draw[<-] (7,-5) to[out=90,in=-90] (3,5);
\draw[<-] (1,5) to[out=-90,in=0] (-2,2) to[out=180,in=-90] (-5,5);
\draw[->] (5,-5) to[out=90,in=0] (2,-2) to[out=180,in=90] (-1,-5);
      \node at (-3,-7) { $j$};
     \node at (1,-7) { $i$};
     \node at (5,-7) { $j$};
  \node at (-5,-7) { $i$};
     \node at (7,-7) { $j$};
  \node at (-7,-7) { $i$};
     \node at (-1,-7) { $j$};
     \node at (3,-7) { $i$};
     \node at (-3,7) { $i$};
     \node at (1,7) { $j$};
     \node at (5,7) { $i$};
  \node at (-5,7) { $j$};
     \node at (-1,7) { $i$};
     \node at (3,7) { $j$};
     \node at (7,7) { $i$};
  \node at (-7,7) { $j$};
    \end{tikzpicture}}
\end{equation}
\end{enumerate}

The confirmation that the morphisms of \eqref{eq:sl2xsl2} act correctly
is straightforward from the definition.  The case of \eqref{eq:sl3}
is confirmed by \cite[6.8]{MSV}.  Of course, we could verify
\eqref{eq:sp4} by direct calculation, but in fact, we have arrived at
this formula systematically by applying the results of  Elias
\cite{Elias} on Soergel bimodules for dihedral groups.  This 
gives a formula \cite[6.4]{Elias} for writing the degree 0 morphism in terms of the Frobenius square $R\supset
R^s,R^t\supset R^{s,t}$.  This square gives four Frobenius extensions:
\begin{equation}
A=R,A'=R^s;\quad A=R,A'=R^t; \quad A=R^s,A'=R^{s,t}; \quad
A=R^t,A'=R^{s,t}. \label{eq:Frob-pair}
\end{equation} 
Elias's formula allows us to write the desired
morphism in terms of:  
\begin{enumerate}
\renewcommand{\theenumi}{\roman{enumi}}
  \item all polynomial multiplications on the left and right;
  \item the unit, counit, multiplication and comultiplication for the
    Frobenius extensions $A\supset A'$ listed in \eqref{eq:Frob-pair}; 
\item the obvious isomorphism of $(R^{s,t}, R)$-bimodules
  \begin{equation}
  {}_{R^{s,t}}R^{s,t}\otimes_{R^{s,t}}R^s\otimes_{R^s}R_{R}\cong {}_{R^{s,t}}R_{R}\cong
   {}_{R^{s,t}}R^{s,t}\otimes_{R^{s,t}}R^t\otimes_{R^{t}}R_{R}.\label{eq:obvious}
\end{equation}
  \end{enumerate}
  Here and below, the left and right subscripts indicate the bimodule structure. For example, ${}_{R^{s,t}}R_{R}$ stands for $R$ viewed as an $(R^{s,t}, R)$-bimodule.
In order to describe these 2-morphisms in terms of  the category $\frakU^\jmath$, we need to consider some auxiliary objects:
\begin{itemize}
\item $\bfo_a=(0,1,\dots, a-1,a+1,a+1,a+2,a+3\dots)$ satisfies $I_{\bfo_a}=\{s_a\}$
  and $ R^{\bfo_a}=R^{s_a}$.  
\item 
  $\bfo_{a,a+1}=(0,1,\dots, a-1,a+2,a+2,a+2,a+3,\dots)$ satisfies
  $I_{\bfo_a}=\{s_a,s_{a+1}\}$ and $ R^{\bfo_a}=R^{s_a,s_{a+1}}$.
\end{itemize}
Furthermore, recalling $a=i-\hf$ and $b=j-\hf$, we have bimodule isomorphisms (ignoring
grading shifts):
\[
1_{\bfo_a}\calE_{i}1_\bfo\cong {}_{R^{s_a}}R_R, \qquad 1_\bfo\calF_{i}1_{\bfo_a}\cong {}_RR{}_{R^{s_a}},
\]
\[ 1_{\bfo_{a,b}}\calE_{i}\calE_{j}1_{\bfo_a}\cong {}_{R^{s_a,s_{b}}}R^{s_a}{}_{R^{s_a}},
\qquad
 1_{\bfo_{a,b}}\calE_{i}^{(2)}1_{\bfo_{b}}\cong {}_{R^{s_a,s_{b}}}R^{s_{b}}{}_{R^{s_{b}}},
 \]
\[ 1_{\bfo_{a}}\calF_{j}\calF_{i}1_{\bfo_{a,b}}\cong {}_{R^{s_a}}R^{s_a}{}_{R^{s_a,s_{b}}},
\qquad
 1_{\bfo_{b}}\calF_{i}^{(2)}1_{\bfo_{a,b}}\cong {}_{R^{s_{b}}}R^{s_{b}}{}_{R^{s_a,s_{b}}}.
\]

In order to apply Elias's formula, we must find a 2-morphism which induces an isomorphism 
  $\calE_{i}\calE_{j}\calE_{i} 1_{\bfo}\cong
  \calE_{i}^{(2)}\calE_{j} 1_{\bfo}$, since this will give
  (iii).  
  Note that $\calE_{j} \calE_{i}^{(2)}1_{\bfo}=0$, since the flag
  variety corresponding to $(0,1,\dots, a+2,a+1,a+2,\dots)$ is empty.  The equation \eqref{now} shows that the 2-morphisms
\[
	B = -\mathord{
		\begin{tikzpicture}[baseline=0,xscale=1.3,yscale=0.6]
			\draw[<-,thick,darkred] (-1,1) -- node[below,at end,scale=1]{$i$} (1, -1);
			\draw[<-,thick,darkred] (0, 1) --node[pos=0.175] {$\color{darkred}\bullet$} node[below,at end,scale=1]{$i$} (-1, -1);
			\draw[->,thick, darkred] (0, -1) -- node[below,at start,scale=1]{$j$} (1,1);
		\end{tikzpicture}
					},	
\qquad
 C = \mathord{
			\begin{tikzpicture}[baseline=0,xscale=1.3,yscale=0.6]
				\draw[<-,thick,darkred] (-1,1) --  node[below,at end,scale=1]{$i$}  (0, -1);
				\draw[->,thick,darkred] (1, -1) -- node[below,at start,scale=1]{$j$} (0, 1);
				\draw[<-,thick, darkred] (1, 1) --  node[below,at end,scale=1]{$i$}  (-1,-1);
			\end{tikzpicture}
				},
\]
are inverse to each other and thus give the desired
isomorphism. Substituting these into \cite[Notation 6.4]{Elias} gives the
equations \eqref{eq:sl3} and \eqref{eq:sp4}. Hence Claim~($\star$) in the case of Step 1 is proved. 

\medskip
\noindent\emph{Step 2. The general case.}
\smallskip

Let us first assume $\bfb=\bfo$ and $\bfa$ is arbitrary. First, assume
that we have found an object $P\in\frakUj(\lambda(\bfo),\lambda(\bfa))$ such that $\Gamma(P)={}_{\bfa}R$ and such that for 
$Q\in\frakUj(\lambda(\bfa),\lambda(\bfo))$ the adjoint of $P$, the product $PQ$ has $1_{\lambda(\bfa)}$ as a direct factor  in degree zero. 
Let $\eta:1_{\lambda(\bfa)}\hookrightarrow PQ$ be a split embedding.
Then $\Gamma$ sends it to a split embedding $R^\bfa\cong
\Gamma(1_{\lambda(\bfa)})\to R\cong \Gamma(PQ)$, with the latter viewed as a $(R^\bfa,R^\bfa)$-bimodule. We have a commutative diagram
\[\xymatrix{
\Hom_{\frakUj}(MP,NP)\ar[d]^{\Gamma}\ar@{=}[r]
&\Hom_{\frakUj}(MPQ,N)\ar[d]^{\Gamma}\ar[r]^{(*)}
&\Hom_{\frakUj}(M,N)\ar[d]_{\Gamma} \\
\Hom_{\frakF_{r,m}}(\Gamma(MP),\Gamma(NP))\ar@{=}[r]
&\Hom_{\frakF_{r,m}}(\Gamma(M)\otimes_{R^\bfa}R,\Gamma(N))\ar[r]^{\quad\Gamma(*)}
&\Hom_{\frakF_{r,m}}(\Gamma(M),\Gamma(N)).}\]
Here $(*)$ is the transpose of the map $M\eta: M\to MPQ$. It is surjective, so is its image by $\Gamma$. 
Since $MP$ and $NP$ belong to $\frakUj(\bfo,\bfo)$, Step 1 implies
that the leftmost vertical map is surjective.  Hence, the rightmost
vertical map  is surjective too. 
The case for arbitrary $\bfb$ follows from the same argument by considering $\Hom_{\frakUj}(QMP, QNP)$ on the left top corner for the $Q$ defined for $\bfb$.

It remains to find $P$.  
Without loss of generality, we may assume 
\begin{equation}\label{eq:adom}
a_\hf<a_{\hf+1}<\cdots<a_{\hf+s}=a_{\hf+s+1}=\cdots=m,\qquad\text{for some }p.\end{equation}
Indeed, if there is any index $p$ such that
$a_{\hf+p-1} =a_{\hf+p}<a_{\hf+p+1}$, then 
let $\bfa'$ be the sequence with $a'_{\hf+p}=a_{\hf+p+1}$ and $a'_i=a_i$ for $i\neq \hf+p+1$.
We have a split surjection $\calF^{(a_{\hf+p+1}-a_{\hf+p})}_{\hf+p} 1_{\lambda(\bfa')} \, \calE^{(a_{\hf+p+1}-a_{\hf+p})}_{\hf+p} 1_{\lambda(\bfa)} \twoheadrightarrow 1_{\lambda(\bfa)}$, and hence we get the following commutative diagram 
\[
\xymatrix{
\Hom_{\frakUj}(M\calF^{(a_{\hf+p+1}-a_{\hf+p})}_{\hf+p}  ,N\calF^{(a_{\hf+p+1}-a_{\hf+p})}_{\hf+p}  )\ar[d]^{\Gamma}\ar@{->>}[r]
&\Hom_{\frakUj}(M,N)\ar[d]_{\Gamma} \\
\Hom_{\frakF_{r,m}}(\Gamma(M\calF^{(a_{\hf+p+1}-a_{\hf+p})}_{\hf+p} ),\Gamma(N\calF^{(a_{\hf+p+1}-a_{\hf+p})}_{\hf+p} ))\ar@{->>}[r]
&\Hom_{\frakF_{r,m}}(\Gamma(M),\Gamma(N)).}
\]
The surjectivity of the left vertical arrow implies the surjectivity of the right one.
Hence, up to replacing $\bfa$ by $\bfa'$ inductively, we are reduced to consider the sequence of
the form \eqref{eq:adom}. 

%

For such $\bfa$, we have $k\leqslant a_{\hf+k}$ for all $k$. We choose a sequence $\bfi=(i_1,i_2,...,i_n)$ such that the weight 
$\lambda^{(k)}=\lambda(\bfo)+\alpha_{i_1}+...+\alpha_{i_k}$ satisfies
$\lambda^{(n)}=\lambda(\bfa)$ and $\l\thal_{i_k}^\vee,\lambda^{(k)}\r>0$ for all $k>0$. 
To do so, we proceed inductively:  
set $\bfa^{(0)}=\bfo$ and define $\bfa^{(k)}={}_{+i_k}\bfa^{(k-1)}$ for all $k>0$.
Let $i_1$ be the largest element such that $i_1-\hf<a_{i_1}$. Given
$i_k$, we then  
\begin{itemize}
\item If the set $\{j <i_k\mid a^{(k)}_j< a_j\}$ is non-empty,
  then we let $i_{k+1}$ be the maximal element in this set.
\item Otherwise, we let $i_{k+1}$ be the largest element of $\{j \geq
  i_k\mid a^{(k)}_{j}< a_j\}$.
\end{itemize}
 Then for each $k$, we always have $a^{(k)}_{i_k+1}-a^{(k)}_{i_k}=1$ and 
$a^{(k)}_{i_k}-a^{(k)}_{i_k-1}\leqslant 2$. For $i_k=\hf$ this inequality follows from the fact that $a^{(k)}_{i_k}\geqslant 1$. 
Thus, we have that $\l\thal_{i_k}^\vee, \lambda^{(k)}\r\geqslant 1$ for all $k$.
We then define the mutually adjoint
1-morphisms: \[P=1_{\lambda(\bfa)}\calE_{i_n}...1_{\lambda^{(2)}}\calE_{i_2}1_{\lambda^{(1)}}\calE_{i_1}1_{\lambda(\bfo)}
\qquad 
Q=1_{\lambda(\bfo)}\calF_{i_1}1_{\lambda^{(1)}}\calF_{i_2}1_{\lambda^{(2)}}...\calF_{i_n}1_{\lambda(\bfa)}.\]
Since $\l\thal_{i_k}^\vee,\lambda^{(k)}\r\geqslant 1$ for all $k$, it follows from \eqref{bubble1} and \eqref{bubble3}
that 
$$\mathord{
\begin{tikzpicture}[baseline = 0]
	\draw[<-,thick,darkred] (0.4,-0.1) to[out=90, in=0] (0.1,0.3);
	\draw[-,thick,darkred] (0.1,0.3) to[out = 180, in = 90] (-0.2,-0.1);
    \node at (-0.2,-.2) {$\scriptstyle{i_k}$};
  \node at (0.3,0.5) {$\scriptstyle{\lambda}$};
\node at (0.27,0.25) {$\color{darkred}\bullet$};
   \node at (1.3,0.25) {$\color{darkred}\scriptstyle{\l\thal_{i_k}^\vee,\lambda^{(k)}\r-1}$};\end{tikzpicture}
}: \calE_i\calF_i1_\lambda\longrightarrow 1_\lambda$$
is split surjective. Hence $1_{\lambda^{(k)}}$  is a direct factor of $\calE_{i_k}\calF_{i_k}1_{\lambda^{(k)}}$ of degree zero. Applying this successively, we 
obtain that $1_{\lambda(\bfa)}$ is a direct factor of $PQ$ of degree zero. 

This completes the proof of Claim~($\star$) and hence the proof of 
Theorem~\ref{thm:full1}. 
\end{proof}

For an arbitrary $r$, there is an obvious problem with extending
Theorem~\ref{thm:full1}: the images of the bubbles under $\Gamma$
do not generate $R^{\bfa}$ unless $a_i=a_{i+1}$ for some $i$ or $a_{\hf}=0$.  We are
able to make the proof work for $r\geq m$, since one of these conditions will be
forced by the pigeonhole principle, but  for $r<m$, the obvious
extension of Theorem~\ref{thm:full1} fails.  
However, this failure to surject to $R^{\bfa}$ is easily fixed, and it
proves to be the only obstruction to fullness.

Fixing this issue requires us to add more 2-morphisms to
$\frakU^\jmath_r$, analogous to the extension of $\frakU$ discussed in \cite[Section 2.1]{Webcomparison}.
Let us write $\frakU^\jmath,\Gamma$ and $X_\jmath$ as $\frakU^\jmath_r,\Gamma^r$ and $X_{\jmath,r}$ to indicate the dependence on rank $r$. Note that for $r<s$, there is a natural embedding $\X_{\jmath,r}\to \X_{\jmath,s}$ given by sending the class of $\sum_{i=-r}^rk_i\vep_i$ to the class denoted by the same notation in $\X_{\jmath,s}$. We have a well-defined $2$-functor
$$\iota_{r,s}: \frakU^\jmath_r\longrightarrow \frakU^\jmath_s$$
which is given by the above embedding on objects, sending $\calE_i$, $\calF_i$ to themselves for $i\in\bbIj_r$, 
and sending the generators of $2$-morphisms in $\frakU^\jmath_r$ to the same diagram in $\frakU^\jmath_s$.
Now, for each $\lambda\in\X_{\jmath,r}$, consider the ring $Z_\lambda^{(s)}=\End_{\frakU^\jmath_s}(1_\lambda)$. 
Using arguments as in \cite[Proposition 3.6]{KLIII}, we can see that $Z_\lambda^{(s)}$ is generated by bubbles indexed by $i\in\bbIj_s$, 
and that for any pair of $1$-morphisms $M,N\in\frakU^\jmath_r$, the space 
$\Hom^\bullet_{\frakU^\jmath_s}(M,N)$ is generated by the image $\iota_{r,s} \big(\Hom^\bullet_{\frakU^\jmath_r}(M,N) \big)$ and $Z_\lambda^{(s)}$.
We define $\frakU^{\jmath,\ext}_r$ as the full $2$-subcategory of $\frakU^\jmath_{r+1}$ generated by objects and $1$-morphisms in the image of $\iota_{r,r+1}$.

Next, consider the embedding $j_{r,s}: \Sigma_{r,m}\hookrightarrow \Sigma_{s,m}$ which sends $\bfa=(a_1,\ldots,a_r)$ to $j_{r,s}(\bfa)=(a_1,\ldots ,a_r, m, m,\ldots)$.
This induces a tautological $2$-functor $\frakF_{r,m}\to\frakF_{s,m}$ which is fully faithful, since $R^{\bfa}=R^{j_{r,s}(\bfa)}$. 
Moreover $\Gamma$ intertwines $\iota_{r,s}$ with this $2$-functor. 
Hence we obtain a commutative diagram
\[\xymatrix{\Hom^\bullet_{\frakU^\jmath_r}(M,N)\ar[rr]^{\iota_{r,s}} \ar[rd]^{\Gamma_{M,N}^r}&&\Hom^\bullet_{\frakU^\jmath_s}(M,N)\ar[ld]_{\Gamma^s_{M,N}}\\
& \Hom^\bullet_{R^\bfb\bgmod R^\bfa}(\Gamma(M),\Gamma(N)).&}\]
Hence $\Gamma$ extends to a $2$-functor $\Gamma^{\ext}:\frakU^{\jmath,\ext}_r\to \frakF_{r,m}$, and we have a similar commutative diagram with $\frakU^\jmath_r$, $\Gamma_{M,N}^r$ on the left-downward arrow replaced by their extended version.
Now, if $s\geqslant m$, then the theorem implies that in the diagram above the map $\Gamma_{M,N}$ on the right is surjective. 
Further by formulas similar to \eqref{eq:bubble1}--\eqref{eq:bubble2}, 
$\Gamma$ sends positive degree bubbles labeled by $i$ to $0$ if $i>r+1$. 
Therefore we have established the following extension of Theorem~\ref{thm:full1}.

\begin{prop}
\label{prop:Fullgeneral}
The $2$-functor $\Gamma^{\ext}:\frakU^{\jmath,\ext}_r\to \frakF_{r,m}$ is locally full for any $r$ and $m$.
\end{prop}

Note that $\frakU^{\jmath,\ext}_r$ only differs from $\frakU^{\jmath}_r$ by bubbles labeled by $\hf+r$, 
which live in strictly positive degree. 
Hence the canonical embedding $\dot\frakU^\jmath_r\to \dot\frakU^{\jmath,\ext}_r$ sends indecomposable $1$-morphisms to indecomposable ones, 
and induces an isomorphism on Grothendieck groups (as argued in
\cite[Proposition 3]{Webcomparison}).
Applying \cite[Lemma~10]{Webcomparison}, we have established the
following corollary.

\begin{cor}\label{cor:CB-match}
  The map $K_0(\dot\frakU^\jmath)\to K_0(\frakF_{r,m})$ induced by $\Gamma$ sends the classes of indecomposable 1-morphisms in $\dot\frakU^\jmath$ 
  which are not annihilated by $\Gamma$ bijectively to the classes of
  indecomposable 1-morphisms in $\frakF_{r,m}$.
\end{cor}


%
%

\subsection {Canonical basis} 

Recall the functor $\aleph: {}_\calA\bdUj\to K_0(\dot\frakUj)$ defined in Lemma \ref{lem: gp}, and the equivalence $\kappa$ in Proposition \ref{prop:iFl-Grothendieck}. By construction, we have a commutative diagram
 \begin{equation}\label{eq:comm}
 \begin{split}
 \xymatrix{
 {}_\calA\bdUj  \ar[r]^{\aleph}\ar[d]_{\gamma}
 &K_0(\dot\frakUj)\ar[d]^{\Gamma}
 \\
 \bfS^\jmath_{r,m} 
 &K_0(\frakF_{r,m}).\ar[l]^{\kappa}_{\sim} }
 \end{split}
 \end{equation}
Recall also the canonical basis $\bfB^\jmath$ for ${}_\calA\bdUj$ defined in \cite[Proposition~5.4, Theorem~5.5]{LW15}. 
It can be characterized as the unique basis such that $\gamma$ maps each element in it to an element in the canonical basis of $\bfS^\jmath_{r,m}$ for infinitely many $m$.

\begin{thm}\label{thm:maindecat}
\mbox{}
  \begin{enumerate}
  \item The functor $\aleph: {}_\calA\bdUj\to K_0(\dot\frakUj)$ is an
    equivalence of categories.

\item     Assume the residue field of $\bfk$ has characteristic zero. Then
    $\aleph$ sends canonical basis of $ {}_\calA\bdUj$ to the classes
    of self-dual indecomposable $1$-morphisms in $\dot\frakUj$.
  \end{enumerate}
\end{thm}

\begin{proof}
To prove (a), note that $\aleph$ is identity on objects, and it is full by Proposition \ref{prop:gammasuj}. It remains to show it is faithful, i.e., the map $\aleph_{\lambda,\mu}: {}_\calA\bdUj(\lambda,\mu)\to K_0(\dot\frakUj(\lambda,\mu))$ is injective.
Note that if $u$ is in the kernel of $\aleph_{\lambda,\mu}$, by the commutativity of \eqref{eq:comm} and the fact that $\kappa$ is invertible, we deduce that $\gamma(u)=0$ for all $m\in\bbN$.
By \cite[Theorem~4.7]{BKLW}, this is possible only when $u=0$. We are done.
Part (b) follows from the characterization of 
$\bfB^\jmath$ and Corollary \ref{cor:CB-match}.
\end{proof}

As a corollary of Theorem~\ref{thm:maindecat} and Corollary \ref{cor:CB-match}, we obtain the following refinement of \cite[Proposition~5.11]{LW15}.
\begin{cor}
The map $\gamma:  {}_\calA\bdUj\to\bfS^\jmath_{r,m}$ in \eqref{eq:ga}  sends
$\bfB^\jmath\setminus (\bfB^\jmath\cap \ker \gamma)$ bijectively to the canonical basis of $\bfS^\jmath_{r,m}$ for all $m$.
\end{cor}

\medskip

\section{Categorical action on category $\calO$}
\label{sec:relat-proj-funct}

In this section, we assume $\bfk=\bbC$.

\subsection{Reminders on Harish-Chandra bimodules}

Let $\frakg$ be a complex semisimple Lie algebra and $\frakt\subset\frakg$ a Cartan subalgebra. Let $U=U(\frakg)$ be the enveloping algebra of $\frakg$ and let $Z$ be the center of $U$. 
We view the ring $R=S(\frakt)$ as the coordinate ring of  $\frakt^\ast$.
Recall that the Harish-Chandra morphism $\xi: Z\to R$ is a ring
homomorphism such that an element $z\in Z$ acts on a Verma module of
highest weight $\lambda$ by the value of $\xi(z)$ at $\lambda$. 

Consider the dot action of the Weyl group $\W$ on $\frakt^\ast$ by $w\bullet\lambda=w(\lambda+\rho)-\rho$, where $\rho$ is the half sum of positive roots. Then $\xi$ sends $Z$ isomorphically onto $W$-invariant functions on $\frakt^\ast$ for the dot action. Equivalently, the central characters of $Z$ are in bijection with $W$-dot-orbits in $\frakt^\ast$. 
For $\lambda\in\frakt^\ast$, denote the corresponding central
character by $\chi_\lambda$ and write $I_\lambda=\ker\chi_\lambda$.

We say that $\chi_\lambda$ is {\it integral} if $\lambda$ is an integral weight of $\frakg$.
Let $\widehat{Z}_\lambda$ be the completion of $Z$ at $\lambda$. Let $\lambda^\#:R\to R$ be the pull back of the translation map $\frakt^\ast\to\frakt^\ast$, $x\mapsto x+\lambda$. Then $\lambda^\#\circ\xi: Z\to R$ induces an isomorphism $\widehat{Z}_\lambda\simeq \widehat{R}^{W_\lambda}$, where $\widehat{R}$ is the completion of $R$ at the ideal generated by $\frakt$, and $W_\lambda$ is the stabilizer of $\lambda$ under the dot action.
We abbreviate $\widehat{R}^\lambda=\widehat{R}^{W_\lambda}$ and denote by $\frakm_\lambda$ the maximal ideal of $\widehat{R}^{W_\lambda}$.

Recall that a {\it Harish-Chandra bimodule} over $\frakg$ is a
$(U,U)$-bimodule for which the adjoint action of $\frakg$ is locally finite \cite{BG80, Soe92}.
Typical examples of Harish-Chandra modules are $E\otimes U$, where $E$
is a finite dimensional left $U$-module.
The left $U$-action on $E\otimes U$ is induced by the $\frakg$-action
such that $u(x\otimes y)=(ux)\otimes y+x\otimes (uy)$ for $u\in
\frakg$, and the right action is the right action of $U$ on itself. In other words, it is the tensor product of $U$ with $E$ considered as a
bimodule with the usual left action and trivial right action.

Let $\calH$ be the category of finitely generated Harish-Chandra bimodules of finite length. It is a direct sum of subcategories of the form 
$${}_\lambda\calH_{\mu}=\{ M\in\calH\,|\, I_\lambda^nM=MI_\mu^n=0\text{ for }n\gg 0\}$$
for $\lambda,\mu\in\frakt^\ast$. 
Let ${}_\lambda\calH_{\mu}^n$ be the full subcategory of
${}_\lambda\calH_{\mu}$ consisting of modules $M$ such that
$M I_\mu^n=0$. 
Let $\widehat{U}_\lambda$ be the completion of $U$ at $I_\lambda$. Let $
{}_\lambda\widehat{\calH}_{\mu}$ denote the category of finitely
generated $(\widehat{U}_\lambda, \widehat{U}_\mu)$-bimodules.

\begin{thm}[\mbox{\cite{Soe92,Str04}}]\label{thm:HC}\hfill 
  \begin{enumerate}
\renewcommand{\theenumi}{\alph{enumi}}
  \item There is a unique exact functor
$$\bbV: {}_\lambda\calH_{\mu}\longrightarrow  \widehat{R}^{\lambda}\operatorname{-mod-}\widehat{R}^{\mu}$$
defined by the property that it sends the unique simple module of
maximal Gelfand-Kirillov dimension to the trivial module $\bbC$, and
all the other simple modules to zero. For each $n\geqslant 0$, it induces an exact functor
$$\bbV: {}_\lambda\calH^n_{\mu}\longrightarrow  \widehat{R}^{\lambda}\operatorname{-mod-}(\widehat{R}^{\mu}/\frakm_\mu^n).
$$

\item For any $n>0$, the category ${}_\lambda\calH_\mu^n$ has
  enough projective objects, given by direct sums of direct summands
  of modules of the form $E\otimes U/(UI_\mu^n)$, where $E$ is a finite
  dimensional $\frakg$-module.

\item For any $n>0$, the functor $\bbV$ restricts to a fully faithful functor on the
  full additive subcategory ${}_\lambda\calP_\mu^n$ of projective objects in
  ${}_\lambda\calH_\mu^n$.
\end{enumerate}

\end{thm}

\begin{proof}
The definition of $\bbV$ is given in \cite[Page~357]{Str04},  Part (b) is proved in  \cite[Theorem~1.1]{Str04}, and  Part~ (c) follows from \cite[Theorem~4.1]{Str04}.
\end{proof}

Let $\calM$ be the category of (left) $U$-modules $M$ over which $Z$ acts locally finitely. It is a direct sum over $\mu\in\frakt^\ast$ of subcategories ${}_\mu \calM\subset\calM$ consisting of modules over which $I_\mu$ acts locally nilpotently. Denote by $\pr_\mu:\calM\to {}_\mu \calM$ the projection functor.
Let ${}^n_\mu\calM\subset{}_\mu\calM$ be the full subcategory consisting of modules over which
$I_\mu^n$ acts trivially.

Recall that a projective functor $\calM\to\calM$ is a direct summand of $E\otimes-$ for some finite dimensional $U$-module $E$. Direct sums and compositions of projective functors are again projective functors.
Given $E$ a finite dimensional representation, consider the projective functor
$$F: {}_\mu\calM\longrightarrow  {}_\lambda\calM,\quad M\mapsto \pr_\lambda(E\otimes M).$$
It is nontrivial if and only if $\lambda-\mu$ is a weight of $E$. 
The restriction of $F$ to the subcategory ${}^n_\mu\calM$ is represented by a direct factor ${}_\lambda\Phi(E)_\mu^n$ of $E\otimes U/U{I_\mu^n}$, which is an object of ${}_\lambda\calP_\mu^n$. 
For $m>n$, we have a projection ${}_\lambda\Phi(E)_\mu^m\twoheadrightarrow {}_\lambda\Phi(E)_\mu^n={}_\lambda\Phi(E)_\mu^m/I_\mu^n$.
Taking the limit yields an object ${}_\lambda\Phi(E)_\mu=\varprojlim{}_\lambda\Phi(E)_\mu^n$ in ${}_\lambda\widehat{\calH}_{\mu}$ which represents the functor $F$.
Let ${}_\lambda\widehat{\calP}_\mu$ be the full additive subcategory of ${}_\lambda\widehat{\calH}_{\mu}$ generated by direct summands of ${}_\lambda\Phi(E)_\mu$ for all $E$. 

\begin{cor}\label{cor:HC}
The functor $\bbV$ induces a fully faithful functor
$$\widehat{\bbV}: {}_\lambda\widehat{\calP}_\mu \longrightarrow \widehat{R}^{\lambda}\operatorname{-mod-}\widehat{R}^{\mu}.$$
\end{cor}
\begin{proof}
To see this, note that for any object $M$ in ${}_\lambda\widehat{\calP}_\mu$ we have $M=\varprojlim M/MI_\mu^n$. 
Applying $\bbV$ to the natural projection $M/MI_\mu^m\to M/MI_\mu^n$ yields a surjective map $\bbV(M/MI_\mu^m)\to \bbV(M/MI_\mu^n)$. 
Define $\widehat{\bbV}(M)$ as the limit of the projective system
$\{\bbV(M/MI_\mu^n) \}_{n\ge 0}$. It is still an object in $\widehat{R}^\lambda\operatorname{-mod-}\widehat{R}^\mu$. 
We have $\widehat{\bbV}(M)/\frakm_\mu^n=\bbV(M/MI_\mu^n)$.
Finally, for two objects $M$, $N$ in ${}_\lambda\widehat{\calP}_\mu$, we have
\begin{eqnarray*}
\Hom_{{}_\lambda\widehat{\calP}_\mu}(M,N)&=&\varprojlim\Hom_{\hat{\calH}}(M,N/NI_\mu^n)\\
&=&
\varprojlim\Hom_{{}_\lambda{\calH}_{\mu}^n}(M/MI_\mu^n,N/NI_\mu^n)\\
&=&
\varprojlim\Hom_{\widehat{R}^\lambda\otimes\widehat{R}^\mu}(\bbV(M/MI_\mu^n),\bbV(N/NI_\mu^n))\\
&=&
\varprojlim\Hom_{\widehat{R}^\lambda\otimes\widehat{R}^\mu}(\widehat{\bbV}(M)/\frakm_\mu^n,\widehat{\bbV}(N)/\frakm_\mu^n)\\
&=&\Hom_{\widehat{R}^\lambda\otimes\widehat{R}^\mu}(\widehat{\bbV}(M),\widehat{\bbV}(N)).
\end{eqnarray*}
where the first equality is the universal property of projective limit and the second one is because 
$I_\mu^n\subset Z$ is central, the third equality is given by Theorem \ref{thm:HC}(c).
\end{proof}

\begin{lemma}\label{lem:tensor}
For $M\in{}_\lambda\widehat{\calP}_\mu$ and $N\in{}_\mu\widehat{\calP}_\nu$
we have $M\otimes_{\widehat{U}_\mu}N\in{}_\lambda\widehat{\calP}_\nu.$
Moreover there is a natural isomorphism $\widehat{V}(M\otimes_{\widehat{U}_\mu}N)\cong\widehat{\bbV}(M)\otimes_{\widehat{R}^\mu}\widehat{\bbV}(N)$, 
which is functorial with respect to $M$ and $N$.
\end{lemma}

\begin{proof}
The first statement follows from the fact that composition of projective functors is again a projective functor. The proof of the second statement is similar to \cite[Proposition~13]{Soe92}, details are left to the reader.
\end{proof}

\medskip

\subsection{The case of types B and C}
Now, let us apply the above results to the type $\B/\C$ situation.
Let $\frakg=\mathfrak{so}_{2m+1}$ or $\frakg=\mathfrak{sp}_{2m}$. 
We choose $\bfa =(a_\hf, a_{1+\hf}, \ldots, a_{r-\hf}) \in
\Sigma_{r,m}$; by convention, we have $a_{r+\hf} =m$.  To this vector, 
we associate a weight $\mu\langle\bfa\rangle\in\frakt^\ast$ such that
$$\mu\langle\bfa\rangle+\rho=(0,...,0,-1,...,-1,....,-r,....,-r)$$
with $0$ appearing $a_\hf$ times, and 
$-k$ appearing $a_{k+\hf}-a_{k-\hf}$ times for $k>0$. 
By definition, we have that $W_{\mu\langle\bfa\rangle}$ is the subgroup of $W$
generated by the reflections in $I_\bfa\subset \S$ (see Section \ref{sec:type-b/c} for
notation). Thus, we have $R^{W_{\mu\langle\bfa\rangle}}=R^\bfa$.
Note that if $\frakg=\mathfrak{sp}_{2m}$ then every integral central character 
is of this form for $r$ sufficiently large.  
For $\frakg=\frak{so}_{2m+1}$, we have $\rho \in \sum_{a=1}^m
(\bbZ+\frac 12) \epsilon_a$, and so the highest weights for modules with
central character $\chi(\bfa)$ lie in $\sum_{a=1}^m (\bbZ+\frac 12)
\epsilon_a$; in particular, the block of the spin representation lies
in this image and the block of the trivial representation does not.

The weights of the natural representation $V$ are $\pm\epsilon_k$ for $k=1,\dots,m$ for $\frakg=\mathfrak{sp}_{2m}$
(and $0$ in addition for $\frakg=\frakso_{2m+1}$).
Hence we have ${}_{\lambda}\Phi(V)_{\mu\langle\bfa\rangle}$ is nontrivial
precisely when $\lambda=\mu\langle{}_{\pm i}\bfa\rangle$ for some $i\in\bbI^\jmath$  (for
${}_{\pm i}\bfa$ as in
Section~\ref{sec:type-b/c}), or $\lam=\mu\langle\bfa\rangle$ if $\frakg=\frakso_{2m+1}$.
Recall the bimodules $\scrE_i1_\bfa$, $\scrF_i1_\bfa$ from Section
\ref{sec:type-b/c}, and let $\widehat{\scrE}_i1_\bfa$,
$\widehat{\scrF}_i1_\bfa$ be their completions with respect to their grading.
\begin{lemma}
We have
\begin{equation}\label{eq:image}
\widehat{\bbV}({}_{\mu\langle{}_{+i}\bfa\rangle}\Phi(V)_{\mu\langle\bfa\rangle})=\widehat{\scrE}_i1_\bfa,\qquad
\widehat{\bbV}({}_{\mu\langle{}_{-i}\bfa\rangle}\Phi(V)_{\mu\langle\bfa\rangle})=\widehat{\scrF}_i1_\bfa.
\end{equation}
\end{lemma}
\begin{proof}
The image $\widehat{\bbV}({}_{\mu\langle{}_{+i}\bfa\rangle}\Phi(V)_{\mu\langle\bfa\rangle})$ is a
completed singular Soergel bimodule.  Its rank as a left module over
$\widehat{R}^{{}_{+i}\bfa}$ is the number $\ell$ of weights $\nu$ of $V$ such
that $\mu\langle\bfa\rangle+\nu$ is in the dot-orbit of $\mu\langle{}_{+i}\bfa\rangle$.
This is precisely $a_i-a_{i-1}$ if $i>\hf$ and $2a_{\hf}$ if
$i=\hf$.  

We claim that $\widehat{\scrE}_i1_\bfa$ is the only completed singular Soergel
bimodule with this property.  By \cite[Theorem 1]{WilSSB}, the
indecomposable singular Soergel bimodules are in bijection with the
cosets $W_{{}_{+i}\bfa}\backslash
W/W_\bfa$.  By \cite[Section 7.5]{WilSSB}, the rank 
of the bimodule associated to a longest double coset representative
$w$ as a free left $\widehat{R}^{{}_{+i}\bfa}$ module
is  the sum $\sum_{u\in
  W/W_{\bfa}}p^{\bfa}_{w,u}(1)$ where $p^{\bfa}_{w,u}(x)$ is the
parabolic Kazhdan-Lusztig polynomial. The constant term of $
p^{\bfa}_{w,u}(x)$ is $1$, and all its coefficients are non-negative by
\cite[3.11, 4.1]{Deodhar}, so $p^{\bfa}_{w,u}(1)\geq 1$ whenever
$u W_\bfa\leq w W_\bfa$ in Bruhat order.  Thus, the rank of an indecomposable
singular Soergel bimodule over $\widehat{R}^{{}_{+i}\bfa}$
corresponding to $w$ is at least the number of left cosets less than
$w W_\bfa$ in Bruhat order.  This number is $\ell$ for the double coset of the
identity, and thus $>\ell$ for any other coset.  

Thus, $\widehat{\scrE}_i1_\bfa$ is the unique singular Soergel
bimodule with rank $\ell$, which  shows the equality
$\widehat{\bbV}({}_{\mu\langle{}_{+i}\bfa\rangle}\Phi(V)_{\mu\langle\bfa\rangle})\cong
\widehat{\scrE}_i1_\bfa$.  The same argument with left and right hand actions reversed shows
that $\widehat{\bbV}({}_{\mu\langle{}_{-i}\bfa\rangle}\Phi(V)_{\mu\langle\bfa\rangle})\cong \widehat{\scrF}_i1_\bfa$.
\end{proof}

For $\mu=\mu\langle\bfa\rangle$, $\lambda=\mu\langle\bfb\rangle$, recall that by Proposition \ref{prop:iFl-Grothendieck}(a) the category
$\frakF(\bfa,\bfb)$ defined in Section \ref{sec:type-b/c} is a full subcategory in $R^\bfa\operatorname{-mod-}R^\bfb$, hence we have an obvious completion functor $\ \widehat{}: \frakF(\bfa,\bfb)\to \widehat{R}^\bfa\operatorname{-mod-}\widehat{R}^\bfb$.
\begin{lemma}
 There exists a unique functor $\bbU: \frakF(\bfa,\bfb)\to {}_\lambda\widehat{\calP}_\mu$ such that $\widehat{\bbV}\circ\bbU: \frakF(\bfa,\bfb)\to \widehat{R}^\bfa\operatorname{-mod-}\widehat{R}^\bfb$ is the completion functor.
\end{lemma}
\begin{proof}
By Corollary \ref{cor:HC}, the functor $\widehat{\bbV}$ induces an equivalence
between ${}_\lambda\widehat{\calP}_\mu$ and its image. Hence to prove the lemma,
it is enough to prove that the image of $\frakF(\bfa,\bfb)$ under the
completion functor lands in $\bbV({}_\lambda\widehat{\calP}_\mu)$. This is a
consequence of \eqref{eq:image} and Lemma \ref{lem:tensor}.
\end{proof}

This shows that the 2-category $\frakF_{r,m}$ has $2$-representations given by
sending $\bfa$ to the category ${}_{\mu\langle\bfa\rangle}\mathcal{M}$, or more
generally, to any subcategory of ${}_{\mu\langle\bfa\rangle}\mathcal{M}$ closed
under the action of projective functors.  Composing with $\Gamma: \frakU^\jmath_r
\rightarrow  \frakF_{r,m}$ from \eqref{eq:Gamma}, we obtain the
following theorem. 

\begin{thm}
\label{thm:actionO}
  The category $\frakU^\jmath_r$ has a representation sending
  $\lambda$ in $\X_\jmath$ to the category ${}_{\mu\langle\bfa\rangle}\mathcal{M}$ if $\lambda=\lambda(\bfa)$ for $\bfa\in\Sigma_{r,m}$, and to zero otherwise.  This
  action descends to the intersection of ${}_{\mu\langle\bfa\rangle}\mathcal{M}$
  with any subcategory of $U\mod$ which is closed under the action of projective functors, including:
  \begin{itemize}
\item the subcategory of finite length $U$-modules,
  \item the subcategory of finite dimensional $U$-modules,
\item  the subcategory of modules locally finite for a subalgebra $\mathfrak{k}\subset \frakg$,
\item the BGG category $\calO$ and its parabolic generalizations $\calO^{\frakp}$,
\item the subcategory of projective-injective modules in these categories.
  \end{itemize}
\end{thm}

Tracing through the definitions, this action sends the 1-morphisms $\calE_{\pm i}$
to the translation functor from the block ${}_{\mu\langle\bfa\rangle}\mathcal{M}$
to  ${}_{{}_{\pm i}\mu\langle\bfa\rangle}\mathcal{M}$.  
By \eqref{eq:image}, these are always given by summands of the functor $V\otimes-$ of tensor product with the defining representation.
In future work, we will expand the discussion of this categorical action in greater detail and describe further applications.

\medskip

\appendix

\section{Categorification of the $\jmath$-Serre relations}\label{sec:jserre}

In this appendix we shall derive some bubble slide formulas.
We then prove the identities \eqref{eq:prop:jserre1} and \eqref{eq:prop:jserre2}.
This completes the proof of Proposition~\ref{prop:jserre} on the categorification of the $\jmath$-Serre relations.

\subsection{Bubble slides} 

We first provide several bubble slide formulas for $\frakU^\jmath$, which are the counterparts of Lauda's bubble slide formulas \cite[Propositions~5.6, 5.7]{L}.
We recall our convention that all the strands without labels should be viewed as labeled by $\hf$. 

\begin{lemma}[Bubble slides] \label{lem:bubbleslides}
The following relations hold for all $\lambda \in \X_\jmath$ (recall $ \lambda_{\hf} = \langle \thal_\hf^\vee, \lambda \rangle$):
\begin{align}
	\mathord{
		\begin{tikzpicture}[baseline= -4]
 			\draw[-,thick,darkred] (.2,0.2) to[out=180,in=90] (0,0);
			 \draw[-,thick,darkred] (0,0) to[out=-90,in=180]  (.2,-0.2);
 			 \draw[->,thick,darkred] (0.4,0) to[out=90,in=0] (.2,.2);
  			 \draw[-,thick,darkred] (.2,-0.2) to[out=0,in=-90] node[label={[shift={(0,-0.75)}] $\scriptstyle -\lambda_{\hf}+\alpha-3$}] {$\bullet$} (0.4,0);			
  			 \node at (0.8,0.05) {$\scriptstyle{\lambda}$};
			\draw[<-,thick,darkred](-.4,.4) to node  {$\bullet$} (-.4,-.4);
		\end{tikzpicture}
			}
		+
			\mathord{
		\begin{tikzpicture}[baseline= -4]
 			\draw[-,thick,darkred] (.2,0.2) to[out=180,in=90] (0,0);
 			 \draw[->,thick,darkred] (0.4,0) to[out=90,in=0] (.2,.2);
 			\draw[-,thick,darkred] (0,0) to[out=-90,in=180]   (.2,-0.2);
  			\draw[-,thick,darkred] (.2,-0.2) to[out=0,in=-90] node[label={[shift={(0,-0.75)}] $\scriptstyle  -\lambda_{\hf}-2+\alpha$}] {$\bullet$} (0.4,0);			
  			 \node at (0.8,0.05) {$\scriptstyle{\lambda}$};
			\draw[<-,thick,darkred](-.4,.4) to (-.4,-.4);
		\end{tikzpicture}
			}
	&	= 
		\sum^{\alpha}_{s = 0} (s+1)  \!\!\!\!
			\mathord{
		\begin{tikzpicture}[baseline=-4]
 			\draw[-,thick,darkred] (.2,0.2) to[out=180,in=90] (0,0);
 			 \draw[->,thick,darkred] (0.4,0) to[out=90,in=0] (.2,.2);
 			\draw[-,thick,darkred] (0,0) to[out=-90,in=180]  node[label={[shift={(-0.2,-0.75)}] $\scriptstyle -\lambda_{\hf}-5+\alpha-s$}] {$\bullet$}  (.2,-0.2);
  			\draw[-,thick,darkred] (.2,-0.2) to[out=0,in=-90]  (0.4,0);  
  			 \node at (1.2,0.05) {$\scriptstyle{\lambda}$};
			\draw[<-,thick,darkred](0.8,.4) to node[label={[shift={(0.2, -0.25)}] $\scriptstyle s$}] {$\bullet$} (0.8,-.4);
		\end{tikzpicture}
			}; \label{eq:bubbleslides1}\\
	\mathord{
		\begin{tikzpicture}[baseline= -4]
 			\draw[-,thick,darkred] (.2,0.2) to[out=180,in=90] (0,0);
			 \draw[-,thick,darkred] (0,0) to[out=-90,in=180] node[label={[shift={(0,-0.75)}] $\scriptstyle \lambda_{\hf} +\alpha + 1$}] {$\bullet$} (.2,-0.2);
 			 \draw[<-,thick,darkred] (0.4,0) to[out=90,in=0] (.2,.2);
  			 \draw[-,thick,darkred] (.2,-0.2) to[out=0,in=-90] (0.4,0);			
  			 \node at (1.2,0.05) {$\scriptstyle{\lambda}$};
			\draw[<-,thick,darkred](0.8,.4) to node {$\bullet$} (0.8,-.4);
		\end{tikzpicture}
			}
		+
			\mathord{
		\begin{tikzpicture}[baseline= -4]
 			\draw[-,thick,darkred] (.2,0.2) to[out=180,in=90] (0,0);
 			 \draw[<-,thick,darkred] (0.4,0) to[out=90,in=0] (.2,.2);
 			\draw[-,thick,darkred] (0,0) to[out=-90,in=180]  node[label={[shift={(0,-0.75)}] $\scriptstyle \lambda_{\hf}+2+\alpha$}] {$\bullet$} (.2,-0.2);
  			\draw[-,thick,darkred] (.2,-0.2) to[out=0,in=-90] (0.4,0);			
  			 \node at (1.2,0.05) {$\scriptstyle{\lambda}$};
			\draw[<-,thick,darkred](0.8,.4) to (0.8,-.4);
		\end{tikzpicture}
			}
		& = 
		\sum^{\alpha}_{s = 0} (s+1)
			\mathord{
		\begin{tikzpicture}[baseline= -4]
 			\draw[-,thick,darkred] (.2,0.2) to[out=180,in=90] (0,0);
 			 \draw[<-,thick,darkred] (0.4,0) to[out=90,in=0] (.2,.2);
 			\draw[-,thick,darkred] (0,0) to[out=-90,in=180] (.2,-0.2);
  			\draw[-,thick,darkred] (.2,-0.2) to[out=0,in=-90] node[label={[shift={(0.1,-0.75)}] $\scriptstyle \lambda_{\hf}-1+\alpha-s$}] {$\bullet$}  (0.4,0);			
  			 \node at (0.8,0.05) {$\scriptstyle{\lambda}$};
			\draw[<-,thick,darkred] (-0.4,.4) to node[label={[shift={(0.2, -0.25)}] $\scriptstyle s$}] {$\bullet$} (-0.4,-.4);
		\end{tikzpicture}
			}; \label{eq:bubbleslides2}
\end{align}
\begin{align}
 \mathord{
		\begin{tikzpicture}[baseline= -4]
 			\draw[-,thick,darkred] (.2,0.2) to[out=180,in=90] (0,0);
			 \draw[-,thick,darkred] (0,0) to[out=-90,in=180]  (.2,-0.2);
 			 \draw[<-,thick,darkred] (0.4,0) to[out=90,in=0] (.2,.2);
  			 \draw[-,thick,darkred] (.2,-0.2) to[out=0,in=-90] node[label={[shift={(0,-0.75)}] $\scriptstyle \lambda_{\hf}-1+\alpha$}] {$\bullet$} (0.4,0);			
  			 \node at (0.8,0.05) {$\scriptstyle{\lambda}$};
			\draw[<-,thick,darkred](-.4,.4) to   (-.4,-.4);
		\end{tikzpicture}
			} 
			&=
			\mathord{
		\begin{tikzpicture}[baseline= -4]
 			\draw[-,thick,darkred] (.2,0.2) to[out=180,in=90] (0,0);
 			 \draw[<-,thick,darkred] (0.4,0) to[out=90,in=0] (.2,.2);
 			\draw[-,thick,darkred] (0,0) to[out=-90,in=180] node[label={[shift={(0,-0.75)}] $\scriptstyle \lambda_{\hf} +\alpha-1$}] {$\bullet$} (.2,-0.2);
  			\draw[-,thick,darkred] (.2,-0.2) to[out=0,in=-90]   (0.4,0);			
  			 \node at (1.2,0.05) {$\scriptstyle{\lambda}$};
			\draw[<-,thick,darkred](0.8,.4) to node[label={[shift={(0.2, -0.25)}] $\scriptstyle 3$}] {$\bullet$} (0.8,-.4);
		\end{tikzpicture}
			}
-
		\mathord{
		\begin{tikzpicture}[baseline= -4]
 			\draw[-,thick,darkred] (.2,0.2) to[out=180,in=90] (0,0);
 			 \draw[<-,thick,darkred] (0.4,0) to[out=90,in=0]  (.2,.2);
 			\draw[-,thick,darkred] (0,0) to[out=-90,in=180] node[label={[shift={(0,-0.75)}] $\scriptstyle \lambda_{\hf} +\alpha $}] {$\bullet$} (.2,-0.2);
  			\draw[-,thick,darkred] (.2,-0.2) to[out=0,in=-90]  (0.4,0);			
  			 \node at (1.2,0.05) {$\scriptstyle{\lambda}$};
			\draw[<-,thick,darkred](0.8,.4) to node[label={[shift={(0.2, -0.25)}] $\scriptstyle 2$}] {$\bullet$} (0.8,-.4);
		\end{tikzpicture}
			}
-
		\mathord{
		\begin{tikzpicture}[baseline= -4]
 			\draw[-,thick,darkred] (.2,0.2) to[out=180,in=90] (0,0);
 			 \draw[<-,thick,darkred] (0.4,0) to[out=90,in=0]  (.2,.2);
 			\draw[-,thick,darkred] (0,0) to[out=-90,in=180] node[label={[shift={(0,-0.75)}] $\scriptstyle \lambda_{\hf} +\alpha +1$}] {$\bullet$} (.2,-0.2);
  			\draw[-,thick,darkred] (.2,-0.2) to[out=0,in=-90]  (0.4,0);			
  			 \node at (1.2,0.05) {$\scriptstyle{\lambda}$};
			\draw[<-,thick,darkred](0.8,.4) to node  {$\bullet$} (0.8,-.4);
		\end{tikzpicture}
			}
+
		\mathord{
		\begin{tikzpicture}[baseline= -4]
 			\draw[-,thick,darkred] (.2,0.2) to[out=180,in=90] (0,0);
 			 \draw[<-,thick,darkred] (0.4,0) to[out=90,in=0]  (.2,.2);
 			\draw[-,thick,darkred] (0,0) to[out=-90,in=180] node[label={[shift={(0,-0.75)}] $\scriptstyle \lambda_{\hf}+2+\alpha$}] {$\bullet$} (.2,-0.2);
  			\draw[-,thick,darkred] (.2,-0.2) to[out=0,in=-90]  (0.4,0);			
  			 \node at (1.2,0.05) {$\scriptstyle{\lambda}$};
			\draw[<-,thick,darkred](0.8,.4) to  (0.8,-.4);
		\end{tikzpicture}
			}; \label{eq:bubbleslides3} \\
\mathord{
		\begin{tikzpicture}[baseline= -4]
 			\draw[-,thick,darkred] (.2,0.2) to[out=180,in=90] (0,0);
 			 \draw[->,thick,darkred] (0.4,0) to[out=90,in=0]  (.2,.2);
 			\draw[-,thick,darkred] (0,0) to[out=-90,in=180] node[label={[shift={(0,-0.75)}] $\scriptstyle -\lambda_{\hf}-5+\alpha$}] {$\bullet$} (.2,-0.2);
  			\draw[-,thick,darkred] (.2,-0.2) to[out=0,in=-90]  (0.4,0);			
  			 \node at (1.2,0.05) {$\scriptstyle{\lambda}$};
			\draw[<-,thick,darkred](0.8,.4) to  (0.8,-.4);
		\end{tikzpicture}
			}
&=
		\mathord{
		\begin{tikzpicture}[baseline= -4]
 			\draw[-,thick,darkred] (.2,0.2) to[out=180,in=90] (0,0);
			 \draw[-,thick,darkred] (0,0) to[out=-90,in=180]  (.2,-0.2);
 			 \draw[->,thick,darkred] (0.4,0) to[out=90,in=0] (.2,.2);
  			 \draw[-,thick,darkred] (.2,-0.2) to[out=0,in=-90] node[label={[shift={(0,-0.75)}] $\scriptstyle -\lambda_{\hf} +\alpha-5$}] {$\bullet$} (0.4,0);			
  			 \node at (0.8,0.05) {$\scriptstyle{\lambda}$};
			\draw[<-,thick,darkred](-.4,.4) to  node[label={[shift={(-0.2, -0.25)}] $\scriptstyle 3$}] {$\bullet$} (-.4,-.4);
		\end{tikzpicture}
			} 
-
		\mathord{
		\begin{tikzpicture}[baseline= -4]
 			\draw[-,thick,darkred] (.2,0.2) to[out=180,in=90] (0,0);
			 \draw[-,thick,darkred] (0,0) to[out=-90,in=180]  (.2,-0.2);
 			 \draw[->,thick,darkred] (0.4,0) to[out=90,in=0] (.2,.2);
  			 \draw[-,thick,darkred] (.2,-0.2) to[out=0,in=-90] node[label={[shift={(0,-0.75)}] $\scriptstyle -\lambda_{\hf} +\alpha-4$}] {$\bullet$} (0.4,0);			
  			 \node at (0.8,0.05) {$\scriptstyle{\lambda}$};
			\draw[<-,thick,darkred](-.4,.4) to node[label={[shift={(-0.2, -0.25)}] $\scriptstyle 2$}] {$\bullet$}  (-.4,-.4);
		\end{tikzpicture}
			} 
-
		\mathord{
		\begin{tikzpicture}[baseline= -4]
 			\draw[-,thick,darkred] (.2,0.2) to[out=180,in=90] (0,0);
			 \draw[-,thick,darkred] (0,0) to[out=-90,in=180]  (.2,-0.2);
 			 \draw[->,thick,darkred] (0.4,0) to[out=90,in=0] (.2,.2);
  			 \draw[-,thick,darkred] (.2,-0.2) to[out=0,in=-90] node[label={[shift={(0,-0.75)}] $\scriptstyle -\lambda_{\hf} +\alpha-3$}] {$\bullet$} (0.4,0);			
  			 \node at (0.8,0.05) {$\scriptstyle{\lambda}$};
			\draw[<-,thick,darkred](-.4,.4) to  node  {$\bullet$} (-.4,-.4);
		\end{tikzpicture}
			} 
+
		\mathord{
		\begin{tikzpicture}[baseline= -4]
 			\draw[-,thick,darkred] (.2,0.2) to[out=180,in=90] (0,0);
			 \draw[-,thick,darkred] (0,0) to[out=-90,in=180]  (.2,-0.2);
 			 \draw[->,thick,darkred] (0.4,0) to[out=90,in=0] (.2,.2);
  			 \draw[-,thick,darkred] (.2,-0.2) to[out=0,in=-90] node[label={[shift={(0,-0.75)}] $\scriptstyle -\lambda_{\hf} +\alpha-2$}] {$\bullet$} (0.4,0);			
  			 \node at (0.8,0.05) {$\scriptstyle{\lambda}$};
			\draw[<-,thick,darkred](-.4,.4) to   (-.4,-.4);
		\end{tikzpicture}
			} \label{eq:bubbleslides4}.
\end{align}
\end{lemma}

\begin{proof}
We shall first prove the case when $\lambda_{\hf} \ge -2$. Note that in this case only real bubbles appear on the left hand side of the equation \eqref{eq:bubbleslides2}, since bubbles with negative degree are set to be $0$. We shall proceed in the following steps:

\begin{itemize}
	\item	[(a)]		Prove \eqref{eq:bubbleslides1} for $\alpha \ge \lambda_{\hf} +3$, and prove \eqref{eq:bubbleslides2} for all $\alpha$;
	\item	[(b)]	Prove \eqref{eq:bubbleslides3} by using \eqref{eq:bubbleslides2};
	\item	[(c)]	To complete the proof of \eqref{eq:bubbleslides1} for $0 \le \alpha \le \lambda_{\hf}+2$ (fake bubbles), we use induction on $\alpha$ based on the definition of fake bubbles, with the help of \eqref{eq:bubbleslides3}.
	\item	[(d)]	Prove \eqref{eq:bubbleslides4} by using \eqref{eq:bubbleslides1} (similar to (b)).
\end{itemize}

	We now proceed with Step (a). \\
From the relation \eqref{j-bicross down}, we have (for $m \ge 0$)
\begin{align*}
\mathord{

			}.
\end{align*}
Replacing $m+\lambda_{\hf}+3 =\alpha$ above, we have proved \eqref{eq:bubbleslides1} for $\alpha \ge \lambda_{\hf}+3$ (thanks to $m \ge 0$).

In an entirely similar way, we can prove (for $\alpha \ge -1 -\lambda_{\hf}$) : 
\[
\mathord{
		%
			}.
\]
On the other hand, if $\alpha \le -2-\lambda_{\hf}$, then $\alpha \le 0$ by the assumption in (a). If $\alpha<0$,
both left and right hand sides of the identity~ \eqref{eq:bubbleslides2} are $0$. If $\alpha=0$ (which is non-trivial only when $\lambda_\hf=-2$), then both sides are equal to $2 {\color{darkred}{\displaystyle\uparrow}} \lambda$. So we have established  \eqref{eq:bubbleslides2} for all $\alpha$.
\\

Let us proceed with Step (b).  We can rewrite the right hand side of \eqref{eq:bubbleslides3} as follows:
\begin{align*}
&	\mathord{
		%
 \Big)
			}.
\end{align*}
Then by applying \eqref{eq:bubbleslides2} three times, we obtain the identity \eqref{eq:bubbleslides3}. 
\\

Now we proceed with Step (c), that is, we prove \eqref{eq:bubbleslides1} for $0 \le \alpha \le \lambda_{\hf} +2$ by induction on $\alpha$. 
Thanks to the bubble relations \eqref{bubble1}--\eqref{bubble3}, we have 
\[
\mathord{
%
} = 2,
\]
and a recursive definition of fake bubbles as follows, for $\alpha>0$:
\begin{equation}\label{eq:deffake}
2 
\mathord{
%
}.
\end{equation}
The base case  $\alpha=0$ of   \eqref{eq:bubbleslides1}  is trivial, since both sides of identity \eqref{eq:bubbleslides1} are $\, \,\mathord{
		%
			}$. 
The case  $\alpha=1$ of   \eqref{eq:bubbleslides1} can be verified
using \eqref{eq:bubbleslides3} by writing the fake bubbles in terms of the real ones using \eqref{eq:deffake}. 
For $\alpha > 1$, it follows by \eqref{eq:deffake}, the inductive assumption, and \eqref{eq:bubbleslides3} that
\begin{align*}
2 & \Big(
 	\mathord{
		%
				}.
\end{align*}
To simplify the right-hand side above, we will compute 4 sums (each of 2 summands which line up vertically) as follows.
First, by the induction hypothesis and using the relation  \eqref{bubble3}, we have 
\begin{align*}
& \sum^{\alpha}_{l=1} \Big(  	
	-\mathord{
		%
}.  
\end{align*}

Similarly we have 
\begin{align*}
 \sum^{\alpha}_{l=1} \Big(  	
	\mathord{
		%
}.  
\end{align*}

Finally,  by the definition of fake bubbles in \eqref{eq:deffake} we have
\begin{align*}
& \sum^{\alpha}_{l=1} \Big(  	
	-\mathord{
		%
			}.  
\end{align*}
Summing up the above 4 identities, we finish the proof of  \eqref{eq:bubbleslides1}. 

The proof of Step (d) that the identity \eqref{eq:bubbleslides4} follows from \eqref{eq:bubbleslides1} 
is entirely similar to the proof of Step~(b) that the \eqref{eq:bubbleslides3} follows by  \eqref{eq:bubbleslides2}. We skip the detail. 
\\

This finishes the proof of the bubble slide formulas for the case $\lambda_{\hf} \ge -2$. 
The case $\lambda_{\hf} \le -3$ is entirely similar, where the left hand side  of the identity \eqref{eq:bubbleslides1} involves only real bubble, hence can be easily proved, as well as the identity \eqref{eq:bubbleslides4}. The proof of the identities  \eqref{eq:bubbleslides2} and  \eqref{eq:bubbleslides3} follows a similar argument as above. 
Alternatively, we could simply apply symmetries of $\frakUj$ in Section~\ref{sec:symmetries}.
\end{proof}

\begin{cor}
The following relations hold for all $\lambda \in \X_\jmath$:
\begin{align*}
			\mathord{
		\begin{tikzpicture}[baseline= -4]
 			\draw[-,thick,darkred] (.2,0.2) to[out=180,in=90] (0,0);
 			 \draw[->,thick,darkred] (0.4,0) to[out=90,in=0] (.2,.2);
 			\draw[-,thick,darkred] (0,0) to[out=-90,in=180]   (.2,-0.2);
  			\draw[-,thick,darkred] (.2,-0.2) to[out=0,in=-90] node[label={[shift={(0,-0.75)}] $\scriptstyle  -\lambda_{\hf}-2+\alpha$}] {$\bullet$} (0.4,0);			
  			 \node at (0.8,0.05) {$\scriptstyle{\lambda}$};
			\draw[<-,thick,darkred](-.4,.4) to (-.4,-.4);
		\end{tikzpicture}
			}
	&	= 
		\sum^{\alpha}_{s = 0} \left \lceil \frac{s+1}{2} \right \rceil  \!\!\!\!
			\mathord{
		\begin{tikzpicture}[baseline=-4]
 			\draw[-,thick,darkred] (.2,0.2) to[out=180,in=90] (0,0);
 			 \draw[->,thick,darkred] (0.4,0) to[out=90,in=0] (.2,.2);
 			\draw[-,thick,darkred] (0,0) to[out=-90,in=180]  node[label={[shift={(-0.2,-0.75)}] $\scriptstyle -\lambda_{\hf}-5+\alpha-s$}] {$\bullet$}  (.2,-0.2);
  			\draw[-,thick,darkred] (.2,-0.2) to[out=0,in=-90]  (0.4,0);  
  			 \node at (1.2,0.05) {$\scriptstyle{\lambda}$};
			\draw[<-,thick,darkred](0.8,.4) to node[label={[shift={(0.2, -0.25)}] $\scriptstyle s$}] {$\bullet$} (0.8,-.4);
		\end{tikzpicture}
			};  \\
			\mathord{
		\begin{tikzpicture}[baseline= -4]
 			\draw[-,thick,darkred] (.2,0.2) to[out=180,in=90] (0,0);
 			 \draw[<-,thick,darkred] (0.4,0) to[out=90,in=0] (.2,.2);
 			\draw[-,thick,darkred] (0,0) to[out=-90,in=180]  node[label={[shift={(0,-0.75)}] $\scriptstyle \lambda_{\hf}+2+\alpha$}] {$\bullet$} (.2,-0.2);
  			\draw[-,thick,darkred] (.2,-0.2) to[out=0,in=-90] (0.4,0);			
  			 \node at (1.2,0.05) {$\scriptstyle{\lambda}$};
			\draw[<-,thick,darkred](0.8,.4) to (0.8,-.4);
		\end{tikzpicture}
			}
		& = 
		\sum^{\alpha}_{s = 0} \left\lceil \frac{s+1}{2} \right\rceil
			\mathord{
		\begin{tikzpicture}[baseline= -4]
 			\draw[-,thick,darkred] (.2,0.2) to[out=180,in=90] (0,0);
 			 \draw[<-,thick,darkred] (0.4,0) to[out=90,in=0] (.2,.2);
 			\draw[-,thick,darkred] (0,0) to[out=-90,in=180] (.2,-0.2);
  			\draw[-,thick,darkred] (.2,-0.2) to[out=0,in=-90] node[label={[shift={(0.1,-0.75)}] $\scriptstyle \lambda_{\hf}-1+\alpha-s$}] {$\bullet$}  (0.4,0);			
  			 \node at (0.8,0.05) {$\scriptstyle{\lambda}$};
			\draw[<-,thick,darkred] (-0.4,.4) to node[label={[shift={(0.2, -0.25)}] $\scriptstyle s$}] {$\bullet$} (-0.4,-.4);
		\end{tikzpicture}
			}.
\end{align*}
\end{cor}

\begin{proof}
The proof follows from Lemma~\ref{lem:bubbleslides} by induction on $\alpha$.
\end{proof}

We can also formulate the downward arrow counterparts of the bubble slide relations in Lemma~\ref{lem:bubbleslides}. The proof is similar, hence shall be omitted. 

\begin{prop}\label{prop:down:bubbleslides}
The following relations hold for all $\lambda \in \X_\jmath$:
\begin{align*}
	\mathord{
		\begin{tikzpicture}[baseline= -4]
 			\draw[-,thick,darkred] (.2,0.2) to[out=180,in=90] (0,0);
			 \draw[-,thick,darkred] (0,0) to[out=-90,in=180]  (.2,-0.2);
 			 \draw[<-,thick,darkred] (0.4,0) to[out=90,in=0] (.2,.2);
  			 \draw[-,thick,darkred] (.2,-0.2) to[out=0,in=-90] node[label={[shift={(0,-0.75)}] $\scriptstyle \lambda_{\hf}+\alpha-2$}] {$\bullet$} (0.4,0);			
  			 \node at (0.8,0.05) {$\scriptstyle{\lambda}$};
			\draw[->,thick,darkred](-.4,.4) to node  {$\bullet$} (-.4,-.4);
		\end{tikzpicture}
			}
		+
			\mathord{
		\begin{tikzpicture}[baseline= -4]
 			\draw[-,thick,darkred] (.2,0.2) to[out=180,in=90] (0,0);
 			 \draw[-,thick,darkred] (0.4,0) to[out=90,in=0] (.2,.2);
 			\draw[<-,thick,darkred] (0,0) to[out=-90,in=180]   (.2,-0.2);
  			\draw[-,thick,darkred] (.2,-0.2) to[out=0,in=-90] node[label={[shift={(0,-0.75)}] $\scriptstyle   \lambda_{\hf}-1+\alpha$}] {$\bullet$} (0.4,0);			
  			 \node at (0.8,0.05) {$\scriptstyle{\lambda}$};
			\draw[->,thick,darkred](-.4,.4) to (-.4,-.4);
		\end{tikzpicture}
			}
	&	= 
		\sum^{\alpha}_{s = 0} (s+1)  \!\!\!\!
			\mathord{
		\begin{tikzpicture}[baseline=-4]
 			\draw[-,thick,darkred] (.2,0.2) to[out=180,in=90] (0,0);
 			 \draw[<-,thick,darkred] (0.4,0) to[out=90,in=0] (.2,.2);
 			\draw[-,thick,darkred] (0,0) to[out=-90,in=180]  node[label={[shift={(-0.2,-0.75)}] $\scriptstyle \lambda_{\hf}-4+\alpha-s$}] {$\bullet$}  (.2,-0.2);
  			\draw[-,thick,darkred] (.2,-0.2) to[out=0,in=-90]  (0.4,0);  
  			 \node at (1.2,0.05) {$\scriptstyle{\lambda}$};
			\draw[->,thick,darkred](0.8,.4) to node[label={[shift={(0.2, -0.25)}] $\scriptstyle s$}] {$\bullet$} (0.8,-.4);
		\end{tikzpicture}
			}; \\
	\mathord{
		\begin{tikzpicture}[baseline= -4]
 			\draw[-,thick,darkred] (.2,0.2) to[out=180,in=90] (0,0);
			 \draw[-,thick,darkred] (0,0) to[out=-90,in=180] node[label={[shift={(0,-0.75)}] $\scriptstyle -\lambda_{\hf} +\alpha  $}] {$\bullet$} (.2,-0.2);
 			 \draw[->,thick,darkred] (0.4,0) to[out=90,in=0] (.2,.2);
  			 \draw[-,thick,darkred] (.2,-0.2) to[out=0,in=-90] (0.4,0);			
  			 \node at (1.2,0.05) {$\scriptstyle{\lambda}$};
			\draw[->,thick,darkred](0.8,.4) to node {$\bullet$} (0.8,-.4);
		\end{tikzpicture}
			}
		+
			\mathord{
		\begin{tikzpicture}[baseline= -4]
 			\draw[-,thick,darkred] (.2,0.2) to[out=180,in=90] (0,0);
 			 \draw[->,thick,darkred] (0.4,0) to[out=90,in=0] (.2,.2);
 			\draw[-,thick,darkred] (0,0) to[out=-90,in=180]  node[label={[shift={(0,-0.75)}] $\scriptstyle -\lambda_{\hf}+1+\alpha$}] {$\bullet$} (.2,-0.2);
  			\draw[-,thick,darkred] (.2,-0.2) to[out=0,in=-90] (0.4,0);			
  			 \node at (1.2,0.05) {$\scriptstyle{\lambda}$};
			\draw[->,thick,darkred](0.8,.4) to (0.8,-.4);
		\end{tikzpicture}
			}
		& = 
		\sum^{\alpha}_{s = 0} (s+1)
			\mathord{
		\begin{tikzpicture}[baseline= -4]
 			\draw[-,thick,darkred] (.2,0.2) to[out=180,in=90] (0,0);
 			 \draw[->,thick,darkred] (0.4,0) to[out=90,in=0] (.2,.2);
 			\draw[-,thick,darkred] (0,0) to[out=-90,in=180] (.2,-0.2);
  			\draw[-,thick,darkred] (.2,-0.2) to[out=0,in=-90] node[label={[shift={(0.1,-0.75)}] $\scriptstyle -\lambda_{\hf}-2+\alpha-s$}] {$\bullet$}  (0.4,0);			
  			 \node at (0.8,0.05) {$\scriptstyle{\lambda}$};
			\draw[->,thick,darkred] (-0.4,.4) to node[label={[shift={(0.2, -0.25)}] $\scriptstyle s$}] {$\bullet$} (-0.4,-.4);
		\end{tikzpicture}
			}; 
\end{align*}
\begin{align*}
 \mathord{
		\begin{tikzpicture}[baseline= -4]
 			\draw[-,thick,darkred] (.2,0.2) to[out=180,in=90] (0,0);
			 \draw[-,thick,darkred] (0,0) to[out=-90,in=180]  (.2,-0.2);
 			 \draw[->,thick,darkred] (0.4,0) to[out=90,in=0] (.2,.2);
  			 \draw[-,thick,darkred] (.2,-0.2) to[out=0,in=-90] node[label={[shift={(0,-0.75)}] $\scriptstyle -\lambda_{\hf}-2+\alpha$}] {$\bullet$} (0.4,0);			
  			 \node at (0.8,0.05) {$\scriptstyle{\lambda}$};
			\draw[->,thick,darkred](-.4,.4) to   (-.4,-.4);
		\end{tikzpicture}
			} 
			&=
			\mathord{
		\begin{tikzpicture}[baseline= -4]
 			\draw[-,thick,darkred] (.2,0.2) to[out=180,in=90] (0,0);
 			 \draw[->,thick,darkred] (0.4,0) to[out=90,in=0] (.2,.2);
 			\draw[-,thick,darkred] (0,0) to[out=-90,in=180] node[label={[shift={(0,-0.75)}] $\scriptstyle -\lambda_{\hf}-2+\alpha $}] {$\bullet$} (.2,-0.2);
  			\draw[-,thick,darkred] (.2,-0.2) to[out=0,in=-90]   (0.4,0);			
  			 \node at (1.2,0.05) {$\scriptstyle{\lambda}$};
			\draw[->,thick,darkred](0.8,.4) to node[label={[shift={(0.2, -0.25)}] $\scriptstyle 3$}] {$\bullet$} (0.8,-.4);
		\end{tikzpicture}
			}
-
		\mathord{
		\begin{tikzpicture}[baseline= -4]
 			\draw[-,thick,darkred] (.2,0.2) to[out=180,in=90] (0,0);
 			 \draw[->,thick,darkred] (0.4,0) to[out=90,in=0]  (.2,.2);
 			\draw[-,thick,darkred] (0,0) to[out=-90,in=180] node[label={[shift={(0,-0.75)}] $\scriptstyle -\lambda_{\hf}-1+\alpha $}] {$\bullet$} (.2,-0.2);
  			\draw[-,thick,darkred] (.2,-0.2) to[out=0,in=-90]  (0.4,0);			
  			 \node at (1.2,0.05) {$\scriptstyle{\lambda}$};
			\draw[->,thick,darkred](0.8,.4) to node[label={[shift={(0.2, -0.25)}] $\scriptstyle 2$}] {$\bullet$} (0.8,-.4);
		\end{tikzpicture}
			}
-
		\mathord{
		\begin{tikzpicture}[baseline= -4]
 			\draw[-,thick,darkred] (.2,0.2) to[out=180,in=90] (0,0);
 			 \draw[->,thick,darkred] (0.4,0) to[out=90,in=0]  (.2,.2);
 			\draw[-,thick,darkred] (0,0) to[out=-90,in=180] node[label={[shift={(0,-0.75)}] $\scriptstyle -\lambda_{\hf}+\alpha$}] {$\bullet$} (.2,-0.2);
  			\draw[-,thick,darkred] (.2,-0.2) to[out=0,in=-90]  (0.4,0);			
  			 \node at (1.2,0.05) {$\scriptstyle{\lambda}$};
			\draw[->,thick,darkred](0.8,.4) to node  {$\bullet$} (0.8,-.4);
		\end{tikzpicture}
			}
+
		\mathord{
		\begin{tikzpicture}[baseline= -4]
 			\draw[-,thick,darkred] (.2,0.2) to[out=180,in=90] (0,0);
 			 \draw[->,thick,darkred] (0.4,0) to[out=90,in=0]  (.2,.2);
 			\draw[-,thick,darkred] (0,0) to[out=-90,in=180] node[label={[shift={(0,-0.75)}] $\scriptstyle -\lambda_{\hf}+1+\alpha$}] {$\bullet$} (.2,-0.2);
  			\draw[-,thick,darkred] (.2,-0.2) to[out=0,in=-90]  (0.4,0);			
  			 \node at (1.2,0.05) {$\scriptstyle{\lambda}$};
			\draw[->,thick,darkred](0.8,.4) to  (0.8,-.4);
		\end{tikzpicture}
			}; \\
\mathord{
		\begin{tikzpicture}[baseline= -4]
 			\draw[-,thick,darkred] (.2,0.2) to[out=180,in=90] (0,0);
 			 \draw[<-,thick,darkred] (0.4,0) to[out=90,in=0]  (.2,.2);
 			\draw[-,thick,darkred] (0,0) to[out=-90,in=180] node[label={[shift={(0,-0.75)}] $\scriptstyle \lambda_{\hf} +\alpha-4$}] {$\bullet$} (.2,-0.2);
  			\draw[-,thick,darkred] (.2,-0.2) to[out=0,in=-90]  (0.4,0);			
  			 \node at (1.2,0.05) {$\scriptstyle{\lambda}$};
			\draw[->,thick,darkred](0.8,.4) to  (0.8,-.4);
		\end{tikzpicture}
			}
&=
		\mathord{
		\begin{tikzpicture}[baseline= -4]
 			\draw[-,thick,darkred] (.2,0.2) to[out=180,in=90] (0,0);
			 \draw[-,thick,darkred] (0,0) to[out=-90,in=180]  (.2,-0.2);
 			 \draw[<-,thick,darkred] (0.4,0) to[out=90,in=0] (.2,.2);
  			 \draw[-,thick,darkred] (.2,-0.2) to[out=0,in=-90] node[label={[shift={(0,-0.75)}] $\scriptstyle \lambda_{\hf} +\alpha-4$}] {$\bullet$} (0.4,0);			
  			 \node at (0.8,0.05) {$\scriptstyle{\lambda}$};
			\draw[->,thick,darkred](-.4,.4) to  node[label={[shift={(-0.2, -0.25)}] $\scriptstyle 3$}] {$\bullet$} (-.4,-.4);
		\end{tikzpicture}
			} 
-
		\mathord{
		\begin{tikzpicture}[baseline= -4]
 			\draw[-,thick,darkred] (.2,0.2) to[out=180,in=90] (0,0);
			 \draw[-,thick,darkred] (0,0) to[out=-90,in=180]  (.2,-0.2);
 			 \draw[<-,thick,darkred] (0.4,0) to[out=90,in=0] (.2,.2);
  			 \draw[-,thick,darkred] (.2,-0.2) to[out=0,in=-90] node[label={[shift={(0,-0.75)}] $\scriptstyle \lambda_{\hf} +\alpha-3$}] {$\bullet$} (0.4,0);			
  			 \node at (0.8,0.05) {$\scriptstyle{\lambda}$};
			\draw[->,thick,darkred](-.4,.4) to node[label={[shift={(-0.2, -0.25)}] $\scriptstyle 2$}] {$\bullet$}  (-.4,-.4);
		\end{tikzpicture}
			} 
-
		\mathord{
		\begin{tikzpicture}[baseline= -4]
 			\draw[-,thick,darkred] (.2,0.2) to[out=180,in=90] (0,0);
			 \draw[-,thick,darkred] (0,0) to[out=-90,in=180]  (.2,-0.2);
 			 \draw[<-,thick,darkred] (0.4,0) to[out=90,in=0] (.2,.2);
  			 \draw[-,thick,darkred] (.2,-0.2) to[out=0,in=-90] node[label={[shift={(0,-0.75)}] $\scriptstyle \lambda_{\hf} +\alpha-2$}] {$\bullet$} (0.4,0);			
  			 \node at (0.8,0.05) {$\scriptstyle{\lambda}$};
			\draw[->,thick,darkred](-.4,.4) to  node  {$\bullet$} (-.4,-.4);
		\end{tikzpicture}
			} 
+
		\mathord{
		\begin{tikzpicture}[baseline= -4]
 			\draw[-,thick,darkred] (.2,0.2) to[out=180,in=90] (0,0);
			 \draw[-,thick,darkred] (0,0) to[out=-90,in=180]  (.2,-0.2);
 			 \draw[<-,thick,darkred] (0.4,0) to[out=90,in=0] (.2,.2);
  			 \draw[-,thick,darkred] (.2,-0.2) to[out=0,in=-90] node[label={[shift={(0,-0.75)}] $\scriptstyle \lambda_{\hf} +\alpha-1$}] {$\bullet$} (0.4,0);			
  			 \node at (0.8,0.05) {$\scriptstyle{\lambda}$};
			\draw[->,thick,darkred](-.4,.4) to   (-.4,-.4);
		\end{tikzpicture}
			} ;
\end{align*}
\begin{align*}
			\mathord{
		\begin{tikzpicture}[baseline= -4]
 			\draw[-,thick,darkred] (.2,0.2) to[out=180,in=90] (0,0);
 			 \draw[<-,thick,darkred] (0.4,0) to[out=90,in=0] (.2,.2);
 			\draw[-,thick,darkred] (0,0) to[out=-90,in=180]   (.2,-0.2);
  			\draw[-,thick,darkred] (.2,-0.2) to[out=0,in=-90] node[label={[shift={(0,-0.75)}] $\scriptstyle  \lambda_{\hf}-1+\alpha$}] {$\bullet$} (0.4,0);			
  			 \node at (0.8,0.05) {$\scriptstyle{\lambda}$};
			\draw[->,thick,darkred](-.4,.4) to (-.4,-.4);
		\end{tikzpicture}
			}
	&	= 
		\sum^{\alpha}_{s = 0} \left \lceil \frac{s+1}{2} \right \rceil  \!\!\!\!
			\mathord{
		\begin{tikzpicture}[baseline=-4]
 			\draw[-,thick,darkred] (.2,0.2) to[out=180,in=90] (0,0);
 			 \draw[<-,thick,darkred] (0.4,0) to[out=90,in=0] (.2,.2);
 			\draw[-,thick,darkred] (0,0) to[out=-90,in=180]  node[label={[shift={(-0.2,-0.75)}] $\scriptstyle \lambda_{\hf}-4+\alpha-s$}] {$\bullet$}  (.2,-0.2);
  			\draw[-,thick,darkred] (.2,-0.2) to[out=0,in=-90]  (0.4,0);  
  			 \node at (1.2,0.05) {$\scriptstyle{\lambda}$};
			\draw[->,thick,darkred](0.8,.4) to node[label={[shift={(0.2, -0.25)}] $\scriptstyle s$}] {$\bullet$} (0.8,-.4);
		\end{tikzpicture}
			};  \\
			\mathord{
		\begin{tikzpicture}[baseline= -4]
 			\draw[-,thick,darkred] (.2,0.2) to[out=180,in=90] (0,0);
 			 \draw[->,thick,darkred] (0.4,0) to[out=90,in=0] (.2,.2);
 			\draw[-,thick,darkred] (0,0) to[out=-90,in=180]  node[label={[shift={(0,-0.75)}] $\scriptstyle -\lambda_{\hf}+1+\alpha$}] {$\bullet$} (.2,-0.2);
  			\draw[-,thick,darkred] (.2,-0.2) to[out=0,in=-90] (0.4,0);			
  			 \node at (1.2,0.05) {$\scriptstyle{\lambda}$};
			\draw[->,thick,darkred](0.8,.4) to (0.8,-.4);
		\end{tikzpicture}
			}
		& = 
		\sum^{\alpha}_{s = 0} \left\lceil \frac{s+1}{2} \right\rceil
			\mathord{
		\begin{tikzpicture}[baseline= -4]
 			\draw[-,thick,darkred] (.2,0.2) to[out=180,in=90] (0,0);
 			 \draw[->,thick,darkred] (0.4,0) to[out=90,in=0] (.2,.2);
 			\draw[-,thick,darkred] (0,0) to[out=-90,in=180] (.2,-0.2);
  			\draw[-,thick,darkred] (.2,-0.2) to[out=0,in=-90] node[label={[shift={(0.2,-0.75)}] $\scriptstyle -\lambda_{\hf}-2+\alpha-s$}] {$\bullet$}  (0.4,0);			
  			 \node at (0.8,0.05) {$\scriptstyle{\lambda}$};
			\draw[->,thick,darkred] (-0.4,.4) to node[label={[shift={(0.2, -0.25)}] $\scriptstyle s$}] {$\bullet$} (-0.4,-.4);
		\end{tikzpicture}
			}.
\end{align*}
\end{prop}

\subsection{Proof of Proposition~\ref{p:left-inv} }

In this subsection, we shall first establish a variant of the identity \eqref{eq:prop:jserre1}, as a preparation toward the proof  of the identity \eqref{eq:prop:jserre1}. 

Let $\lambda \in \X_\jmath$ with $\lambda_\hf \ge 2$. Recall the $2$-morphism $\kappa \in \End(\calE_\hf \calF^2_\hf 1_\lambda)$ in Lemma~\ref{lem:dec:eff}. The following lemma follows from the proof therein.

\begin{lemma}
\label{lem:*}
We have $\kappa \cdot \iota_t=0$, for $0\le t \le \lambda_{\hf}-3$.
\end{lemma}
%

Let us define
\[
\kappa_2 = 
{\frac 12} \sum_{a+b + c=-2}
	\mathord{
		\begin{tikzpicture}[baseline = 0pt]
   		\draw[->,thick,darkred] (0.9, 0.2) -- (0.9,-0.7);
 		\draw[thick,darkred] (0.9,-0.7) to[out=-90, in=90] (0.3, -1.3);
  		\draw[->,thick,darkred] (0.2,0) to[out=90,in=0] (0,0.2);
  		\draw[-,thick,darkred] (0,0.2) to[out=180,in=90] (-.2,0);
		\draw[-,thick,darkred] (-.2,0) to[out=-90,in=180] (0,-0.2);
  		\draw[-,thick,darkred] (0,-0.2) to[out=0,in=-90] (0.2,0);
   		\node[label={[shift={(0.2,-0.5)}]$\color{darkred}\scriptstyle{a}$}] at (0.2,0) {$\color{darkred}\bullet$};
		\draw[<-,thick,darkred] (0.3,-.7) to[out=90, in=0] (0,-0.3);
		\draw[-,thick,darkred] (0.3,-.7)  to[out=-90, in=90] (0.9,-1.3) ;
		\draw[-,thick,darkred] (0,-0.3) to[out = 180, in = 90] (-0.3,-.7);
		\draw[-,thick, darkred] (-0.3,-.7) -- (-0.3, -1.3);
   		\node[label={[shift={(-0.2,-0.5)}]$\color{darkred}\scriptstyle{b}$}] at (-0.25,-0.5) {$\color{darkred}\bullet$};
   		\draw[-, thick, darkred] (-0.3, 0.9) -- (-0.3, 1.6);
		\draw[<-, thick, darkred] (0.9, 0.9) to[out=90, in=-90] (0.3, 1.6);
		\draw[->, thick, darkred] (0.9, 1.6) to[out=-90, in=90] node[near start] {$\color{darkred}\bullet$}  (0.3, 0.9);
		\draw[-, thick, darkred] (0.9, 0.9) to[out=-90, in=90] (0.9, 0.5);
		\draw[-,thick, darkred] (0.9, 0.5) -- (0.9, 0.2); 
		\draw[-,thick,darkred] (0.3,0.9) to[out=-90, in=0] (0,0.5);
		\draw[->,thick,darkred] (0,0.5) to[out = 180, in = -90] (-0.3,0.9);
      	\node[label={[shift={(-.3,-0.5)}] $\color{darkred}\scriptstyle{c}$}] at (0,0.5) {$\color{darkred}\bullet$};
        \node at (1.2,1.5) {$\scriptstyle{\lambda}$};
		\end{tikzpicture}
			}\quad 
			\text{ and }
			\quad 
			\kappa_3 = 
			{\frac 12} \sum_{a+b +c = -3}
	\mathord{
		\begin{tikzpicture}[baseline = 0pt]
   			\draw[->,thick,darkred] (0.9, 0.2) -- (0.9,-0.7);
 			\draw[thick,darkred] (0.9,-0.7) to[out=-90, in=90] (0.3, -1.3);
  			\draw[->,thick,darkred] (0.2,0) to[out=90,in=0] (0,0.2);
  			\draw[-,thick,darkred] (0,0.2) to[out=180,in=90] (-.2,0);
			\draw[-,thick,darkred] (-.2,0) to[out=-90,in=180] (0,-0.2);
  			\draw[-,thick,darkred] (0,-0.2) to[out=0,in=-90] (0.2,0);
   			\node[label={[shift={(0.2,-0.5)}]$\color{darkred}\scriptstyle{a}$}] at (0.2,0) {$\color{darkred}\bullet$};
			\draw[<-,thick,darkred] (0.3,-.7) to[out=90, in=0] (0,-0.3);
			\draw[-,thick,darkred] (0.3,-.7)  to[out=-90, in=90] (0.9,-1.3) ;
			\draw[-,thick,darkred] (0,-0.3) to[out = 180, in = 90] (-0.3,-.7);
			\draw[-,thick, darkred] (-0.3,-.7) -- (-0.3, -1.3);
   			\node[label={[shift={(-0.2,-0.5)}]$\color{darkred}\scriptstyle{b}$}] at (-0.25,-0.5) {$\color{darkred}\bullet$};
   			\draw[-, thick, darkred] (-0.3, 0.9) -- (-0.3, 1.6);
			\draw[<-, thick, darkred] (0.9, 0.9) to[out=90, in=-90] (0.3, 1.6);
			\draw[->, thick, darkred] (0.9, 1.6) to[out=-90, in=90] node[near start] {$\color{darkred}\bullet$}  (0.3, 0.9);
			\draw[-, thick, darkred] (0.9, 0.9) to[out=-90, in=90] (0.9, 0.5);
			\draw[-, thick, darkred] (0.9, 0.5) -- node {$\color{darkred}\bullet$} (0.9, 0.2); 
			\draw[-,thick,darkred] (0.3,0.9) to[out=-90, in=0] (0,0.5);
			\draw[->,thick,darkred] (0,0.5) to[out = 180, in = -90] (-0.3,0.9);
     		\node[label={[shift={(-0.3,-0.5)}] $\color{darkred}\scriptstyle{c}$}] at (0,0.5) {$\color{darkred}\bullet$};
        		\node at (1.2,1.5) {$\scriptstyle{\lambda}$};
		\end{tikzpicture}
			}.
\]

Recall the definition of $\kappa$ in Lemma~\ref{lem:dec:eff} uses $\iota_s$, which consists of the top parts of $\kappa_2$ and $\kappa_3$.
An alternative way of decomposing $\kappa$ is available by use of the low parts of $\kappa_2$ and $\kappa_3$. To that end, we set
\[
\tilde{\pi}_s =
\mathord{
		\begin{tikzpicture}[baseline = -25]
		\draw[->,thick,darkred] (0.9, 0.2) -- (0.9,-0.7);
		\draw[<-,thick,darkred] (0.3,-.7) to[out=90, in=0] (0,-0.3);
		\draw[-,thick,darkred] (0.3,-.7)  to[out=-90, in=90] (0.9,-1.3) ;
		\draw[-,thick,darkred] (0,-0.3) to[out = 180, in = 90] (-0.3,-.7);
		\draw[-,thick, darkred] (-0.3,-.7) -- (-0.3, -1.3);
   		\node[label={[shift={(-0.2,-0.5)}]$\color{darkred}\scriptstyle{s}$}] at (-0.25,-0.5) {$\color{darkred}\bullet$};
 		\draw[thick,darkred] (0.9,-0.7) to[out=-90, in=90] (0.3, -1.3);
		\end{tikzpicture}
		}
\quad \text{ and } \quad
\tilde{\iota}_s = -  {\frac 12} \sum_{a+c=-s-2}
	\mathord{
		\begin{tikzpicture}[baseline = 10]
		\draw[-, thick, darkred] (-0.3, 0.9) -- (-0.3, 1.6);
		\draw[<-, thick, darkred] (0.9, 0.9) to[out=90, in=-90] (0.3, 1.6);
		\draw[->, thick, darkred] (0.9, 1.6) to[out=-90, in=90] node[near start] {$\color{darkred}\bullet$}  (0.3, 0.9);
		\draw[-, thick, darkred] (0.9, 0.9) to[out=-90, in=90] (0.9, 0.5);
		\draw[-,thick,darkred] (0.3,0.9) to[out=-90, in=0] (0,0.5);
		\draw[->,thick,darkred] (0,0.5) to[out = 180, in = -90] (-0.3,0.9);
      	\node[label={[shift={(-.3,-0.5)}] $\color{darkred}\scriptstyle{c}$}] at (0,0.5) {$\color{darkred}\bullet$};
  		\draw[->,thick,darkred] (0.2,0) to[out=90,in=0] (0,0.2);
  		\draw[-,thick,darkred] (0,0.2) to[out=180,in=90] (-.2,0);
		\draw[-,thick,darkred] (-.2,0) to[out=-90,in=180] (0,-0.2);
  		\draw[-,thick,darkred] (0,-0.2) to[out=0,in=-90] (0.2,0);
   		\node[label={[shift={(0.2,-0.5)}]$\color{darkred}\scriptstyle{a}$}] at (0.2,0) {$\color{darkred}\bullet$};
		\draw[-,thick, darkred] (0.9, 0.5) -- (0.9, 0.2); 
		\draw[->,thick,darkred] (0.9, 0.2) -- (0.9,-0.7);
		\end{tikzpicture}
			}+\;
			{\frac 12} \sum_{a+c=-s-3}
	\mathord{
		\begin{tikzpicture}[baseline = 10]
		\draw[-, thick, darkred] (-0.3, 0.9) -- (-0.3, 1.6);
		\draw[<-, thick, darkred] (0.9, 0.9) to[out=90, in=-90] (0.3, 1.6);
		\draw[->, thick, darkred] (0.9, 1.6) to[out=-90, in=90] node[near start] {$\color{darkred}\bullet$}  (0.3, 0.9);
		\draw[-, thick, darkred] (0.9, 0.9) to[out=-90, in=90] (0.9, 0.5);
		\draw[-,thick,darkred] (0.3,0.9) to[out=-90, in=0] (0,0.5);
		\draw[->,thick,darkred] (0,0.5) to[out = 180, in = -90] (-0.3,0.9);
      	\node[label={[shift={(-.3,-0.5)}] $\color{darkred}\scriptstyle{c}$}] at (0,0.5) {$\color{darkred}\bullet$};
  		\draw[->,thick,darkred] (0.2,0) to[out=90,in=0] (0,0.2);
  		\draw[-,thick,darkred] (0,0.2) to[out=180,in=90] (-.2,0);
		\draw[-,thick,darkred] (-.2,0) to[out=-90,in=180] (0,-0.2);
  		\draw[-,thick,darkred] (0,-0.2) to[out=0,in=-90] (0.2,0);
   		\node[label={[shift={(0.2,-0.5)}]$\color{darkred}\scriptstyle{a}$}] at (0.2,0) {$\color{darkred}\bullet$};
		\draw[-,thick, darkred] (0.9, 0.5) -- (0.9, 0.2); 
		\draw[->,thick,darkred] (0.9, 0.2) -- node {$\color{darkred}\bullet$} (0.9,-0.7);
		\end{tikzpicture}
			}
\]

Then clearly we also  have 
\begin{equation}
\label{k2}
\kappa = \id_{\calE_{\hf} \calF^{(2)}_{\hf} 1_{\lambda}} - \sum_{0\le s \le \lambda_\hf -3} \tilde{\iota}_s \tilde{\pi}_s.
\end{equation}

The following lemma is a counterpart of Lemma~\ref{lem:*}, whose proof is skipped.
\begin{lemma}
\label{lem:**}
We have 
$\tilde{\pi}_t \tilde{\iota}_s = \delta_{s,t} \id_{\calF_{\hf} 1_{\lambda}}$, for $0\le s \le \lambda_{\hf}-3$ and $0\le t \le \lambda_{\hf}-3$.
Moreover, we have 
$\tilde{\pi}_t \cdot \kappa =0$ for $0\le t \le \lambda_{\hf}-3$.
\end{lemma}

%

Recall the diagrams $B_1, B_2, P_k, C_1, C_2$ and $I_k$, for $0\le k \le \lambda_\hf-1$, from 
\eqref{diag:BBP}--\eqref{diag:CCI}. 
We shall establish first a simpler version of the identity \eqref{eq:prop:jserre1}. 

\begin{prop}
\label{p:left-inv}
The following identity holds in $\mathfrak{\tilde U}^\jmath$:
\begin{align}
  \label{eq:inverse1}

=\text{id}_{(\lambda_{\hf}+2) \times (\lambda_{\hf}+2)}.
\end{align}
\end{prop}
The proof of the identity in Proposition~\ref{p:left-inv} is divided into Lemmas~\ref{lem:BC1}-\ref{lem:B2C2} below. 
As we shall not need the relation \eqref{Pi=1} in the process, the identity holds in $\mathfrak{\tilde U}^\jmath$. 

\begin{lemma}
\label{lem:BC1}
We have $B_2 \cdot C_1=0$ and $B_1 \cdot C_2=0$.
\end{lemma}

\begin{proof}
We have
\[
 B_1 \cdot C_2 = - \,
			\mathord{
				%
				}
				=0.
\]
The proof for the second identity is similar and will be skipped.
\end{proof}

\begin{lemma}
\label{lem:PI}
We have $P_k \cdot I_\ell = \delta_{k,\ell} \id_{\calF_{\hf} 1_{\lambda}}$, for $0\le k, \ell \le \lambda_{\hf}-1$.
\end{lemma}

\begin{proof}
By the bubble relation \eqref{bubble3}, we compute
$$
P_k \cdot I_\ell = {\frac 12} \sum_{s+t = l}\,
	\mathord{
		%
					}
= 0
\]
thanks to the vanishing of bubbles of negative degree in relations \eqref{bubble1}--\eqref{bubble2}.

Again because of the vanishing of bubbles of negative degree in relation \eqref{bubble1}, we have
\[
P_k \cdot C_1 = 
\,
		\mathord{
			%
					}.					
$$
Hence $B_2 \cdot I_k$ is of the form
$$
B_2 \cdot I_k = \sum_{0\le t \le \lambda_{\hf}-3} \iota_t \cdot x_t,
$$
for some suitable $x_t$. By Lemma~\ref{lem:*}, we have
$$
\kappa \cdot B_2 \cdot I_k =\sum_{t} \kappa  \iota_t x_t =0.
$$
\end{proof}

\begin{lemma}
\label{lem:PC2}
We have $P_k \cdot C_2 \cdot \kappa =0$, for $0 \le k \le \lambda_{\hf}-1$.
\end{lemma}

\begin{proof}
We have
$$
P_k \cdot C_2 =
 	\mathord{
				%
					} , 
$$
which is of the form
$$
P_k \cdot C_2 = \sum_{0\le s \le \lambda_{\hf} -2} y_s \tilde{\pi}_s ,
$$
for some suitable $y_s$.  This implies by Lemma~\ref{lem:**} that 
$P_k \cdot C_2 \cdot \kappa  =0$.
\end{proof}

\begin{lemma}
We have $B_1 \cdot C_1 =  \id_{ \calF^{(2)}_{\hf} \calE_{\hf}1_{\lambda}}.$
\end{lemma}

\begin{proof}
Using equation \eqref{j-bicross up}, we compute that
\begin{align*}
B_1 \cdot C_1 & = -\,
		\mathord{
			%
				} 
				=0, \quad  \text{ for } \lambda_{\hf} \ge 2,
\]
because of the bubble relation \eqref{bubble1}. 
\end{proof}

\begin{lemma}
\label{lem:B2C2}
We have $  \kappa \cdot B_2 \cdot C_2 = \kappa$. 
\end{lemma}

\begin{proof}
We compute that
\begin{align*}
B_2 \cdot C_2 =\, 
	 \mathord{
			%
			}
\end{align*}
Note that the third summand above equals $ \id_{\calE_{\hf} \calF^{(2)}_{\hf} 1_{\lambda}}$ and the second equals $0$.
Denote by $Y$ the first summand above. It remains to show that $ \kappa \cdot  Y=0$.
Note that $Y$ is of the form $ Y =\sum_{0\le t \le \lambda_{\hf}-3} \iota_t z_t $ for some $z_t$.
This implies by Lemma~\ref{lem:*} that $ \kappa \cdot Y=0$. The lemma follows.
\end{proof}
Therefore, we have completed the proof of Proposition~\ref{p:left-inv}.

\subsection{Proof of the identity \eqref{eq:prop:jserre1} }

Now we would like to modify some $2$-morphisms involved in Proposition~\ref{p:left-inv}.
Introduce

\[ 
A = 
		 \mathord{
			\begin{tikzpicture}[baseline=0,scale=0.5]
				\draw[<-, thick, darkred] (1, -1) -- (-1, 1);
				\draw[-, thick, darkred] (0, 1) to[out=-90, in=180] (0.5, 0.3);
				\draw[<-, thick, darkred] (0.5, 0.3) to[out=0, in=-90] (1, 1);
				\draw[->, thick, darkred] (0, -1) to[out=90, in=0] (- 0.5, -0.3);
				\draw[-, thick, darkred] (- 0.5, -0.3) to[out=180, in=90] (-1, -1);
				\node at (1.8, 0) {$\scriptstyle \lambda$};
			\end{tikzpicture}
			},
\qquad
B= 
 {\frac 12}\sum_{u+s + t = -3} \,
	\mathord{
		\begin{tikzpicture}[baseline=0]
			\draw[<-,thick,darkred] (0.4,-0.6) to[out=90, in=0] (.2,-0.3);
			\draw[-,thick,darkred] (.2,-0.3) to[out = 180, in = 90] (0,-.6);
  			 \node at (.35,-0.35) {$\color{darkred}\bullet$};
  			 \node at (0.5,-.35) {$\color{darkred}\scriptstyle{t}$};
  			\draw[-,thick,darkred] (.2,0.2) to[out=180,in=90] (0,0);
 			\draw[->,thick,darkred] (0.4,0) to[out=90,in=0] (.2,.2);
 			\draw[-,thick,darkred] (0,0) to[out=-90,in=180]  (.2,-0.2);
  			\draw[-,thick,darkred] (.2,-0.2) to[out=0,in=-90] node[near end, label={[shift={(0.2,-0.25)}]$\color{darkred}\scriptstyle{s}$}] {$\color{darkred} \bullet$}  (0.4,0);			
  			 \node at (0.8,0.05) {$\scriptstyle{\lambda}$};
			\draw[-,thick,darkred] (0.4,.6) to[out=-90, in=0] (.2,0.3);
			\draw[->,thick,darkred] (.2,0.3) to[out = -180, in = -90] (0,.6);
			\draw[->,thick,darkred](-.4,.6) to node[label={[shift={(-0.2,-.45)}]$\color{darkred}\scriptstyle{u}$}] {$\color{darkred} \bullet$} (-.4,-.6);
		\end{tikzpicture}
			},
\qquad 
C =
  \mathord{
			\begin{tikzpicture}[baseline=0,scale=0.5]
				\draw[->, thick, darkred] (1, 1) -- (-1, -1);
				\draw[->, thick, darkred] (0, -1) to[out=90, in=180] (0.5, -0.3);
				\draw[-, thick, darkred] (0.5, -0.3) to[out=0, in=90] (1, -1);
				\draw[-, thick, darkred] (0, 1) to[out=-90, in=0] (- 0.5, 0.3);
				\draw[<-, thick, darkred] (- 0.5, 0.3) to[out=180, in=-90] (-1, 1);
				\node at (1.8, 0) {$\scriptstyle \lambda$};
			\end{tikzpicture}
			},
\]
\\
\[
D= 
 {\frac 12}\sum_{u+s + t = -3} \,
	\mathord{
		\begin{tikzpicture}[baseline=0]
			\draw[<-,thick,darkred] (0.4,-0.6) to[out=90, in=0] (.2,-0.3);
			\draw[-,thick,darkred] (.2,-0.3) to[out = 180, in = 90] (0,-.6);
  			\draw[-,thick,darkred] (.2,0.2) to[out=180,in=90] (0,0);
 			\draw[->,thick,darkred] (0.4,0) to[out=90,in=0] (.2,.2);
 			\draw[-,thick,darkred] (0,0) to[out=-90,in=180]  (.2,-0.2);
  			\draw[-,thick,darkred] (.2,-0.2) to[out=0,in=-90] node[near end, label={[shift={(0.2,-0.25)}]$\color{darkred}\scriptstyle{s}$}] {$\color{darkred} \bullet$}  (0.4,0);			
  			 \node at (0.8,0.05) {$\scriptstyle{\lambda}$};
			\draw[-,thick,darkred] (0.4,.6) to[out=-90, in=0] (.2,0.3);
			\draw[->,thick,darkred] (.2,0.3) to[out = -180, in = -90] (0,.6);
   			\node at (0.4,0.45) {$\color{darkred}\bullet$};
   			\node at (0.6,0.45) {$\color{darkred}\scriptstyle{t}$};
			\draw[->,thick,darkred](-.4,.6) to node[label={[shift={(-0.2,-.45)}]$\color{darkred}\scriptstyle{u}$}] {$\color{darkred} \bullet$} (-.4,-.6);
		\end{tikzpicture}
			},
\qquad 
E = {\frac 12} \,
	\mathord{
		\begin{tikzpicture}[baseline=-4]
			 \node at (0.6,0.05) {$\scriptstyle{\lambda}$};
			\draw[<-,thick,darkred] (0.4,-0.5) to[out=90, in=0] (.2,-0.2);
			\draw[-,thick,darkred] (.2,-0.2) to[out = 180, in = 90] (0,-.5);
			\draw[-,thick,darkred] (0.4,.5) to[out=-90, in=0] (.2,0.2);
			\draw[->,thick,darkred] (.2,0.2) to[out = -180, in = -90] (0,.5);
			\draw[->,thick,darkred](-.3,.5) to  (-.3,-.5);
			\draw[-,thick,darkred] (-0.8,0.2) to[out=180,in=90] (-1,0);
 			\draw[->,thick,darkred] (-0.6,0) to[out=90,in=0] (-0.8,.2);
 			\draw[-,thick,darkred] (-1,0) to[out=-90,in=180]  (-0.8,-0.2);
  			\draw[-,thick,darkred] (-0.8,-0.2) to[out=0,in=-90]   (-0.6,0);	
		\end{tikzpicture}
			}
			-\,
			 {\frac 12} \,
	\mathord{
		\begin{tikzpicture}[baseline=-4]
			 \node at (0.6,0.05) {$\scriptstyle{\lambda}$};
			\draw[<-,thick,darkred] (0.4,-0.5) to[out=90, in=0] (.2,-0.2);
			\draw[-,thick,darkred] (.2,-0.2) to[out = 180, in = 90] (0,-.5);
			\draw[-,thick,darkred] (0.4,.5) to[out=-90, in=0] (.2,0.2);
			\draw[->,thick,darkred] (.2,0.2) to[out = -180, in = -90] (0,.5);
			\draw[->,thick,darkred](-.3,.5) to node {$\color{darkred} \bullet$} (-.3,-.5);
			\draw[-,thick,darkred] (-0.8,0.2) to[out=180,in=90] (-1,0);		
			\draw[->,thick,darkred] (-0.6,0) to[out=90,in=0] (-0.8,.2);
 			\draw[-,thick,darkred] (-1,0) to[out=-90,in=180] node[label={[shift={(-0.1,-.75)}]$\color{darkred}\scriptstyle{-1}$}] {$\color{darkred} \bullet$} (-0.8,-0.2);
  			\draw[-,thick,darkred] (-0.8,-0.2) to[out=0,in=-90]   (-0.6,0);	
		\end{tikzpicture}
			}.
\]

We define $I_{\lambda_{\hf} -1}' =(1-C+D) \cdot I_{\lambda_{\hf}-1}$.
It is easy to show that
\begin{equation}
\label{eq:I'}
\begin{split}
I_{\lambda_{\hf}-1}' &  =   {\frac 12}\sum_{s+t = \lambda_{\hf} -1 }
	\mathord{
		\begin{tikzpicture}[baseline=0]
  			\draw[-,thick,darkred] (.2,0.2) to[out=180,in=90] (0,0);
 			\draw[-,thick,darkred] (0,0) to[out=-90,in=180] (.2,-0.2);
  			\draw[-,thick,darkred] (.2,-0.2) to[out=0,in=-90]  node[label={[shift={(0.2,-0.75)}]$\color{darkred}\scriptstyle{-\lambda_{\hf}-2+t}$}] {$\color{darkred}\bullet$} (0.4,0);
			 \draw[->,thick,darkred] (0.4,0) to[out=90,in=0] (.2,.2);			
  			 \node at (0.8,0.05) {$\scriptstyle{\lambda}$};
			\draw[-,thick,darkred] (0.4,.6) to[out=-90, in=0] (.2,0.3);
			\draw[->,thick,darkred] (.2,0.3) to[out = -180, in = -90] (0,.6);
   			\node at (0.4,0.45) {$\color{darkred}\bullet$};
   			\node at (0.6,0.45) {$\color{darkred}\scriptstyle{s}$};
			\draw[->,thick,darkred](-.4,.6) to (-.4,-.6);
 		\end{tikzpicture}
			}
		-
	\mathord{
		\begin{tikzpicture}[baseline=0, scale=0.4]
			\draw[-, thick, darkred] (-1, 1) to[out=-90, in =180] (-0.5, 0);
			\draw[->, thick, darkred] (-0.5, 0) to[out=0, in=-90] (0,1);
			\draw[->, thick, darkred] (1,1) -- (1,-1);
			\node at (1.6, 1) {$\scriptstyle{\color{black} \lambda}$};
		\end{tikzpicture}
			}
		+ {\frac 12}\sum_{s+t +u  = \lambda_{\hf} -1 }
	\mathord{
		\begin{tikzpicture}[baseline=0]
  			\draw[-,thick,darkred] (.2,0.2) to[out=180,in=90] (0,0);
 			\draw[-,thick,darkred] (0,0) to[out=-90,in=180] (.2,-0.2);
  			\draw[-,thick,darkred] (.2,-0.2) to[out=0,in=-90]  node[label={[shift={(0.2,-0.75)}]$\color{darkred}\scriptstyle{-\lambda_{\hf}-2+t}$}] {$\color{darkred}\bullet$} (0.4,0);
			 \draw[->,thick,darkred] (0.4,0) to[out=90,in=0] (.2,.2);			
  			 \node at (0.8,0.05) {$\scriptstyle{\lambda}$};
			\draw[-,thick,darkred] (0.4,.6) to[out=-90, in=0] (.2,0.3);
			\draw[->,thick,darkred] (.2,0.3) to[out = -180, in = -90] (0,.6);
   			\node at (0.4,0.45) {$\color{darkred}\bullet$};
   			\node at (0.6,0.45) {$\color{darkred}\scriptstyle{s}$};
			\draw[->,thick,darkred](-.4,.6) --node[label={[shift={(-0.2, -0.2)}]$\scriptstyle u$}] {$\bullet$} (-.4,-.6);
 		\end{tikzpicture}
			}	
			\\
&= I_{\lambda_{\hf}-1} -I(2) +I(3).
\end{split}
\end{equation}
The second line above defines $I(2)$ and $I(3)$ as the second and third summands without signs in the first line; these notations will be used below. 

We also define
$P_0' = P_0 \cdot (1 -A +B -E).
$
Then we can show readily that

\begin{equation}
\label{eq:P'}
\begin{split}
P_{0}' 
&= \mathord{
		\begin{tikzpicture}[baseline=0,scale=0.4]
			\draw[->, thick, darkred] (-1, 1) -- (-1, -1);
			\draw[->, thick, darkred] (0, -1) to[out=90, in=180] (0.5, 0);
			\draw[-, thick, darkred] (0.5, 0) to[out=0, in=90] node[label={[shift={(0.1, -0.05)}]$\scriptstyle\lambda_{\hf} - 1$}] {$\color{darkred} \bullet$} (1, -1);
			\node at (1.8, 0) {$\scriptstyle \lambda$};
		\end{tikzpicture}
			}
		- 2 \,
		\mathord{
			\begin{tikzpicture}[baseline=0,scale=0.4]
				\draw[<-, thick, darkred] (-1, -1) to[out=90, in=180] (-0.5, 0);
				\draw[-, thick, darkred] (-0.5,0) to[out=0, in=90] (0, -1);
				\draw[->, thick, darkred] (1, 1) --node[label={[shift={(0.3,0)}]$\color{black}\scriptstyle \lambda$}] {} (1, -1);
			\end{tikzpicture}
				}
		+ \sum_{u+s+t= -3} 
	\mathord{
		\begin{tikzpicture}[baseline=-4]
			\draw[<-,thick,darkred] (0.4,-0.6) to[out=90, in=0]  node[label={[shift={(0.2, -0.45)}]$\color{darkred}\scriptstyle{t}$}] {$\color{darkred}\bullet$} (.2,-0.3);
			\draw[-,thick,darkred] (.2,-0.3) to[out = 180, in = 90] (0,-.6);
  			\draw[-,thick,darkred] (.2,0.2) to[out=180,in=90] (0,0);
 			 \draw[->,thick,darkred] (0.4,0) to[out=90,in=0] (.2,.2);
 			\draw[-,thick,darkred] (0,0) to[out=-90,in=180] (.2,-0.2);
  			\draw[-,thick,darkred] (.2,-0.2) to[out=0,in=-90] (0.4,0);			
  			 \node at (0.8,0.05) {$\scriptstyle{\lambda}$};
   			\node at (0,0) {$\color{darkred}\bullet$};
   			\node at (-0.2,0) {$\color{darkred}\scriptstyle{s}$};
			\draw[->,thick,darkred](-.4,.4) --node[label={[shift={(-0.2, -0.35)}]$\scriptstyle u$}] {$\color{darkred}\bullet$} (-.4,-.6);
		\end{tikzpicture}
			}
		+
	\mathord{
		\begin{tikzpicture}[baseline=-4]
			\draw[-,thick,darkred] (.2,0.2) to[out=180,in=90] (0,0);
 			\draw[-,thick,darkred] (0,0) to[out=-90,in=180] (.2,-0.2);
  			\draw[-,thick,darkred] (.2,-0.2) to[out=0,in=-90] node[label={[shift={(0, -0.75)}]$\scriptstyle {-1}$}] {$\color{darkred}\bullet$} (0.4,0);
			 \draw[->,thick,darkred] (0.4,0) to[out=90,in=0] (.2,.2);			
  			 \node at (1.1,0.05) {$\scriptstyle{\lambda}$};
			\draw[->, thick, darkred] (0.8,.4) --node {$\color{darkred}\bullet$} (0.8,-.6);
			\draw[<-,thick,darkred] (1.6,-0.6) to[out=90, in=0]  (1.4,-0.3);
			\draw[-,thick,darkred] (1.4,-0.3) to[out = 180, in = 90] (1.2,-.6);
		\end{tikzpicture}
			}
		- \,
			\mathord{
		\begin{tikzpicture}[baseline=-4]
			\draw[-,thick,darkred] (.2,0.2) to[out=180,in=90] (0,0);
 			\draw[-,thick,darkred] (0,0) to[out=-90,in=180] (.2,-0.2);
  			\draw[-,thick,darkred] (.2,-0.2) to[out=0,in=-90] (0.4,0);
			 \draw[->,thick,darkred] (0.4,0) to[out=90,in=0] (.2,.2);			
  			 \node at (1.1,0.05) {$\scriptstyle{\lambda}$};
			\draw[->, thick, darkred] (0.8,.4) -- (0.8,-.6);
			\draw[<-,thick,darkred] (1.6,-0.6) to[out=90, in=0]  (1.4,-0.3);
			\draw[-,thick,darkred] (1.4,-0.3) to[out = 180, in = 90] (1.2,-.6);
		\end{tikzpicture}
		}	
\\
& =P_0 -P(2) +P(3) -P(4) +P(5).
		\end{split}
\end{equation}
The second line above defines $P(2)$, $P(3)$, $P(4)$, and $P(5)$ as the second to fifth summands without signs in the first line. 

We would like to redo Proposition~\ref{p:left-inv} with $I_{\lambda_{\hf}-1}$ and $P_0$ replaced by
$I_{\lambda_{\hf} -1}'$ and $P_0'$, respectively. 

\begin{prop}
  \label{p:left-inv'}
The identity \eqref{eq:prop:jserre1} holds, that is, we have 
\begin{align*}
\begin{bmatrix}
B_1 \\ \kappa \cdot B_2 \\ P'_0 \\ \vdots \\ P_{\lambda_{\hf}-1}
\end{bmatrix}
\cdot
\begin{bmatrix}
C_1 & C_2 \cdot \kappa &I_0 & \cdots & I'_{\lambda_{\hf}-1} 
\end{bmatrix}
=\text{id}_{(\lambda_{\hf}+2) \times (\lambda_{\hf}+2)}. 
\end{align*}
\end{prop}
Thanks to Proposition~\ref{p:left-inv}, we only need to consider the relations involving $P'_0$ or $I'_{\lambda_{\hf}-1}$. 
The computation is divided into Lemmas~\ref{lem:BI'}--\ref{lem:PI'} below.
\begin{lemma}
 \label{lem:BI'}
We have
$B_1 \cdot I_{\lambda_{\hf}-1}' =0$.
\end{lemma}

\begin{proof}
Recall from \eqref{eq:I'} that $I_{\lambda_{\hf}-1}' = I_{\lambda_{\hf}-1} -I(2) +I(3)$.
By Lemma~\ref{lem:B1I} that $B_1 \cdot I_{\lambda_{\hf}-1}=0$. By the same argument as for Lemma~\ref{lem:B1I},
we show that $B_1 \cdot I(3)=0$. Finally, we have
\begin{align*}
B_1 \cdot I(2) = \mathord{

				}
				=0.
\end{align*}
The lemma is proved. 
\end{proof}

\begin{lemma}
We have $\kappa \cdot B_2 \cdot I_{\lambda_{\hf}-1}' =0$.
\end{lemma}

\begin{proof}
We compute that
$$
B_2 \cdot I(2) = \, \mathord{
			%
					}.								
$$
Observe that $B_2 \cdot I(k)$ for $k=2,3$ are of the form 
$$
B_2 \cdot I(k) = \sum \iota_t \cdot \alpha_t
$$ 
for some $\alpha_t$, and this implies by Lemma~\ref{lem:*} that $\kappa \cdot B_2 \cdot I(k) =0$. 
Recall $\kappa \cdot B_2 \cdot I_{\lambda_{\hf}-1}=0$ from Lemma~\ref{lem:B2I}.
The lemma follows. 
\end{proof}

\begin{lemma}
We have $P_k \cdot I_{\lambda_{\hf}-1}' =0$, for $0< k < \lambda_{\hf}-1$.
\end{lemma}

\begin{proof}
We have $P_k \cdot I_{\lambda_{\hf}-1}=0$ by Lemma~\ref{lem:PI}. Now we compute that
\begin{align*}
P_I \cdot (-I(2) + I(3)) &= -\, \,
	\mathord{
		%
		 			}
		=0.
\end{align*}
The lemma follows. 
\end{proof}

\begin{lemma}
We have $P_0' \cdot C_2 \cdot \kappa =0$.
\end{lemma}

\begin{proof}
Recall from \eqref{eq:P'} that $P_{0}' =P_0 -P(2) +P(3) -P(4) +P(5)$.
First we note that
$$
P(4) \cdot C_2 =\,
		 \mathord{
				%
					} 
					 =0;
$$
similarly it can be shown that $P(5) \cdot C_2=0$.

On the other hand, we compute that
$$
P(2) \cdot C_2 = \,
		\mathord{
			%
				}.
$$
The particular forms of $P(k) \cdot C_2$, for $k=2,3$, allow us to apply Lemma~\ref{lem:**} to conclude that
$P(k)\cdot C_2 \cdot \kappa =0$. 
Note $P_0 \cdot C_2 \cdot \kappa =0$ by Lemma~\ref{lem:PC2}. The lemma now follows.
\end{proof}

\begin{lemma}
We have $P_0' \cdot C_1=0$.
\end{lemma}

\begin{proof}
First we compute that
$$
P(2) \cdot C_1 = \,
	\mathord{
			%
				}		
				 =0.
$$
On the other hand, recall from Lemma~\ref{lem:B1I} that $P_0 \cdot C_1 =0$ thanks to the vanishing of bubbles of negative degrees.
Similar computations show that $P(i) \cdot C_1=0$ $(i=3,4,5)$ again thanks to the  relations \eqref{bubble1} and \eqref{bubble2}. 
The lemma is proved.
\end{proof}

\begin{lemma}
We have $P_0' \cdot I_k =0$, for $0<k<\lambda_{\hf}-1$.
\end{lemma}

\begin{proof}
Recall $P_0 \cdot I_k=0$ by Lemma~\ref{lem:PI}. We now compute that
$$
(-P(2) +P(3)) \cdot I_k =\, \sum_{s+t =k }\,
		\mathord{
			%
			}
 =0.
$$
by the bubble relation \eqref{bubble3}. 
In addition, thanks to the vanishing of bubbles of negative degrees, we have
$P(i) \cdot I_k =0$ for $i=4,5$. The lemma follows.
\end{proof}

\begin{lemma}
\label{lem:PI'}
We have $P_0' \cdot I_{\lambda_{\hf}-1}' =0$.
\end{lemma}

\begin{proof}
We compute that
\begin{align*}
P_0' \cdot I_{\lambda_{\hf}-1}' 
&= P_0 \cdot I_{\lambda_{\hf}-1} - \, 
	\mathord{
		%
			} 
			=0,
\end{align*}
where the last identity follows from Lemma~\ref{lem:bubbleslides}.
\end{proof}
This completes the proof of Proposition~\ref{p:left-inv'}.

\subsection{Proof of the identity \eqref{eq:prop:jserre2} }

\begin{prop}\label{prop:invR:2}
The identity \eqref{eq:prop:jserre2} holds in $\mathfrak{\dot U}^\jmath$, that is, we have 
\begin{align*}
&%
}.
\end{align*}
\end{prop}
The last identity follows from the relation \eqref{Pi=1}. So it suffices to prove only the first identity.
We first compute $
%
$. We list the relevant computation in the following lemma.
\begin{lemma}\label{lem:CB}
We have 
	\begin{enumerate}
		\item	\[
				C_1 \cdot B_1
					 = -
					\mathord{
			%
			}.
\end{align*}
	\end{enumerate}
\end{lemma}

\begin{proof}
The first four items are straightforward. Here we show the computation for  (e). 
We remind the reader that we shall use the bubble slides lemma (Lemma~\ref{lem:bubbleslides}) extensively here. We have 
\begin{align*}
&C_2 \cdot \kappa_2 \cdot  B_2 - C_2 \cdot  \kappa_3 \cdot B_2 \\
 = & 
		\,
			 {\frac 12} \sum_{s+t+v =-2}\,
				\mathord{

			}.
\end{align*}
This finishes the computation.
\end{proof}

Then we compute the product $%
$. Recall the definition of $P'_0$ and $I'_{\lambda_{\hf}-1}$ in \eqref{eq:I'} and \eqref{eq:P'}. We list the relevant computation in the following lemma. The computations are straightforward, hence the proof shall be omitted. 

\begin{lemma}\label{lem:IP}
We have
\begin{enumerate}
		\item	\[
			\sum^{\lambda_{\hf}-1}_{k=0} I_k \cdot P_k =
							 {\frac 12} \sum_{t + s + v  = -3} \,
				\mathord{
					%
							}.
				\]
\end{enumerate}
\end{lemma}

Now combining Lemma~\ref{lem:CB} and Lemma~\ref{lem:IP}
gives us Proposition~\ref{prop:invR:2}. 
Finally Proposition~\ref{prop:jserre} follows by Proposition~\ref{p:left-inv'} and Proposition~\ref{prop:invR:2}. 
This completes the categorification of the $\jmath$Serre relations.

\vspace{5mm}



\begin{thebibliography}{BLMAB}

\bibitem[B16]{Ba16} H.~Bao, {\em Kazhdan-Lusztig theory of super type $\D$ and quantum symmetric pairs}, Represent. Theory {\bf 21} (2017), 247--276.



\bibitem[BHLW]{BHLW} A.~Beliakova, K.~Habiro, A.~ Lauda and B.~Webster,
\emph{Cyclicity for categorified quantum groups},
J. Algebra  {\bf 452} (2016), 118--132. 

\bibitem[BG80]{BG80}
J.~Bernstein and S.~Gelfand,
{\em Tensor products of finite and infinite representations of semisimple Lie algebras},
Compos. Math. {\bf 41} (1980), 245--285. 


\bibitem[BW89]{BMW1} J.~Birman,  H.~Wenzl, 
\emph{Braids, link polynomials and a new algebra}, 
Trans. Amer. Math. Soc. \textbf{313} (1989),  249--273. 
  
\bibitem[Br16]{B} J. ~Brundan, {\em On the definition of Kac-Moody
    $2$-category}, Math. Ann. {\bf 364}  (2016), 353--372. 


  
 \bibitem[BKLW]{BKLW}  H.~Bao, J.~Kujawa, Y.~Li and W.~Wang,
{\em Geometric Schur duality of classical type},  (with Appendix by Bao, Li and Wang), Transform. Groups (to appear), 
\arxiv{1404.4000}v3.

\bibitem[BLM90]{BLM90} A. Beilinson, G. Lusztig and R. McPherson,
          {\em A geometric setting for the quantum deformation of $GL_n$}, Duke Math. J., {\bf 61} (1990), 655--677.

\bibitem[BW13]{BW13} H.~ Bao and W.~ Wang,
{\em  A new approach to Kazhdan-Lusztig theory  of type $B$ via
  quantum symmetric pairs},   to appear in Ast\'erisque.
\arxiv{1310.0103}

\bibitem[CR08]{CR}
J.~ Chuang and R.~ Rouquier, 
\emph{Derived equivalences for symmetric groups and {$\mathfrak {sl}_2$}-categorification}, 
Ann. of Math.  (2) \textbf{167} (2008),  245--298.

\bibitem[D89]{Deodhar} V.~Deodhar, 
\emph{On some geometric aspects of Bruhat orderings. II. The parabolic analogue of Kazhdan-Lusztig polynomials.} 
J. Algebra {\bf 111} (1987), no. 2, 483--506. 

\bibitem[ES13]{ES13} M.~Ehrig and C.~Stroppel, 
\emph{Nazarov-Wenzl algebras, coideal subalgebras and categorified skew Howe duality}, \arxiv{1310.1972}.

\bibitem[E13]{Elias} B.~Elias, {\em The two-color Soergel calculus},  Compos. Math. {\bf 152} (2016), 327--398, \arxiv{1308.6611}.

\bibitem[EW13]{EW} B.~Elias and G.~Williamson, {\em Soergel calculus}, Represent. Theory {\bf 20} (2016), 295--374. 

\bibitem[FKM02]{FKM} V.~Futorny, S.~K\"onig and V.~Mazorchuk, 
  \emph{Categories of induced modules and standardly stratified algebras}, Alg. Represent. Theory \textbf{5} (3) (2002) 259-276.
    
\bibitem[FL15]{FL15}  
Z. Fan and Y. Li,
{\em Geometric Schur duality of classical type, II},
Trans. Amer. Math. Soc., Series {\bf B2} (2015), 51--92, \arxiv{1408.6740}.

\bibitem[FLLLW]{FLLLW}  
Z. Fan,  Y. Li, C. Lai, L. Luo and W. Wang,
{\em Affine flag varieties and quantum symmetric pairs}, Memoirs of AMS (to appear), 
            \arxiv{1602.04383}.


\bibitem[GL92]{GL92} I. Grojnowski and G. Lusztig, 
          {\em  On bases of irreducible representations of quantum $GL_n$}. In:  {Kazhdan-Lusztig theory and related topics}  
          (Chicago, IL, 1989), 167-174, Contemp. Math., {\bf 139}, Amer. Math. Soc., Providence, RI, 1992.

\bibitem[Jim86]{Jim} M. Jimbo,
{\em A $q$-analogue of $U({\mathfrak g\mathfrak l}(N+1))$, Hecke
algebra, and the Yang-Baxter equation}, Lett. Math. Phys. {\bf 11}
(1986), 247--252.    
    
\bibitem[JMW14]{JMW}
D.~Juteau, C.~Mautner and G.~Williamson, \emph{Parity sheaves}, J. Amer. Math.
  Soc. \textbf{27} (2014),  1169--1212.

\bibitem[Ka93]{Ka93} M. Kashiwara,
{\em Global crystal bases of  quantum groups}, Duke Math.~J.~{\bf 69} (1993), 455--485.

\bibitem[KL09]{KLI}
M. Khovanov and A. Lauda, 
\emph{A diagrammatic approach to  categorification of quantum groups. {I}}, Represent. Theory \textbf{13}
  (2009), 309--347.

\bibitem[KL10]{KLIII}
\bysame, 
\emph{A categorification of quantum {${\rm sl}(n)$}}, Quantum Topol.
  \textbf{1} (2010),  1--92.  
   
\bibitem[Ko14]{Ko14}  
S. Kolb,
{\em Quantum symmetric Kac-Moody pairs}, 
Adv. in Math. {\bf 267} (2014), 395--469. 

\bibitem[La10]{L} 
A.~Lauda, 
{\em A categorification of quantum  $\mathfrak{sl}(2)$.} Adv. Math. \textbf{225} (2010),  3327--3424. 

\bibitem[Le03]{Le03} G. Letzter, {\em Quantum symmetric pairs and their zonal spherical functions}, 
Transformation Groups {\bf 8} (2003), 261--292.

\bibitem[Lib08]{Lib} N.~ Libedinsky,  
{\em Sur la cat\'egorie des bimodules de Soergel.}
J. Algebra \textbf{320} (2008),  2675--2694. 

\bibitem[Lu90]{Lu90} G. ~Lusztig, 
{\em Canonical bases arising from quantized enveloping algebras}, J.~Amer.~Math.~Soc.~{\bf 3} (1990),
447--498.

\bibitem[Lu92]{Lu92} G. Lusztig, 
{\em Canonical bases in tensor products}, Proc. Nat. Acad. Sci. {\bf 89} (1992), 8177--8179. 

\bibitem[Lu94]{Lu94} G.~ Lusztig, {\em Introduction to Quantum Groups},
Modern Birkh\"auser Classics, Reprint of the 1993 Edition,
Birkh\"auser, Boston, 2010.

\bibitem[LW15]{LW15}  Y.~Li and W.~Wang,
{\em Positivity vs negativity of canonical bases}, Proceedings for Lusztig's 70th birthday conference,
Bulletin of  Institute of Math. Academia Sinica (N.S.),  to appear, 
\arxiv{1501.00688}. 
            
\bibitem[MSV13]{MSV}
M. ~Mackaay, M. ~Sto{\v{s}}i{\'c}, and P. ~Vaz, 
\emph{A diagrammatic categorification of the {$q$}-{S}chur algebra}, Quantum Topol. \textbf{4}
  (2013),  1--75.  

\bibitem[M87]{BMW2} 
J.~Murakami,
\emph{The Kauffman polynomial of links and representation theory. }
Osaka J. Math. \textbf{24} (1987),  745--758. 
 
\bibitem[R08]{R}
R.~Rouquier,  {\em 2-Kac-Moody Algebras},  \arxiv{0812.5023}.


  
\bibitem[Soe90]{Soe90}
W. Soergel, \emph{Kategorie {$\mathcal{ O}$}, perverse {G}arben und
  {M}oduln \"uber den {K}oinvarianten zur {W}eylgruppe}, J. Amer. Math. Soc.
  \textbf{3} (1990),  421--445.  

\bibitem[Soe92]{Soe92}
\bysame, \emph{The combinatorics of {H}arish-{C}handra bimodules}, J. Reine
  Angew. Math. \textbf{429} (1992), 49--74.  

\bibitem[Soe07]{Soe07} 
\bysame, \emph{Kazhdan-Lusztig-Polynome und unzerlegbare Bimoduln\"uber Polynomringen},
J. Inst. Math. Jussieu \textbf{6} (2007), 501--525.


\bibitem[Str04]{Str04}
C.~Stroppel, \emph{A Structure Theorem for Harish-Chandra Bimodules via Coinvariants and Golod Rings}, 
J. of Algebra \textbf{282},  (2004) 349--367.


\bibitem[W15]{WebCB}
B.~Webster, \emph{Canonical bases and higher representation theory}, Compos. Math.
  \textbf{151} (2015),  121--166.

\bibitem[W15b]{Webcomparison}
\bysame, \emph{Comparison of canonical bases for {S}chur and universal
  enveloping algebras.}, \arxiv{1503.08734}.
  

\bibitem[Wil11]{WilSSB}
G. Williamson,
\emph{Singular {S}oergel bimodules}.
Int. Math. Res. Not. IMRN, (20) 4555--4632, 2011.

\bibitem[Wil08]{Wilthesis}
\bysame,  \emph{Singular {S}oergel bimodules}, PhD thesis, Albert-Ludwigs-Universit\"at Freiburg, 2008.

\end{thebibliography}
\end{document}